\documentclass[11pt,a4paper]{article}

\usepackage{graphicx, graphics, epsfig, color}
\usepackage{booktabs}
\usepackage{subcaption}
\usepackage{natbib}
\usepackage{amsmath,amssymb,amsthm}
\usepackage{tikz, pgfplots}
\usetikzlibrary{plotmarks,spy}
\pgfplotsset{compat=newest}
\usepackage[margin=1in]{geometry}
\usepackage{hyperref}
\hypersetup{colorlinks=false}
\usepackage[linesnumbered,ruled,vlined]{algorithm2e}
\usepackage{enumitem}
\usepackage[capitalize,noabbrev]{cleveref}

\title{Fundamental Bias in Inverting Random Sampling Matrices with Application to SSN }

\author{
  Chengmei Niu, \quad  Zhenyu Liao\footnote{Author to whom any correspondence should be addressed: Zhenyu Liao (\href{mailto:zhenyu_liao@hust.edu.cn}{zhenyu\_liao@hust.edu.cn}).}, \quad  Zenan Ling,\\ 
  School of Electronic Information and Communications,\\
  Huazhong University of Science and Technology, Wuhan, China\\
  \texttt{chengmeiniu@hust.edu.cn}, \texttt{zhenyu\_liao@hust.edu.cn}, \texttt{lingzenan@hust.edu.cn}
  \and
  Michael W.\@ Mahoney\\
  ICSI, LBNL, and Department of Statistics\\
  University of California, Berkeley, USA\\
  \texttt{mmahoney@stat.berkeley.edu}
}
\date{\today}

\DeclareMathOperator{\tr}{tr}

\DeclareMathOperator{\diag}{diag}

\DeclareMathOperator{\Var}{Var}

\DeclareMathOperator*{\argmin}{arg\,min}

\newcommand{\RR}{{\mathbb{R}}}

\newcommand{\EE}{{\mathbb{E}}}

\newcommand{\A}{\mathbf{A}}
\newcommand{\B}{\mathbf{B}}
\newcommand{\C}{\mathbf{C}}
\newcommand{\D}{\mathbf{D}}
\newcommand{\F}{\mathbf{F}}

\newcommand{\J}{\mathbf{J}}
\newcommand{\M}{\mathbf{M}}
\newcommand{\V}{\mathbf{V}}

\newcommand{\X}{\mathbf{X}}
\newcommand{\Z}{\mathbf{Z}}
\newcommand{\Q}{\mathbf{Q}}

\newcommand{\E}{\mathbf{E}}

\renewcommand{\S}{\mathbf{S}}
\renewcommand{\H}{\mathbf{H}}
\renewcommand{\P}{\mathbf{P}}
\newcommand{\K}{\mathbf{K}}
\newcommand{\TT}{\mathbf{T}}

\newcommand{\uu}{\mathbf{u}}
\newcommand{\vv}{\mathbf{v}}

\newcommand{\ba}{\mathbf{a}}

\newcommand{\ee}{\mathbf{e}}

\newcommand{\g}{\mathbf{g}}
\newcommand{\pp}{\mathbf{p}}
\newcommand{\x}{\mathbf{x}}
\newcommand{\y}{\mathbf{y}}

\newcommand{\s}{\mathbf{s}}

\newcommand{\zo}{\mathbf{0}}

\newcommand{\I}{\mathbf{I}}

\newcommand{\bbeta}{\boldsymbol{\beta}}

\newcommand{\bupsilon}{\boldsymbol{\upsilon}}

\newcommand{\eff}{ \mathrm{eff} }

\definecolor{RED}{rgb}{0.7,0,0}
\definecolor{BLUE}{rgb}{0,0,0.69}
\definecolor{GREEN}{rgb}{0,0.6,0}
\definecolor{PURPLE}{rgb}{0.69,0,0.8}
\definecolor{BLACK}{rgb}{0,0,0}
\definecolor{TEAL}{rgb}{0, 0.5, 0.5}
\newcommand{\RED}{\color[rgb]{0.70,0,0}}

\newcommand{\GREEN}{\color[rgb]{0,0.6,0}}

\newcommand{\qedwhite}{\hfill \ensuremath{\Box}}

\theoremstyle{plain}
\newtheorem{theorem}{Theorem}[section]
\newtheorem{proposition}[theorem]{Proposition}
\newtheorem{lemma}[theorem]{Lemma}
\newtheorem{corollary}[theorem]{Corollary}
\theoremstyle{definition}
\newtheorem{definition}[theorem]{Definition}
\newtheorem{assumption}[theorem]{Assumption}
\newtheorem{condition}[theorem]{Condition}
\theoremstyle{remark}
\newtheorem{remark}[theorem]{Remark}

\begin{document}
\maketitle

\begin{abstract}
A substantial body of work in machine learning (ML) and randomized numerical linear algebra (RandNLA) has exploited various sorts of random sketching methodologies, including random sampling and random projection, with much of the analysis using Johnson--Lindenstrauss and subspace embedding techniques.  
Recent studies have identified the issue of \emph{inversion bias} -- the phenomenon that inverses of random sketches are \emph{not} unbiased, despite the unbiasedness of the sketches themselves. 
This bias presents challenges for the use of random sketches in various ML pipelines, such as fast stochastic optimization, scalable statistical estimators, and distributed optimization.
In the context of random projection, the inversion bias can be easily corrected for dense Gaussian projections (which are, however, too expensive for many applications). 
Recent work has shown how the inversion bias can be corrected for sparse sub-gaussian projections.
In this paper, we show how the inversion bias can be corrected for random sampling methods, both uniform and non-uniform leverage-based, as well as for structured random projections, including those based on the Hadamard transform.  
Using these results, we establish problem-independent local convergence rates for sub-sampled Newton methods.
\end{abstract}

\section{Introduction} 
\label{sec:intro}

Randomized numerical linear algebra (RandNLA) significantly reduces computation, communication, and/or storage overheads by using randomness as an algorithmic resource. 
As a pivotal technique in RandNLA, random sketching -- which encompasses both random projection and random sampling -- is becoming increasingly critical in many modern large-scale machine learning (ML) applications.

More precisely, for a tall matrix $\A \in \RR^{n \times d}$ with $n \gg d$, random projection proposes to obtain a sketch $\tilde \A \in \RR^{m \times d}$, of size $m \ll n$ of $\A$ by randomly and linearly combining the rows of $\A$; while 
random sampling, on the other hand, carefully selects a small  subset (of size $m$ say) of the rows $\A$ and rescales them to obtain $\tilde \A \in \RR^{m \times d}$.
For both approaches, it follows from Johnson--Lindenstrauss (JL) type analysis~\citep{johnson1984extensions} that the random sketch $\tilde \A$ can be used as a ``proxy'' of $\A$ in many downstream ML tasks, e.g., $\tilde \A^\top \tilde \A \approx \A^\top \A$ are close in some sense with high probability, even with $m \ll n$.
This leads to a significant boost in the running time, communication time, and/or memory cost, for many numerical methods~\citep{drineas2006sampling,drineas2011faster,drineas2012fast,avron2017faster,DM21_NoticesAMS,lacotte2022adaptive}.
See also~\citet{mahoney2011randomized,halko2011finding,david2014sketching,RandNLA_PCMIchapter_chapter,martinsson2020randomized,DM21_NoticesAMS,randlapack_book_v2_arxiv,randnla_kdd24_TR} and reference therein for an overview of RandNLA and the applications in modern~ML.

Despite this promising ``complexity-accuracy'' trade-off achieved with random sketching, in many ML pipelines ranging from linear/ridge regression to scalable statistical estimation and fast stochastic optimization, the object of direct interest is the sketched matrix \emph{inverse} $(\tilde{\A}^\top \tilde \A + \C)^{-1}$ for some (perhaps all-zeros or diagonal) positive semi-definite (p.s.d.) $\C$ (instead of $\tilde \A^\top \tilde \A$ itself).
Since the matrix inverse is a nonlinear operator, the fact that $\tilde \A^\top \tilde \A$ is an unbiased or nearly unbiased estimator of $\A^\top \A$ (i.e., $\EE[\tilde \A^\top \tilde \A] = \A^\top \A$, which is a major guide in forming the sketch $\tilde \A$) does \emph{not}, in general, imply the unbiasedness of its inverse, i.e., $\EE[ (\tilde \A^\top \tilde \A)^{-1} ] \not \approx (\A^\top \A)^{-1}$.
This phenomenon of \emph{inversion bias} has been long known in the literature: in the case of Gaussian random projection with $\tilde \A = \S \A$ for $\S \in \RR^{m \times n}$ having i.i.d.\@ $\mathcal N(0,1/m)$ entries, $(\tilde \A^\top \tilde \A)^{-1}$ is known to follow the \emph{inverse Wishart distribution}~\citep{haff1979identity} with $\EE[(\tilde \A^\top \tilde \A)^{-1}] = \frac{m}{m - d -1} (  \A^\top \A)^{-1}$ for $m > d +1$.
However, much less is known beyond the Gaussian setting.
Building upon recent progress in non-asymptotic random matrix theory (RMT), it has recently been shown by \citet{derezinski2021sparse} that a similar inversion bias holds, and that it can be corrected, with $\EE[(\tilde \A^\top \tilde \A)^{-1}] \approx \frac{m}{m - d} (  \A^\top \A)^{-1}$, in the sense of partial order of p.s.d.\@ matrices, for sub-gaussian \emph{and} the so-called LEverage Score Sparsified (LESS) projections. 
The latter can be of greater practical interest than, e.g., dense Gaussian or sub-gaussian projections, as it is significantly sparser and can thus be evaluated more efficiently.

This precise characterization of inversion bias has direct implications for RandNLA and ML. 
As a telling example, \citet{derezinski2021newtonless} applied LESS sketches to Newton Sketch~\citep{pilanci2017Newton} and showed that this Newton-LESS approach enjoys almost the same local convergence rates as Newton Sketch with dense Gaussian projections.
This leads to a significantly better ``complexity-convergence'' trade-off than the vanilla Gaussian projections in stochastic second-order optimization methods.

\subsection{Our Contributions} 

In this paper, we consider the inversion bias of random sampling, including uniform and non-uniform sampling, as well as \emph{structured} random projections such as the Subsampled Randomized (Walsh-)Hadamard Transform (SRHT)~\citep{ailon2006approximate}.
The analysis framework in~\citet{derezinski2021newtonless,derezinski2021sparse} does \emph{not} apply (to get non-vacuous results; see, e.g., \Cref{prop:coarse-RS}~and~\Cref{rem:on_coarse} below).
Instead, we exploit novel and non-trivial connections between non-asymptotic RMT and RandNLA to show that this inversion bias can be precisely characterized and numerically corrected.
We also show how this inversion bias result can be used to improve the local convergence rates of the popular sub-sampled Newton (SSN) method~\citep{yao2018inexact,roosta2019subsampled,xu2020newton}.

Our main contributions can be summarized as follows.

\begin{enumerate}
    \item We provide precise characterization of the inversion bias for general random sampling (in \Cref{theo:inverse-bias}) and the corresponding de-biased approach (in \Cref{prop:debias}). 
    The proposed analysis and debiasing technique hold for exact and approximate leverage-based sampling (\Cref{coro:inver_bias_constant_debias}), as well as the structured SRHT (\Cref{coro:debiasing_srht}) as special cases.
    \item With this precise inversion bias result, we further establish (in \Cref{theom:RS-convergence}), the first \emph{problem-independent} local convergence rates for sub-sampled Newton that approximately matches the dense Gaussian Newton Sketch scheme.
    Numerical results are provided in \Cref{sec:num} to support these findings.
\end{enumerate}

\subsection{Related Work}
\label{subsec:related}

\paragraph{Inversion bias.}  
Given a matrix $\A \in \RR^{n \times d}$, the matrix inverse $(\A^\top\A)^{-1}$ is fundamental in ML, numerical computation, and statistics.  
Examples include linear functions $(\A^\top\A)^{-1}\y$ that are crucial to Newton's methods~\citep{boyd2004convex}, quadratic forms $\ba_i^\top(\A^\top\A)^{-1}\ba_i$ (with $\ba^\top_i$ the $i^{th}$ row of $\A$) in computing matrix leverage scores~\citep{drineas2012fast}, and trace forms, $\tr \mathbf{L} (\A^\top\A)^{-1} $ for some given $\mathbf{L}$, of interest in uncertainty quantification~\citep{kalantzis2013Accelerating} and experimental designs~\citep{pukelsheim2006optimal}.  
In the case of tall matrices with $n\gg d$, random sketching applies to efficiently reduce the computational overhead of $(\A^\top\A)^{-1}$, by using a sketch $\tilde\A=\S\A$ of $\A$ for random matrix $\S\in\RR^{m\times n}$ with $m\ll n$. 
It has been shown recently~\citep{michal2020precise,derezinski2021newtonless,derezinski2021sparse} that these sketched inverses are \emph{biased} for \emph{unstructured} sub-gaussian $\S$, with $\EE[(\tilde \A^\top\tilde \A)^{-1}] \not \approx (\A^\top\A)^{-1}$. 
In this paper, we consider the (often practically more interesting) case of \emph{structured} random matrix $\S$, including random sampling matrices (\Cref{def:RS}) and randomized Hadamard transforms (\Cref{def:srht}).

Different from our approach that explicitly modifies the sketch to correct the bias, another line of work proposes to use shrinkage-based correction techniques~\citep{pmlr-v202-zhang23ah,romanov2024newton}.

\paragraph{Random sampling.}
Random sampling is at the core of RandNLA~\citep{drineas2006fast1,mahoney2011randomized,ma2015statistical,DM16_CACM,DM21_NoticesAMS,randnla_kdd24_TR}, and it plays a central role in fast matrix multiplication~\citep{drineas2006fast1}, approximate regression~\citep{drineas2006sampling}, and low-rank approximation~\citep{cohen2017input}, to name a few. 
It is of particular interest in scenarios where the dataset is massive and cannot be stored and/or computed on a single machine, e.g., the census data~\citep{wang2018optimal} and online network data~\citep{deng2024subsampling}.
See \Cref{def:RS} below for a formal definition of random sampling and discussions thereafter for commonly-used sampling schemes including (exact and approximate) leverage score sampling~\citep{mahoney2011randomized,cohen2017input}, shrinkage leverage sampling~\citep{ma2015statistical}, as well as optimal subsampling~\citep{wang2018optimal,wang2021optimal,yu2022optimal,ma2022asymptotic}. 
In this paper, instead of providing classical JL and subspace embedding-type results on random sampling, we precisely characterize (and correct) the inversion bias for a variety of commonly-used random sampling schemes.

\paragraph{Sub-sampled Newton.}
Sub-sampled Newton (SSN) methods propose to approximate the Hessian in Newton's method using a small subset of samples, and they have been extensively studied within the fields of ML, RandNLA, and optimization~\citep{xu2016subsampled,bollapragada2019exact,roosta2019subsampled,xu2020newton,ye2021approximate}.
Although these fast optimization methods are easy to implement, their convergence rates are challenging to analyze.
Existing results often depend on the Hessian condition number or the Lipschitz constant and fall short of, e.g., the \emph{problem-independent} convergence rates achieved by sub-gaussian Newton Sketch~\citep{lacotte2019faster,derezinski2021newtonless}.
In this paper, we establish the first \emph{problem-independent} local convergence rates for SSN that closely align with Newton Sketch. 
This addresses the convergence guarantee gap identified in Iterative Hessian Sketch~\citep{pilanci2016iterative} for random sampling.

\paragraph{Random matrix theory (RMT).}
RMT studies the (limiting) eigenspectra of large-dimensional random matrices~\citep{anderson2010introduction} and finds its applications in signal processing and communication~\citep{couillet2011random}, statistical finance~\citep{plerou2002Random}, optimization~\citep{paquette2021SGDa,paquette2023Haltinga}, and more recently in large-scale ML~\citep{pennington2017nonlinear,fan2020spectra,mei2021generalization,couillet2022RMT4ML}.
A recent line of work~\cite{liao2020random,liao2021sparse,liao2021hessian} has highlighted non-trivial connections between RMT and RandNLA that this paper further develops.

\subsection{Notations and Organization of the Paper} 

We denote scalars by lowercase letters, vectors by bold lowercase, and matrices by bold uppercase.
For a matrix $\A\in \RR^{n\times d}$, we denote $\A^\top$, $\ba^\top_i \in \RR^{d}$, and $\|\A\|$ the transpose, $i^{th}$ row, and spectral norm of $\A$, respectively.
We denote $\A \preceq \B$ if $\B - \A$ is p.s.d., and use $\I_d$ for the identity matrix of size $d$.
For a vector $\x \in \RR^d$ and a matrix $\B \in \RR^{d \times d}$, we denote $\| \x \|_\B \equiv \sqrt{\x^\top \B \x}$, with the convention $\| \x \| = \| \x \|_{\I_d}$. 
For a random variable $x$, we denote $\EE[x]$ the expectation of $x$ and $\EE_{\zeta}[x]$ the expectation of $x$, conditional on the event $\zeta$.
We use  $\Theta(\cdot)$, $\tilde\Theta(\cdot)$, $O(\cdot)$, and $\tilde O(\cdot)$ notations as in standard computer science literature.

The remainder of this paper is organized as follows.  
\Cref{sec:pre} presents preliminaries on random sampling and a \emph{coarse-grained} characterization of its inversion bias, by directly (and naively) adapting the proof approach from~\citet{derezinski2021sparse} (which turns out to be vacuous in our setting).
\Cref{sec:main} delivers a \emph{fine-grained} analysis of this inversion bias and proposes an efficient non-vacuous debiasing approach.
\Cref{sec:application} demonstrates how these technical results apply to establish \emph{problem-independent} local convergence rates for SSN.
\Cref{sec:num} provides numerical results that support our theoretical findings.
\Cref{sec:conclusion} provides a conclusion, summarizing our findings and discussing future perspectives.
Additional material can be found in the appendices.

\section{Preliminaries on Random Sampling}
\label{sec:pre}

In this section, we introduce a few definitions that will be used in the remainder of this paper.
 
\begin{definition}[Random sampling]\label{def:RS}
For a matrix $\A \in \RR^{n \times d}$ with $n\geq d$, a sketch $\tilde \A \in \RR^{m \times d}$ of $\A$ can be constructed by randomly sampling \emph{with replacement} $m$ from the $n$ rows of $\A$ with an \underline{\emph{importance sampling distribution}}, $\{\pi_i\}_{i=1}^n$, $\sum_{i=1}^n \pi_i = 1$, and then rescaling by $1/\sqrt{m \pi_i}$.
This can be expressed as $\tilde\A=\S\A$, with \emph{random sampling  matrix} $\S\in \RR^{m\times n}$ having only one nonzero entry per row.
\end{definition}

\Cref{def:RS} includes commonly-used random sampling schemes such as uniform sampling (with $\pi_i =1/n$), row-norm-based sampling (with $\pi_i = \| \ba_i \|^2/( \sum_{i=1}^n \| \ba_i \|^2 )$), exact or approximate leverage and ridge leverage score sampling~\citep{mahoney2011randomized,alaoui2015fast} defined below, as well as a mix between them, e.g., the shrinkage leverage sampling~\citep{ma2015statistical}. 
\begin{definition}[Leverage score sampling,~\citet{mahoney2011randomized}]\label{def:lev} 
For a matrix $\A \in \RR^{n \times d}$ of rank $d$ with $n \geq d$ and a p.s.d.\@ matrix $\C \in \RR^{d \times d}$, the $i^{th}$ \underline{\emph{leverage score}} $\ell^{\C}_i$ of $\A$ given $\C$, is defined as $\ell^{\C}_i = \ba_i^\top(\A^\top\A+\C)^{-1}\ba_i, i \in \{ 1, \ldots, n \}$.
The exact leverage score sampling refers to the random sampling approach in \Cref{def:RS} with $\pi_i=\ell^{\C}_i/d_{\eff}$, for $d_{\eff}= \sum_{i=1}^n \ell^{\C}_i$ the \underline{\emph{effective dimension}} of $\A$ given $\C$.  
\end{definition}
The leverage score sampling has been extensively studied in RandNLA and ML. 
By taking $\C = \zo_d$ in \Cref{def:lev}, we obtain the standard leverage scores~\citep{mahoney2011randomized,drineas2012fast}; and by taking $\C = \lambda\I_d$, we obtain the $\lambda$-ridge leverage scores~\citep{alaoui2015fast}.
Given $\A\in\RR^{n\times d}$, its leverage scores can be approximately computed in $O(\mathrm{nnz}(\A)\log n + d^3 (\log d)^2+d^2\log n)$ time, for $\mathrm{nnz}(\A)$ the number of non-zero entries in $\A$, see~\citet{drineas2012fast,clarkson2017low,cohen2017input}.

We also introduce an ``approximation factor'' to measure the extent to which one importance sampling distribution approximates another importance sampling distribution (the exact leverage score distribution, in our case).
\begin{definition}[Importance sampling approximation factor]\label{def:approx_factor}
For a matrix $\A \in \RR^{n \times d}$ with $n \geq d$, a p.s.d.\@ matrix $\C \in \RR^{d \times d}$, and a random sampling matrix $\S \in \RR^{m\times n}$ as in \Cref{def:RS}, with importance sampling distribution $\{\pi_i \}_{i=1}^n$, the (min and max) \underline{\emph{importance sampling approximation}} \underline{\emph{factors}} of the random sampling scheme $\S$ is defined as the pair $(\rho_{\min}, \rho_{\max})$, with $\rho_{\min} \equiv \min_{1\leq i \leq n}\ell_i^{\C}/(\pi_i d_{\eff})$ and $\rho_{\max} \equiv \max_{1\leq i \leq n}\ell_i^{\C}/(\pi_i d_{\eff})$.
\end{definition}
For us, the importance sampling approximation factors $(\rho_{\min}, \rho_{\max})$ in \Cref{def:approx_factor} provide qualitative characterization on how the random sampling scheme under study differs from the \emph{exact} leverage score sampling in \Cref{def:lev}.
This extends the classical notion of  sampling approximation factor in~\citet{drineas2006fast1} to include both the maximum and the minimum.
While prior work primarily focuses on the max factor, here we focus on the inversion bias, where the min factor also plays a natural and significant role; see below in \Cref{sec:main} and also \citet{ma2015statistical}, who first noted the importance of the min factor in statistical style analysis.
By the (generalized) median inequality, we have $\rho_{\min} \leq 1 \leq \rho_{\max}$, with equality for exact leverage score sampling.

It follows from \Cref{def:RS} that $\EE[\tilde{\A}^\top \tilde{\A}] = \A^\top \A$, so that the randomly sampled matrix $\tilde \A^\top\tilde \A$ is an \emph{unbiased estimator} of the true $\A^\top \A$.
This, together with controls on the higher-order moments of $\tilde \A^\top\tilde \A$, allows one to conclude that $\tilde \A^\top\tilde \A$ fluctuates, with high probability and within a small ``distance,'' around the true $\A^\top \A$ of interest. 
This can be made precise using the relative error approximation (for scalars and matrices) defined as follows. 

\begin{definition}[Relative error approximation]\label{def:rela_error_approxi}
For a non-negative scalar $\tilde x \geq 0$, we say $\tilde x$ is an $\epsilon$-approximation of (another scalar) $x$, denoted $\tilde x \approx_{\epsilon} x$, if
\begin{equation}\label{eq:relative_error_scalar}
  ({1+\epsilon})^{-1} x \leq \tilde x \leq (1+\epsilon)x.
\end{equation}
For $\tilde x$ being random, we say $\tilde x$ is an $(\epsilon,\delta)$-approximation of $x$ if \eqref{eq:relative_error_scalar} holds with probability at least $1-\delta$.
Similarly, for a p.s.d.\@ matrix $\tilde \X$, we say $\tilde \X$ is an $\epsilon$-approximation (or an $(\epsilon,\delta)$-approximation when being random) if
\begin{equation}
     \tilde \X \approx_{\epsilon} \X \Leftrightarrow (1+\epsilon)^{-1}\X\preceq \tilde \X \preceq(1+\epsilon)\X.
\end{equation}
\end{definition}

\begin{remark}[Subspace embedding]\normalfont\label{rem:subspace_embed}
For $\tilde \A$ a sketch of $\A$, the property that $\tilde \A^\top \tilde \A \approx_{\epsilon} \A^\top \A$ holds with probability $1-\delta$ is known as $(\epsilon,\delta)$-\emph{subspace embedding} in RandNLA. 
This concept was introduced by~\citet{drineas2006sampling}; see also \citet{mahoney2011randomized} for a history. 
It was subsequently used in data-oblivious form by~\citet{Sarlos06,drineas2011faster}, and then popularized in data-oblivious form (and mis-attributed to~\citet{Sarlos06}) by~\citet{david2014sketching}. 
It plays a central role in the statistical characterization of random sketching techniques.
\end{remark}

The focus of this paper is to go beyond the subspace embedding-type results in \Cref{def:rela_error_approxi}~and~\Cref{rem:subspace_embed}, and to assess the \emph{inversion bias} of the form $\EE[ (\tilde{\A}^\top\tilde{\A} + \C)^{-1}]$ (versus the true inverse $(\A^\top \A + \C)^{-1}$).
To this end, we need the following measure of unbiased estimators.

\begin{definition}[Unbiased estimator]\label{def:unbiased_estimator}
We say a random p.s.d.\@ matrix $\tilde \X$ is an $(\epsilon,\delta)$-unbiased estimator of $\X$ if, conditioned on an event $\zeta$ that happens with probability at least $1-\delta$, 
\begin{equation}
   (1+\epsilon)^{-1}\X \preceq  \EE_{\zeta}[\tilde \X]\preceq   (1+\epsilon)\X,~\text{and}~\tilde \X\preceq O(1) \X. 
\end{equation}
\end{definition}

Note that   the error parameter $\epsilon$ in subspace embeddings quantifies spectral approximation error, whereas the notion of “unbiasedness” specifically refers to inversion bias.
Furthermore,  while subspace embeddings automatically ensure $\epsilon$-unbiasedness up to the \emph{same} level of error $\epsilon$, they are generally \emph{not} guaranteed to remain unbiased for a smaller~$\epsilon$~\citep{derezinski2021sparse}. 

With these definitions and notations at hand, we are ready to assess the statistical properties of random sampling.
A first quantity of interest to the design of random sampling is $m$, the number of trials needed to construct an $(\epsilon,\delta)$-subspace embedding, for some given importance sampling distribution $\{ \pi_i \}_{i=1}^n$.
A slightly more general result is given as follows.%
\footnote{Here we present the subspace embedding result in \Cref{lem:sub_embed} on the (regularized) matrix $\A_\C \equiv \A(\A^\top \A+\C)^{-1/2}$. This is of direct use in analyzing sub-sampled Newton methods in \Cref{sec:application}.}

\begin{lemma}[Subspace embedding for random sampling] \label{lem:sub_embed}
Given $\A\in \RR^{n\times d}$ of rank $d$ with $n\geq d$ and p.s.d.\@ $\C\in\RR^{d\times d}$,
let $\S$ be a random sampling matrix with number of trials $m$ and importance sampling distribution $\{ \pi_i \}_{i=1}^n$ as in \Cref{def:RS}, and let $d_{\eff}=\tr( \A_\C^\top \A_\C )$ be the effective dimension of $\A$ given $\C$ with $\A_\C \equiv \A(\A^\top \A+\C)^{-1/2}$. 
Then, there exists $C > 0$ independent of $n,d_{\eff}$ such that for $m\geq C\rho_{\max} d_{\eff} \log(d_{\eff}/\delta )/\epsilon^2$, failure probability $ \delta \in (0,1/2)$, $\epsilon > 0$, and $\rho_{\max}$ in \Cref{def:approx_factor}, $ \A_\C^\top\S^\top\S\A_\C$ is an $(\epsilon,\delta)$-approximation of $\A_\C^\top\A_\C$. 
\end{lemma}

The proof of \Cref{lem:sub_embed} uses standard matrix concentration techniques and is given in \Cref{sec:proof_sub_embedd}.
Note that \Cref{lem:sub_embed} includes existing results of both leverage ($\C = \zo_d$) and ridge leverage score ($\C = \lambda \I_d$) sampling as special cases, see~\citet[Theorm~3]{chowdhury2018iterative}.  

With \Cref{lem:sub_embed}, we are now ready to evaluate the inversion bias of random sampling.
Since the matrix inverse is nonlinear, for $\A^\top \S^\top\S \A$ with $\EE[\A^\top \S^\top\S \A] = \A^\top \A$, one should, a priori, \emph{not} expect that $(\A^\top \S^\top\S \A)^{-1}$ is an unbiased or nearly unbiased estimator of $(\A^\top \A)^{-1}$.
In the following result, we show (by adapting, in an almost straightforward fashion, the scalar debiasing proof approach of~\citet{derezinski2021newtonless,derezinski2021sparse}) that this inversion bias can be corrected, \emph{but only to some extent}, using the \emph{same} scalar factor as for sub-gaussian or LESS projections.
The proof of \Cref{prop:coarse-RS} is given in \Cref{sec:proof_coarse-RS} for completeness.

\begin{proposition}[Coarse-grained debiasing of random sampling]
\label{prop:coarse-RS}
Given $\A\in \RR^{n\times d}$ of rank $d$ with $n \geq d$ and p.s.d.\@ $\C\in \RR^{d\times d}$, let $\S \in \RR^{m \times n}$ be a random sampling  matrix with importance sampling distribution $\{ \pi_i \}_{i=1}^n$ as in \Cref{def:RS} and max importance sampling approximation factor $\rho_{\max}$ as in \Cref{def:approx_factor}.
Then, there exists $C > 0$ independent of $n,d_{\eff}$ such that if $m\geq C \rho_{\max} d_{\eff}(\log(d_{\eff}/\delta )+\sqrt{d_{\eff}}/\epsilon)$ with $\delta\leq m^{-3}$, $(\frac{m}{m-d_{\eff}}\A^\top  \S^\top\S  \A + \C )^{-1}$ is an $(\epsilon,\delta)$-unbiased estimator of $(\A^\top \A + \C)^{-1}$.
\end{proposition}

\begin{remark}[On \Cref{prop:coarse-RS}]\normalfont
\label{rem:on_coarse}
While the debiasing factor $\frac{m}{m - d_{\eff}}$ is the same as that proposed for random (e.g., sub-gaussian or LESS) projections~\citep{derezinski2021newtonless,derezinski2021sparse}, the resulting inversion bias is significantly larger.
In particular, we have, in the case of \Cref{prop:coarse-RS} and for $m= \Theta (\rho_{\max} d_{\eff}\log d_\eff)$, that $ ( \frac{m}{m-d_{\eff}} \A^\top  \S^\top\S  \A + \C )^{-1}$ has an inversion bias of order $O(\sqrt{d_\eff}/\log d_{\eff} )$. 
This is a \emph{vacuous} bound. It follows from \Cref{lem:sub_embed} that for the same choice of $m$, $ (\A^\top  \S^\top\S  \A + \C )^{-1}$ is, \emph{without} the debiasing factor, an $(O(1),\delta)$-approximation of $(\A^\top \A + \C)^{-1}$, and thus has an inversion bias of order $O(1)$.
See \Cref{lemm:sub_embed_inversion} in \Cref{sec:proof_sub_embedd} for a proof of this fact.
\end{remark}

From \Cref{rem:on_coarse}, it appears that the inversion bias result in \Cref{prop:coarse-RS} is disappointing: 
random sampling, in contrast with sub-gaussian or LESS random projections, while being numerically attractive and easy to implement, does \emph{not} lead to a small inversion bias, at least with the $\frac{m}{m - d_{\eff}}$ debiasing factor \emph{and} the proof approach in~\citet{derezinski2021newtonless,derezinski2021sparse} under \Cref{prop:coarse-RS}.
One may thus wonder:
\begin{quote}
  \emph{Is it possible to get sharper control on the inversion bias of random sampling, either by introducing a different debiasing scheme and/or by using a more refined proof than \Cref{prop:coarse-RS}?}
\end{quote}
Below, we show that such improvement is indeed possible.

\section{Fine-grained Analysis of Inversion Bias for Random Sampling}
\label{sec:main}

We have seen in \Cref{prop:coarse-RS}~and~\Cref{rem:on_coarse} that the scalar debiasing and the proof approach in~\citet{derezinski2021newtonless,derezinski2021sparse} do \emph{not} lead, in the case of random sampling, to a non-vacuous small inversion bias.
In the following result, we provide fine-grained analysis of the inversion bias of random sampling (finer than that in \Cref{prop:coarse-RS}), and we show that the inverse $(\A^\top \S^\top\S \A + \C)^{-1}$ for random sampling $\S$ is biased in a more involved fashion than random projections studied in~\citet{derezinski2021newtonless,derezinski2021sparse}.

\begin{theorem}[Inversion bias for random sampling: fine-grained analysis]
\label{theo:inverse-bias}
Given $\A\in \RR^{n\times d}$ of rank $d$ with $n \geq d$ and p.s.d.\@ $\C\in \RR^{d\times d}$, let $\S \in \RR^{m \times n}$ be a random sampling  matrix with importance sampling distribution $\{ \pi_i \}_{i=1}^n$ as in \Cref{def:RS} and $(\rho_{\min}, \rho_{\max})$ as in \Cref{def:approx_factor}.
Then, for diagonal matrix $\D = \diag\{ D_{ii} \}_{i=1}^n$ the solution to%
\footnote{ It can be checked that $\frac{m}{m+2\rho_{\max}  d_{\eff}}\I_n \preceq \D \preceq \frac{m}{m+\rho_{\min} d_{\eff}}\I_n$ with $\rho_{\min}, \rho_{\max}$ in \Cref{def:approx_factor}. See \Cref{lem:range_D} in \Cref{sec:proof_of_theo:inverse-bias}.\label{footnote:D}}
\begin{equation}\label{eq:def_Dii}
\hspace{-1mm}
    D_{ii}
\hspace{-1mm}
    = \hspace{-1mm}  \frac{m}{ m + \ba_i^\top (\A^\top \D \A + \C)^{-1} \ba_i/\pi_i},
\hspace{-1mm}
\end{equation}
there exists $C > 0$ independent of $n,d_{\eff}$ so that for $m\geq C\rho_{\max} d_{\eff}(\log(d_{\eff}/\delta )+ (\log n)^{2/3} /(\epsilon\log (\log n))^{2/3} )$, $ \delta\leq m^{-3}$, $(\A^\top \S^\top\S \A + \C )^{-1} $ is an $(\epsilon,\delta)$-unbiased estimator of $(\A^\top \D \A + \C)^{-1}$.
\end{theorem}
\begin{proof}[Heuristic derivation of \Cref{theo:inverse-bias}]
For a more transparent understanding of the self-consistent equation in~\eqref{eq:def_Dii} of \Cref{theo:inverse-bias}, we provide here a heuristic derivation. 
The detailed  proof of \Cref{theo:inverse-bias} is deferred to \Cref{sec:proof_of_theo:inverse-bias}.
Denote $\x^\top_s=\ee^\top_{i_s} \A/\sqrt{\pi_{i_s}}$, $\Q=(\A^\top\S^\top\S \A +\C)^{-1}= \left(\frac{1}{m} \sum^{m}_{s=1} \x_s{\x}^{\top}_s + \C \right)^{-1}$
 and $ \Q_{-s}=(\sum_{j\neq s}\frac{1}{m} \x_j{\x}^{\top}_j + \C)^{-1}$, for which we have  $ \sum^{m}_{s=1}\frac{1}{m} \EE[\x_s{\x}^{\top}_s]=\sum^n_{i=1} {\ba}_i{\ba}_i^\top=\A^\top\A$. Then, we follow the \emph{deterministic equivalent} framework (see, e.g.,~\citet[Chapter~2]{couillet2022RMT4ML} for an introduction) and show that $\|\EE[ \Q]-\widetilde\H^{-1}\|\simeq 0$, for $\widetilde{\H}= \A^\top\D\A+\C$, for $\D\in \RR^{n\times n}$ given in \eqref{eq:def_Dii}. 
 First, note that $  \|\EE[ \Q]-\widetilde\H^{-1}\|= \|\EE[\Q]\A^\top  \D\A\widetilde\H^{-1}-\EE[\Q\A^\top \S^\top\S \A ]\widetilde\H^{-1}\| $, it then follows from Sherman-Morrison formula (\Cref{lemm:sherman-morrison}) that
\begin{align*}
\hspace{-1mm}\EE[ \Q\A^\top \S^\top\S \A ]\widetilde\H^{-1}
\hspace{-1mm}
=\hspace{-1mm}\sum^m_{s=1}\EE \left[\frac{\frac{1}{m} \Q_{-s}\x_s\x_s^\top\widetilde\H^{-1}}{1+\x_s^\top \Q_{-s}\x_s/m}\right]\hspace{-1mm}
=\hspace{-1mm}\sum^n_{i=1}\EE\left[\frac{ \Q_{-s}\ba_i\ba_i^\top\widetilde\H^{-1}}{1+\ba_i^\top \Q_{-s}\ba_i/m\pi_{i}}\right] \hspace{-1mm}.
\end{align*}
Using the rank-one perturbation lemma of matrix inverse, see, e.g., \citet[Lemma~2.6]{silverstein1995empirical}, we obtain
\begin{align*}
\hspace{-1mm} \EE[ \Q\A^\top \S^\top\S \A ]\widetilde\H^{-1}\hspace{-1mm} \simeq \hspace{-1mm}\sum^n_{i=1}\EE\left[\frac{ \Q\ba_i\ba_i^\top\widetilde\H^{-1}}{1+\ba_i^\top\widetilde\H^{-1}\ba_i/m\pi_{i}}\right] \hspace{-1mm}
=\hspace{-1mm}\EE\left[\Q\right]\A^\top\D\A\widetilde\H^{-1}\hspace{-1mm},
\end{align*}
for $\D=\diag\{m\pi_i/(m\pi_i+\ba_i^\top\widetilde\H^{-1}\ba_i)\}^n_{i=1}$.
This leads to the self-consistent equation in \eqref{eq:def_Dii} of \Cref{theo:inverse-bias}. 
\end{proof}

\Cref{theo:inverse-bias} says that the (conditional) expectation $\EE_{\zeta}[ (\A^\top \S^\top\S \A + \C )^{-1} ]$, instead of being close to $(\A^\top \A + \C)^{-1}$, is in fact close to $(\A^\top \D \A + \C)^{-1}$, with $\D$ depending on $\A$ and the random sampling scheme per $m$ and $\{ \pi_i \}_{i=1}^n$ in an implicit fashion.
While seemingly uninterpretable and unusable at first sight, \Cref{theo:inverse-bias} can be tuned to design a de-biased random sampling approach.
This is given in the following result. 

\begin{proposition}[Fine-grained debiasing for general random sampling]\label{prop:debias}
Under the settings and notations of \Cref{theo:inverse-bias}, for $\ell_{i_s}^\C$ the $i_s^{th}$ leverage score of $\A$ as in \Cref{def:lev} and standard random sampling matrix $\S$ as in \Cref{def:RS}, define the de-biased sampling matrix $ \check\S \in \RR^{m\times n}$~as
\begin{equation}\label{eq:debias_check_S}
  \hspace{-1mm}\check\S\hspace{-1mm} =\hspace{-1mm}\diag \left\{\sqrt{m/(m-\ell_{i_s}^\C/\pi_{i_s})} \right\}^m_{s=1}\cdot\S.\hspace{-1mm} 
\end{equation}
Then, there exists constant $C > 0$ independent of $n,d_{\eff}$ such that for $m\geq C\rho_{\max} d_{\eff} (\log(d_{\eff}/\delta )+(\log n)^{2/3} /(\epsilon\log (\log n))^{2/3} )$, $\delta\leq m^{-3}$, $(\A^\top \check\S^\top \check\S \A + \C )^{-1}$ is an $(\epsilon,\delta)$-unbiased estimator of $(\A^\top  \A + \C)^{-1}$.
\end{proposition}
\begin{proof}[Heuristic derivation of \Cref{prop:debias}]
To make the intuition behind \Cref{prop:debias} more accessible, we present here a heuristic derivation of \eqref{eq:debias_check_S}.
We refer the reader to \Cref{sec:proof_prop_debias} for the detailed proof of \Cref{prop:debias}. 
Let $\check{\S}^\top \check{\S}=\sum^m_{s=1}F_{i_si_s}\cdot\ee_{i_s}\ee^\top_{i_s}/m\pi_{i_s}$
for some \emph{deterministic} $F_{ii}$ to be specified, $\check{\Q}=(\A^\top \check{\S}^\top \check{\S} \A +\C)^{-1}= \left(\frac{1}{m} \sum^{m}_{s=1} F_{i_si_s} \x_s{\x}^{\top}_s + \C \right)^{-1}$, and similarly $\check{\Q}_{-s}=( \frac1m \sum_{l\neq s} F_{i_li_l}\x_l\x_l^\top+\C)^{-1}$ as in the heuristic derivation of \Cref{theo:inverse-bias} above.
We thus have $ \frac{1}{m} \sum^{m}_{s=1} \EE[F_{i_si_s}\x_s{\x}^{\top}_s]=\sum^n_{i=1}  F_{ii}{\ba}_i{\ba}_i^\top$.
Our goal is to determine $\check{\Q}$ (and $F_{ii}$) such that $\|\EE[\check{\Q}]-\H^{-1}\|\simeq 0$, for $\H^{-1}= (\A^\top  \A + \C)^{-1}$. 
To this end, observe that
 \begin{equation*} 
 \hspace{-1mm}
    \EE[\check{\Q}]-\H^{-1}\hspace{-1mm} = \hspace{-1mm}\EE[\check{\Q}]\A^\top  \A\H^{-1}-\EE[\check{\Q}\A^\top   \check{\S}^\top \check{\S}\A ]\H^{-1} \simeq 0.\hspace{-1mm}
 \end{equation*}
By Sherman-Morrison formula (\Cref{lemm:sherman-morrison}), we  obtain
 \begin{align*}
 \hspace{-1mm}
 \EE[\check{\Q}\A^\top   \check{\S}^\top \check{\S}\A ]\H^{-1}\hspace{-1mm}&=\hspace{-1mm}\EE \left[\frac{\check{\Q}_{-s}F_{i_si_s}\A^\top\ee_{i_s}\ee_{i_s}^\top/\pi_{i_s}\A\H^{-1}}{1+F_{i_si_s}\ee_{i_s}^\top\A\check{\Q}_{-s}\A^\top\ee_{i_s}/m\pi_{i_s}}\right]\hspace{-1mm}
\\
&=\hspace{-1mm}\sum^n_{i=1}\EE\left[\frac{\check{\Q}_{-s}F_{ii}\A^\top\ee_{i}\ee_{i}^\top\A\H^{-1}}{1+F_{ii}\ee_{i}^\top\A\check{\Q}_{-s}\A^\top\ee_{i}/m\pi_{i}}\right]\hspace{-1mm},
 \end{align*}
 where we see the \emph{exact} leverage score $\ee_{i}^\top\A\H^{-1}\A^\top\ee_{i} = \mathbf{a}_i^\top (\A^\top  \A + \C)^{-1} \mathbf{a}_i = \ell_i^\C$ as in \Cref{def:lev} naturally appears in the denominator from the derivation.
Invoking the rank-one perturbation lemma once more,
 we have that 
\begin{align*}
\hspace{-1mm}\EE[\check{\Q}\A^\top   \check{\S}^\top \check{\S}\A ]\H^{-1}\hspace{-1mm}&\simeq \hspace{-1mm} \EE[  \check{\Q}]  \A^\top\sum^n_{i=1}\frac{F_{ii}\ee_{i}\ee_{i}^\top}{1+F_{ii}\ell^{\C}_i/m\pi_{i}}\A\H^{-1}\hspace{-1mm}= \hspace{-1mm}\EE[  \check{\Q}]  \A^\top\A\H^{-1}\hspace{-1mm},  
\end{align*}
where we take the \emph{debiasing factor} $F_{ii}= m\pi_{i}/( m\pi_{i} - \ell_{i}^\C )$ such that $F_{ii}/(1+F_{ii}\ell^{\C}_i/m\pi_{i})=1$.
This leads to the form of the debiasing matrix $\check\S$ as in \eqref{eq:debias_check_S} of \Cref{prop:debias}.
\end{proof}

Comparing the fine-grained results in \Cref{prop:debias} to the coarse-grained results in \Cref{prop:coarse-RS}, we see that the large inversion bias in \Cref{prop:coarse-RS} is indeed a consequence of the proof approach adapted from \citet{derezinski2021newtonless,derezinski2021sparse}, that is \emph{inadequate} for random sampling and for structured random projections such as the SRHT.
\begin{remark}[$\check{\S}$ as a random sampling scheme]\normalfont\label{rem:check S scheme}
Note that $\check\S \in \RR^{m \times n}$ in \Cref{prop:debias} is nothing but another random sampling matrix: it features exactly one nonzero entry per row that is equal to $(m \pi_{i_s} - \ell_{i_s}^\C)^{-1/2}$, as opposed to $(m \pi_{i_s})^{-1/2}$ for the standard random sampling $\S$ in \Cref{def:RS}.
This non-standard re-weighting (that uses the leverage scores of $\A$) ensures that $(\A^\top \check{\S}^\top \check\S \A + \C )^{-1}$ is a nearly unbiased estimate of $(\A^\top \A + \C)^{-1}$, per \Cref{prop:debias}. 
From a computational perspective, the re-weighted random sampling $\check{\S}$ in \eqref{eq:debias_check_S} can be computationally demanding due to the need for \emph{exact} computation of leverage scores $\ell_i^\C$ in \eqref{eq:debias_check_S}.
In \Cref{coro:proof_of_approximate_lev} of \Cref{sec:proof_prop_debias}, we consider approximate leverage scores (which are much faster to compute~\citep{drineas2012fast,clarkson2017low,cohen2017input}).
We show that for a given sampling scheme $\{ \pi_i \}_{i=1}^n$, replacing  exact leverage scores with their approximate counterparts in the de-biased sampling matrix $\check \S$ in \eqref{eq:debias_check_S} increases the inversion bias, but only very slightly.
\end{remark}

Note from \Cref{prop:debias} that the proposed fine-grained de-biasing matrix $\check{\S}$ depends on the importance sampling distribution \emph{only} via $\ell_i^\C/\pi_i$. 
(See \Cref{subsec:RMT_intuition_prop:debias} for the RMT intuition on how the exact leverage scores arise from the derivation.)
As such, for any random sampling method with $\pi_i \approx \ell_i^\C/d_{\eff}$ close to those of exact leverage score sampling in \Cref{def:lev}, we have $\check{\S} \approx \frac{m}{m - d_{\eff}} \S$. 
This \emph{coincides} with the scalar debiasing scheme in the coarse-grained result of~\Cref{prop:coarse-RS}, but it has a much smaller inversion bias.
This special case is discussed in the following result, proven in~\Cref{rem:proof_inver_bias_constant_debias}.

\begin{corollary}[Inversion bias using scalar debiasing under approximate leverage]\label{coro:inver_bias_constant_debias}
Under the settings and notations of \Cref{theo:inverse-bias}, for random sampling scheme with sampling distribution
$\pi_i\in [\ell^{\C}_{i}/(d_{\eff} \rho_{\max}), \ell^{\C}_{i}/(d_{\eff} \rho_{\min})]$ with $\rho_{\min} \in [1/2,1]$ as in \Cref{def:approx_factor},\footnote{Note that by \Cref{def:approx_factor} we have $\rho_{\min}\leq \ell_i^\C/(\pi_i  d_{\eff})\leq \rho_{\max}$ for all $i$, which, 
together with $\rho_{\min}\geq 1/2$ yields that $|\pi_i- \ell_i^{\C}/d_{\eff}|\leq \epsilon_{\rho}\ell_i^{\C}/d_{\eff}\leq\ell_i^{\C}/d_{\eff}$.} there exists $C> 0$, $\nu \geq \log_{d_{\eff}} (\log d_{\eff}), \delta < m^{-3}$ such that for $m \geq C\rho_{\max} d_{\eff}^{1+\nu}$, $(\frac{m}{m-d_{\eff}}\A^\top \S^\top\S \A + \C )^{-1}$ is an $(\epsilon,\delta)$-unbiased estimator of $(\A^\top  \A + \C)^{-1}$ with inversion bias $\epsilon=\max\{\tilde O(d_{\eff}^{-3\nu/2}), O(\epsilon_{\rho}d_{\eff}^{-\nu})\}$ and $\epsilon_{\rho}=\max\{\rho_{\min}^{-1}-1,1 - \rho_{\max}^{-1}\}$.  
\end{corollary}

\begin{remark}[Inversion bias for exact versus approximate leverage score sampling]\normalfont
\label{rem:inv_bias_exact_VS_approx}
It follows from \Cref{coro:inver_bias_constant_debias} that for exact and/or approximate leverage score sampling with $\rho_{\max} \geq \rho_{\min}/(2\rho_{\min}-1)\geq 1$ and $\rho_{\min}>(1+\tilde\Theta( d_{\eff}^{-\nu/2}))^{-1}> 1/2$ (so that $\epsilon_{\rho}= 1 - \rho_{\max}^{-1}$), the inversion bias induced by the scalar debiasing $ \frac{m}{m-d_{\eff}}$ establishes the following \emph{phase transition} behavior:
\begin{enumerate}
    \item if the random sampling scheme is sufficiently close to exact leverage sampling, in that $\rho_{\max} \in [\rho_{\min}/(2\rho_{\min}-1), 1/(1 - \tilde\Theta(d_{\eff}^{-\nu/2}) )]$ (or equivalently the importance sampling probabilities satisfy $\pi_i \in [(1 \pm  \tilde \Theta(d_{\eff}^{-\nu/2})) \ell_i^{\C}/d_{\eff}]$), then the inversion bias under scalar debiasing is the \emph{same} as that (of the fine-grained matrix debiasing) in \Cref{prop:debias}; but
    \item if the random sampling scheme significantly deviates from exact leverage sampling with $\rho_{\max} > 1/(1 - \tilde\Theta(d_{\eff}^{-\nu/2}) )$ (or equivalently $|\pi_i- \ell_i^{\C}/d_{\eff}|>\tilde\Theta(d_{\eff}^{-\nu/2})\ell_i^{\C}/d_{\eff}$), then the inversion bias under scalar debiasing becomes larger than that in \Cref{prop:debias}, increases with $\rho_{\max}$, and saturates at $\rho_{\max} = \Theta(1)$.
\end{enumerate}
This phase transition behavior is visualized in \Cref{fig: specal_lev}. 
See also \Cref{fig:sketch size of inversion bias} in \Cref{sec:imple_detail_nmerical_exper} for the numerical comparison of inversion bias using scalar debiasing between exact and approximate leverage score sampling.
\end{remark}

\begin{figure}[thb]
  \centering
\begin{tikzpicture}
\renewcommand{\axisdefaulttryminticks}{5} 
\pgfplotsset{every major grid/.style={densely dashed}}       
\tikzstyle{every axis y label}+=[yshift=-10pt] 
\tikzstyle{every axis x label}+=[yshift=5pt]
\pgfplotsset{every axis legend/.append style={cells={anchor=east},fill=white, at={(0.7,0.85)}, anchor=north west, font=\tiny }}
\begin{axis}[
    width=0.45\columnwidth,
    height=.35\columnwidth,
    xlabel style={font=\small},
    ylabel style={font=\small},
    tick label style={font=\tiny},
    xmin = 0,
    xmax =9.2,
    ymax =0.34,
    ymin =0,
    ymajorgrids=true,
    scaled ticks=true,
    xlabel = {$\rho_{\max}$},
    ylabel = {Inversion bias $\epsilon$},
    legend style={font=\tiny, at={(0.5,-0.15)}, anchor=north}, 
           xtick={0,2.8,8}, 
  ytick={0.07, 1-1/8-0.93+1/2.8},
       xticklabels={$\frac{\rho_{\min}}{2\rho_{\min}-1}$, $\frac1{1 -  \tilde \Theta(d_{\eff}^{-\nu/2}) }$,$ \Theta(1)$}, 
       yticklabels={ $\tilde O(d_{\eff}^{-3\nu/2})$,$O(d_{\eff}^{-\nu})$},
]
    \addplot[ color=RED, line width=1.5pt, domain=0:2.8]{0.07};   
    \addplot[ color=RED, line width=1.5pt, domain=2.8:8] {1-1/x-0.93+1/2.8};
    \addplot[ color=RED, line width=1.5pt, domain=8:9.2] {1-1/8-0.93+1/2.8};
  \draw[dashed,gray!50,line width=0.5pt](2.8,0)--(2.8, 0.375);
\draw[dashed,gray!50,line width=0.5pt](8,0)--(8, 0.375);
\end{axis}
\end{tikzpicture}
\caption{{The phase transition behavior of inversion bias $\epsilon$ as a function of $\rho_{\max}$ discussed in~\Cref{rem:inv_bias_exact_VS_approx} with scalar debiasing. }} 
\label{fig: specal_lev}
\end{figure}
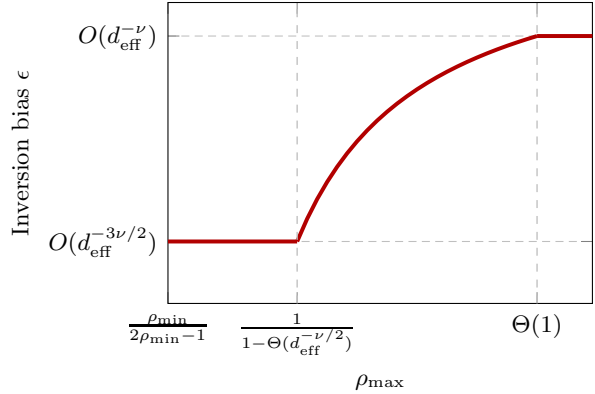

As a side remark, it is known from~\citet[Theorem~10]{derezinski2021sparse} that for \emph{approximate} leverage sampling, and \emph{any} scalar $\gamma > 0, m > 0$,  $(\gamma \A^\top \S^\top \S \A + \C)^{-1}$ with $\C=\zo_d$ is \underline{not} an $(\epsilon, \delta)$-unbiased estimator of $(\A^\top \A + \C)^{-1}$, with any $\epsilon \leq c d_{\eff}/m$ and  $c > 0$ an  absolute constant. 
Thus, 
\begin{enumerate}
  \item in the case of \emph{approximate} leverage score sampling with $\rho_{\max}=3/2$ and $\rho_{\min}=1/2$, for any $m\geq C\rho_{\max} d_{\eff} \log d_{\eff}$, it follows from the proof of \Cref{coro:inver_bias_constant_debias} that the inversion bias is upper bounded by $O(d_{\eff}/m)$, and this coincides with the lower bound in \citet[Theorem~10]{derezinski2021sparse}; 
  and 
  \item in the case of \emph{exact} leverage score sampling with $\rho_{\max}=\rho_{\min} =1$, the inversion bias can be made smaller than $d_{\eff}/m$ under scalar debiasing.%
  \footnote{Notably, using \emph{exact} (instead of approximate) leverage score sampling in the same setting of \citet[Theorem~10]{derezinski2021sparse}, the inversion bias (conditioned on any event $\zeta$ that ensures invertibility) can be made \emph{zero} by taking $\gamma = \frac{m}{d}\EE_{\zeta}[1/b]$ for $b$ distributed as ${\rm Binomial}(m, 1/d)$.
  This aligns with our conclusion of a possibly smaller inversion bias than $d_{\eff}/m$. See~\Cref{coro: exact_lev_zero_bias} in~\Cref{sec:proof_prop_debias} for a proof of this fact.}
\end{enumerate}

As an important consequence, \Cref{prop:debias} also applies to effectively de-bias another commonly-used data-oblivious sketching scheme, the SRHT~\citep{ailon2006approximate}. 

\begin{definition}[Sub-sampled randomized Walsh--Hadamard transform, SRHT,~\citet{ailon2006approximate}]\label{def:srht}
For a given matrix $\A \in \RR^{n \times d}$ of rank $d$ with $n \geq d$, assume without loss of generality that $n = 2^p$ for some integer $p$.
Then, the SRHT of $\A$ is given by
\begin{equation}
  \tilde{\A}_{\rm SRHT} = \S \H_n \D_n \A/\sqrt n \in \RR^{m \times n},
\end{equation}
for uniform random sampling matrix $\S \in \mathbb{R}^{m \times n}$, $\pi_i = 1/n$ as in \Cref{def:RS}, $\H_{n} \in \mathbb{R}^{n \times n}$ the Walsh--Hadamard matrix of size $n$, and diagonal $\D_n \in \RR^{n \times n}$ having i.i.d.\@ Rademacher random variables on its diagonal.
\end{definition}

The SRHT in~\Cref{def:srht} enjoys the following properties: 
the randomized Walsh--Hadamard transform $\H_n \D_n \A$ of $\A$ is known to have approximately uniform leverages scores, that is $\ell^{\C}_i (\H_n \D_n \A/\sqrt{n}) \approx d_{\eff}/n$, see~\citet{drineas2011faster} and~\citet[Theorems~3.1~and~3.2]{tropp2011improved}, as well as \Cref{lemm: HD_balance_row_norms} in \Cref{subsec:discussion_prop:debias} in our setting; and
since $\H_{n}^\top \H_n/n = \I_n $ and $\D^2_n=\I_n$, one has $\frac1n \A^\top \D_n \H_n^\top \H_n \D_n \A =\A^\top\A$, so that $\H_n \D_n \A/\sqrt n$ and $\A$ have the same effective dimension.
These lead to the following fine-grained debiasing result for SRHT with scalar debiasing, proven in~\Cref{rem:proof of SRHT}.

\begin{corollary}[Fine-grained debiasing for SRHT using scalar debiasing]\label{coro:debiasing_srht} 
Under the setting and notations of \Cref{theo:inverse-bias}, for $\tilde{\A}_{\rm SRHT} \in \RR^{m \times n}$ the SRHT of $\A$ as in \Cref{def:srht}, then there exists $C> 0$, $\nu \geq 0$, $n \exp(-d_{\eff}) < \delta < m^{-3}$ such that
for   $m\geq C\rho_{\max} d_{\eff}^{1+\nu}$, $(\frac{m}{m-d_{\eff}}\tilde{\A}_{\rm SRHT}^\top \tilde{\A}_{\rm SRHT} + \C )^{-1}$ is an $(\tilde O(d_{\eff}^{-3\nu/2})+O( \rho^{-1}_{\max} \sqrt{\log(n/\delta)} d_{\eff}^{-\nu-1/2}),\delta)$-unbiased estimator of $(\A^\top  \A + \C)^{-1}$.\footnote{Recall from \Cref{def:approx_factor} that here for SRHT, we have $\pi_i=1/n$ and thus $\rho_{\max} = \max_{1\leq i \leq n}n\ell^{\C}_i (\H_n \D_n \A/\sqrt{n})/d_{\eff} $. }
\end{corollary}

\section{Application to De-biased Sub-sampled Newton with Improved Convergence}  
\label{sec:application}

In this section, we show that the precise characterizations of random sampling inversion bias and the debiasing techniques in \Cref{sec:main} apply to establish \emph{problem-independent} local convergence rates of SSN methods.

Consider the following optimization problem:
\begin{equation}\label{eq:optimization_problem}
    \bbeta^\ast= \argmin_{\bbeta\in \mathcal{C}} F(\bbeta)=\argmin_{\bbeta\in \mathcal{C} } f(\bbeta) + \Phi(\bbeta),
\end{equation}
for some smooth function $F \colon \RR^d \to \RR$ that can be decomposed into $f$ and $\Phi$, and $\mathcal{C}\subseteq \RR^d$ a convex set.
This decomposition naturally arises in ML when, e.g., the loss function $F$ is the sum of the empirical risk $f$ over a set of $n$ training samples and some regularization penalty $\Phi$.

Newton's method solves \eqref{eq:optimization_problem} by performing iterative updates $\bbeta_{t+1}=\bbeta_{t}- \mu_t \H_t^{-1}(\bbeta_{t}) \nabla F(\bbeta_{t})$, with $\mu_t$ the step size, $\nabla F(\bbeta_{t}) \in \RR^d$ the gradient, and $\H_t(\bbeta_{t}) \in \RR^{d\times d}$ the Hessian of $F$ at $\bbeta_{t}$ that can be decomposed as
\begin{equation}
  \H_t(\bbeta_{t})=\A(\bbeta_{t})^\top\A(\bbeta_{t}) +\C(\bbeta_{t}),
\end{equation}
with $\A(\bbeta_{t})\in \RR^{n\times d}$ and some p.s.d.\@ matrix $\C(\bbeta_{t})=\nabla^2\Phi(\bbeta_t)\in \RR^{d\times d}$ that takes a simple form, e.g., $\C(\bbeta_t) = 2 \lambda \I_d$ in the case of $L_2$ regularization $\Phi(\bbeta) = \lambda \| \bbeta \|^2$.

Despite having a locally super-linear convergence rate, Newton's method suffers from a heavy computational burden in forming the Hessian matrix $\H_t(\bbeta_{t})$, particularly when the training samples $n$ is large, e.g., $n \gg d$.
In this case, the major computational bottleneck of Newton's method lies in the construction of $\A(\bbeta_{t})^\top \A(\bbeta_{t})$ for the \emph{inverse} Hessian matrix, as this takes $O(nd^2)$ time.%
\footnote{Of course, this does \emph{not} need to be done explicitly.}
Many randomized second-order methods have been proposed to replace the exact Hessian inverse by some computationally efficient estimate.
Here, we consider SSN methods, that randomly sample the Hessian~\citep{yao2018inexact,roosta2019subsampled,xu2020newton} as follows.
\begin{definition}[Sub-sampled Newton, SSN]\label{def:sub-sampled_Newton}
To solve the optimization problem in \eqref{eq:optimization_problem}, the SSN method performs the following iteration:
\begin{align}\label{eq:newton-sub-sampled}
  \hspace{-1mm}
  \bbeta_{t+1}
  \hspace{-1mm}
  =
  \hspace{-1mm} 
  \bbeta_{t}
  \hspace{-1mm}
  -
  \hspace{-1mm}
  \mu_t \left(\A(\bbeta_t)^\top\S_t^\top  \S_t \A(\bbeta_t)+\C(\bbeta_t) \right)^{-1}\g_t,
\end{align}
for  $t = 0,1,\ldots T$, with $\g_t \equiv \nabla F(\bbeta_{t}) \in \RR^{d}$ the gradient of $F$ at $\bbeta_{t}$, $\mu_t$ the step size at time $t$, and random sampling matrix $\S_t \in \RR^{m \times n}$ as in \Cref{def:RS}.
\end{definition}

By randomly sampling the (computationally intense component of the) Hessian matrix, each SSN step in \eqref{eq:newton-sub-sampled} takes only $O(md^2)$ time. 
The (local or global) convergence rates of SSN have been extensively studied in the literature of optimization, RandNLA, and ML; see, e.g.,~\citet{xu2016subsampled,bollapragada2019exact,roosta2019subsampled,ye2021approximate,lacotte2021adaptive} and \Cref{subsec:related} above.
In particular, \citet{lacotte2021adaptive} proposed an adaptive SSN-type algorithm that achieves a quadratic convergence rate by dynamically adjusting the sketch size, improving upon the well-established linear–quadratic rates of traditional SSN methods.  
Incorporating their adaptive algorithm into our proposed de-biased approach could potentially accelerate convergence further; however, this exploration lies beyond the scope of the present work.

Due to the absence of precise characterizations of the sub-sampled Hessian inverse, as in \Cref{theo:inverse-bias}~and~\Cref{prop:debias} (that allow to, e.g., prove the \emph{near-unbiasedness} of SSN iteration $\EE[\bbeta_{t+1}] \approx \bbeta_t - \mu_t \H_t^{-1}(\bbeta_{t}) \g_t$), it is technically challenging to obtain \emph{problem-independent} convergence rates for SSN.
In the following, we fill this gap by showing how our inversion bias results in \Cref{sec:main} apply to establish \emph{problem-independent} local convergence rates for \emph{de-biased} SSN.
We position ourselves under the following standard assumption on the objective function~$F$.
\begin{assumption}[Lipschitz Hessian]\label{assum:lip_f}
$F, f \colon \RR^d\rightarrow \RR$ in \eqref{eq:optimization_problem} have Lipschitz continuous Hessian with Lipschitz constant $L$, that is, for any $\bbeta, \bbeta'\in\RR^d$, $\max\{ \|\nabla^2F(\bbeta)-\nabla^2F(\bbeta')\|, \|\nabla^2 f(\bbeta)-\nabla^2 f(\bbeta')\|\}\leq L \|\bbeta-\bbeta'\|$.
\end{assumption}

Under \Cref{assum:lip_f}, we evaluate the local convergence rate of the following \emph{de-biased} SSN iterations:
\begin{equation}\label{eq:de_newton-sub-sampled}
  \hspace{-1mm}
  \check{\bbeta}_{t+1}
  \hspace{-1mm}
  =
  \hspace{-1mm} 
  \check{\bbeta}_{t}
  \hspace{-1mm}
  -
  \hspace{-1mm}
  \mu_t \left(\A(\check\bbeta_t)^\top \check\S_t^\top  \check\S_t \A(\check\bbeta_t)+\C(\check\bbeta_t) \right)^{-1}\g_t,
\end{equation}
with de-biased $\check{\S}_t = \diag\Big\{ \sqrt{ m/(m - \ell^\C_{i_s}(\check{\bbeta}_t)/\pi_{i_s} ) } \Big\}_{s=1}^m \cdot \S_t$ as in \Cref{prop:debias}, with $\ell^\C_{i_s}(\check{\bbeta}_t)$ the $i_s^{th}$ leverage score of $\A(\check{\bbeta}_t)$ given $\C(\check{\bbeta}_t)$.
This leads to the following result.

\begin{theorem}[Local convergence of de-biased SSN]\label{theom:RS-convergence}
Let \Cref{assum:lip_f} hold. 
For p.d.\@ $\A(\bbeta^{\ast})^\top\A(\bbeta^{\ast})= \nabla^2f(\bbeta^{\ast})$ and p.s.d.\@ $\C(\bbeta^{\ast})=\nabla^2\Phi(\bbeta^{\ast})$, there exists a neighborhood $U$ of $\bbeta^{\ast}$ such that the \emph{de-biased} SSN iteration in \eqref{eq:de_newton-sub-sampled} starting from $\check{\bbeta}_0 \in U$ satisfies, for $U=\{\bbeta \colon \|\bbeta-\bbeta^{\ast}\|_{\H}<\log n(\rho_{\max} d_{\eff} \sigma_{\min}/m)^{3/2}/(\log (\log n)L)\}$, step size $\mu_t=1-\frac{\rho_{\max} }{m/d_{\eff} + \rho_{\max} }$, $m\geq C\rho_{\max} d^{1+\nu}_{\eff}$, and $\nu \geq \log_{d_{\eff}} (\log (d_{\eff}T/\delta) )$ that
\begin{equation}\label{eq:result_conver}
   \left( \EE_{\zeta}\left[\frac{\|\check{\bbeta}_T-\bbeta^{\ast}\|_{\H}}{\|\check{\bbeta}_0-\bbeta^{\ast}\|_{\H}}\right]\right)^{1/T}\leq \frac{\rho_{\max} d_{\eff}}{m} ( 1+ \epsilon), 
\end{equation}
holds for $\epsilon= \tilde O( d_{\eff}^{-\nu/2})$ and conditioned on an event $\zeta$ that happens with probability at least $1-\delta$.
Here, $\sigma_{\min}$ is the smallest singular value of $\H \equiv\A(\bbeta^{\ast})^\top \A(\bbeta^{\ast})+\C(\bbeta^{\ast})$, $\rho_{\max}$ is the max importance sampling approximation factor in \Cref{def:approx_factor} for $\ell_{i}^\C=\max_{1\leq t \leq T}\ell_{i}^\C(\check{\bbeta}_t)$ and $d_{\eff}=\max_{1\leq t \leq T}d_{\eff}(\check{\bbeta}_t)$ with $\ell_{i}^\C(\check{\bbeta}_t)$ and $d_{\eff}(\check{\bbeta}_t)$ the leverage scores and effective dimension of $\A(\check{\bbeta}_t)$ given $\C(\check{\bbeta}_t)$, respectively.
\end{theorem}

The proof of \Cref{theom:RS-convergence} can be found in \Cref{subsec:proof_convergence}.
The proof relies on a precise characterization of the second inverse moment of the randomly sampled Hessian matrix, extending beyond the first inverse moment result in \Cref{prop:debias}. 
Due to space limitation, this technical result is presented in~\Cref{prop:sec.moment} of \Cref{sec:proof_application}. 
For the sake of practical implementation, we also present, in Corollaries~\ref{coro:ssn_exact_approximate_lev} and~\ref{coro:ssn_SRHT} of~\Cref{sec:proof_application}, local convergence rates of SSN iterations \emph{using scalar debiasing} $m/(m-d_{\eff})$ under exact and approximate leverage score sampling as well as SRHT.


\section{Numerical Experiments of De-biased SSN}
\label{sec:num}

In this section, we provide empirical evidence showing the improved convergence (and a better ``complexity--convergence'' trade-off as a consequence) of different de-biased SSN methods proposed in \Cref{sec:application}, with respect to several first-~and~second-order baseline optimization methods.
We solve the following logistic regression problem 
\begin{equation}
  \min_{\bbeta\in \RR^d } \frac{1}{n}\sum^n_{i=1}\log \left(1+\exp(-y_i\ba^\top_i\bbeta) \right)+\frac{\lambda}{2}\|\bbeta\|^2,
\end{equation}
of the form \eqref{eq:optimization_problem}, for regularization $\lambda>0$, where $\ba_i^\top \in \RR^d$ is the $i^{th}$ row of data matrix $\A \in \RR^{n \times d}$ sampled from both MNIST~\cite{lecun1998gradient} and CIFAR-10~\cite{krizhevsky2009Learning} datasets, and $\y\in\{\pm 1\}^n$ the response vector. 
Implementation details are provided in~\cref{sec:imple_detail_nmerical_exper}. Note that the time reported in Figures~\ref{fig:sketch size}~and~\ref{fig:convergence_time} include both the input data pre-processing time ( e.g., the computation of exact or approximate leverage scores,
and Walsh–Hadamard transform)    and the computational overhead associated with the sketching process.

\Cref{fig:sketch size} assesses the impact of the sketch size $m$ on both relative error ($\|\check{\bbeta}_T-\bbeta^{\ast}\|^2_{\H}/\|\check{\bbeta}_0-\bbeta^{\ast}\|^2_{\H}$) and running time of de-biased SSN employing approximate ridge leverage score sampling (SSN-ARLev), in comparison to the Newton-LESS method~\citep{derezinski2021newtonless} based on random projection.  
The results in \Cref{fig:sketch size} demonstrate that the proposed de-biased SSN consistently outperforms Newton-LESS across all tested sketch size $m$, exhibiting a \emph{superior convergence--complexity trade-off}.
Notably, while the running time for Newton-LESS increases significantly with $m$, that of SSN-ARLev remains approximately constant.

\begin{figure}[thb]
  \centering
\begin{subfigure}[c]{0.48\textwidth}
\begin{tikzpicture}
\renewcommand{\axisdefaulttryminticks}{4} 
\pgfplotsset{every major grid/.style={densely dashed}}       
\tikzstyle{every axis y label}+=[yshift=-10pt] 
\tikzstyle{every axis x label}+=[yshift=5pt]
\pgfplotsset{every axis legend/.append style={cells={anchor=east},fill=none,draw=none, at={(0.475,1.11)}, anchor=north east, font=\tiny}}
\begin{axis}[
 axis y line*=left,  
width=0.9\columnwidth,
height=0.75\columnwidth,
xlabel style={font=\tiny},
    ylabel style={font=\tiny},
    tick label style={font=\tiny},
    xmin =200,
        xmax = 1300,
        ymax =2.52E-07,
        ymin =  7.27E-10, 
         ytick={1E-9,1E-8,1E-7},
         yticklabels = {$10^{-9}$,$10^{-8}$,$10^{-7}$},
        ymajorgrids=true,
        scaled ticks=true,
          xlabel = { Sketch size $m$ },
          ylabel = { Relative error },
        ymode=log
        ]

 \addplot[mark=diamond*,color=GREEN,line width=0.8pt] coordinates{
           (200,3.6478e-07)(  400,4.0231e-08)   (700,4.8583e-09 )(1000,1.9774e-09)(1300,1.1004e-09)
        };
  
           \addlegendentry{{Newton-LESS}}[ font=\tiny];
        \addplot[mark=*,color=RED,line width=0.8pt] coordinates{
       (200,1.1732e-07)(  400,1.3025e-08)(700,  2.6956e-09)(1000,1.4680e-09)(   1300, 8.7431e-10)
      };

\end{axis}

\begin{axis}[
axis y line*=right, 
axis x line=none,
width=0.9\columnwidth,
height=0.75\columnwidth,
    ylabel style={font=\tiny},
    tick label style={font=\tiny},
      xmin =200,
        xmax = 1300,
        ymax =0.85,
        ymin =0.063,
          ytick={0.07,0.2,0.7},
         yticklabels ={0.07,0.2,0.7},
        ymajorgrids=true,
         scaled ticks=true,
          ylabel = { Wall-clock time (s) },
        ymode=log
        ]

        \addplot[densely dashed,mark=diamond*,color=GREEN,line width=1pt, mark options={solid,fill=GREEN}] coordinates{
            (200,0.7022)   (400, 0.7096)(  700, 0.7134 )(1000,  0.7250) (1300,0.7327)
        };
      \addplot[densely dashed,mark=*,color=RED,line width=1pt, mark options={solid,fill=RED}] coordinates{
     (200,0.0704)  (400, 0.0706)(700,0.0736)( 1000,0.0728)(  1300,0.0741)
      };
      
\end{axis}
\end{tikzpicture}
\captionsetup{font=scriptsize}
 \caption{MNIST data}
\end{subfigure}
  \hfill{}
  \begin{subfigure}[c]{0.48\textwidth}
  \begin{tikzpicture}
\renewcommand{\axisdefaulttryminticks}{4} 
\pgfplotsset{every major grid/.style={densely dashed}}       
\tikzstyle{every axis y label}+=[yshift=-10pt] 
\tikzstyle{every axis x label}+=[yshift=5pt]
\pgfplotsset{every axis legend/.append style={cells={anchor=east},fill=none, draw=none,at={(0,1.11)}, anchor=north west, font=\tiny}}
\begin{axis}[
 axis y line*=left,  
width=0.9\columnwidth,
height=0.75\columnwidth,
xlabel style={font=\tiny},
    ylabel style={font=\tiny},
    tick label style={font=\tiny},
 xmin =400,
        xmax = 1600,
        ymax =1.4188E-07,
        ymin =  1.93E-11,
         ytick={1E-10,1E-9,1E-7},
         yticklabels = {$10^{-10}$,$10^{-9}$,$10^{-7}$},
        ymajorgrids=true,
        scaled ticks=true,
          xlabel = { Sketch size $m$ },
          ylabel = { Relative error },
        ymode=log,
        ]

     \addplot[mark=*,color=RED,line width=0.8pt] coordinates{
      (400,5.6276e-08)  (700,1.5780e-09)  (1000, 2.5373e-10 )(1300,3.7715e-11)(1600,1.9407e-11)  
      };
 
  \addlegendentry{{SSN-ARLev}}[ font=\tiny];
        \addplot[mark=diamond*,color=GREEN,line width=0.8pt] coordinates{
        (400,1.4187e-07)(700,3.1523e-09)(  1000, 4.7697e-10)  (1300,6.7144e-11) (1600, 3.2042e-11)
        };
\end{axis}
\begin{axis}[
axis y line*=right, 
axis x line=none,
width=0.9\columnwidth,
height=0.75\columnwidth,
    ylabel style={font=\tiny},
    tick label style={font=\tiny},
         xmin =400,
        xmax = 1600,
        ymax =3.2,
        ymin = 0.35,
         ytick={0.35,1.5,3},
         yticklabels = {0.35,1.5,3},
        ymajorgrids=true,
        scaled ticks=true,
          ylabel = { Wall-clock time (s) },
        ]

        \addplot[densely dashed, mark=diamond*, color=GREEN,line width=1pt] coordinates{
            (400,2.7128)  (700,2.9832)(1000,3.0230)(1300,3.0761 )(1600, 3.1263)
        };
      \addplot[densely dashed,mark=*, color=RED,line width=1pt] coordinates{
        (400,0.4960)( 700, 0.4939)(1000, 0.4906)(  1300,0.5024) (1600,0.5218)
      };
\end{axis}
\end{tikzpicture}
\captionsetup{font=scriptsize}
 \caption{CIFAR-10 data}
  \end{subfigure}
\caption{{ Relative errors (in solid lines) and wall-clock time (in dashed lines) as a function of the sketch size $m$, for \textbf{\GREEN Newton-LESS} and the proposed de-biased \textbf{\RED SSN-ARLev} methods, applied to logistic regression on both MNIST and CIFAR-10 data. 
Relative errors are obtained after a fixed number of iterations ($T=5$ for MNIST data and $T=7$ for CIFAR-10 data).
Results are obtained by averaging over $30$ independent runs. 
}}
\label{fig:sketch size}
\end{figure}
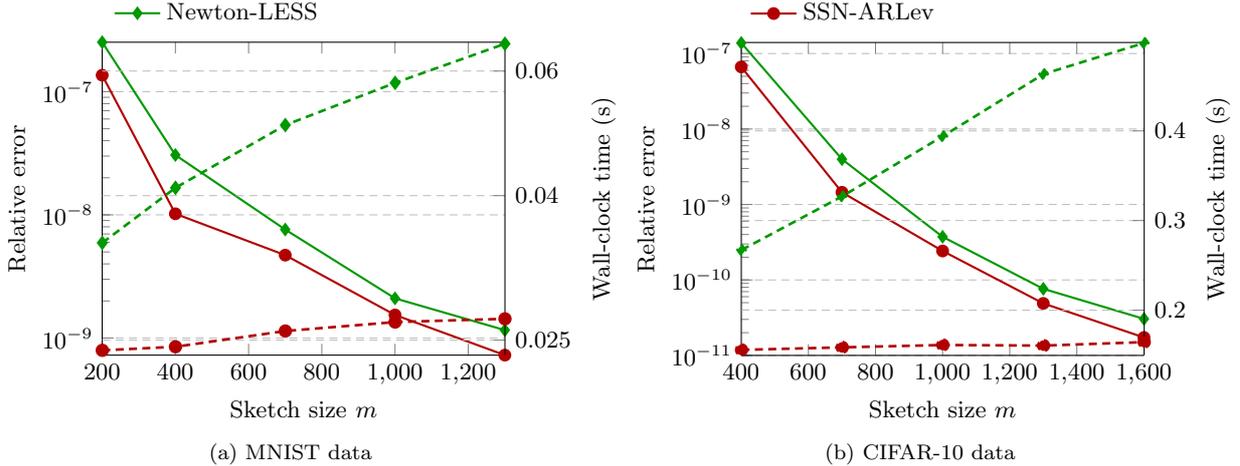

\begin{figure}[thb]
  \centering
   \begin{subfigure}[c]{0.48\textwidth}
\begin{tikzpicture}
\renewcommand{\axisdefaulttryminticks}{5} 
\pgfplotsset{every major grid/.style={densely dashed}}       
\tikzstyle{every axis y label}+=[yshift=-10pt] 
\tikzstyle{every axis x label}+=[yshift=5pt]
\pgfplotsset{every axis legend/.append style={cells={anchor=east},fill=none,draw=none, at={(0,1.19)}, anchor=north west, font=\tiny,legend columns=2,
    transpose legend }}
\begin{axis}[
width=0.9\columnwidth,
height=0.75\columnwidth,
xlabel style={font=\tiny},
    ylabel style={font=\tiny},
    tick label style={font=\tiny},
xmin = 0,
                xmax = 1,
                ymax =2,
                ymin =  0.000000001,
                ymajorgrids=true,
                scaled ticks=true,
                xlabel = { Wall-clock time },
                ylabel = { Relative error },
                ymode=log
                ]
                
        \addplot[mark=triangle*,color=RED!60!white,line width=0.8pt] coordinates{
     (0,1)  ( 0.267311811,0.993604343)   (1.266705751,0.968672711)      (2.016103268,0.950636461)   (2.307114601,0.944746647 )
    };
    \addlegendentry{{GD}}[ font=\tiny];
      \addplot[mark=star,color=BLUE!60!white,line width=0.8pt] coordinates{
        (0,1)   (0.806855273,0.78755906)    (1.626041698,0.642284321)   (2.438386321,0.540377984)   (3.264471269,0.466236189 )
      };
     \addlegendentry{{SGD}}[ font=\tiny];
              
                 \addplot[mark=diamond*,color=GREEN!80!white,line width=0.8pt] coordinates{
                    (0,1)   (0.2013, 3.2483e-02)   (0.4004,4.5267e-04)   (0.6047,5.5929e-06)  ( 0.8073,4.7584e-08)  ( 1.0084, 5.3148e-10)          
                };
       \addlegendentry{{LESS}}[ font=\tiny];
              \addplot[mark=square*,color=RED!25!BLUE,line width=0.8pt] coordinates{
                (0,1)   (0.0383, 4.3196e-02)   ( 0.0772,7.8704e-04)   (0.1163,1.0414e-05)  ( 0.1548, 1.3459e-07)  ( 0.1929, 1.6113e-09) 
            };
            
    \addlegendentry{{ SRHT}}[ font=\tiny];
              \addplot[mark=*,color=RED,line width=0.8pt] coordinates{
                (0,1)(0.0183,3.4379e-02)( 0.0358,4.4612e-04)    ( 0.0522,5.0805e-06)  (0.0706,5.7521e-08)  (0.0865,5.1619e-10)
            };
            
       \addplot[mark=+,color=GREEN!70!RED,line width=0.8pt] coordinates{
            (0,1)   (0.0225, 8.4835e-02)   (0.0416,7.1881e-03)   ( 0.0594, 3.3483e-04)   (0.0771, 2.1644e-05)  (0.0978, 1.7750e-06)      ( 0.1199,  6.2849e-08)  (0.1461, 2.2001e-09)  (0.1713,6.6473e-11)
        };
          \addplot[mark=x,color=GREEN!40!BLUE,line width=0.8pt] coordinates{
        (0,1)   ( 0.0242, 6.3618e-02)   ( 0.0472,1.1230e-03)( 0.0684,2.7534e-05)  (0.0874, 6.8463e-07)      (0.1084, 1.7612e-08)   (0.1297,4.1981e-10)(0.1559,1.1647e-11)
    };
\end{axis}
\end{tikzpicture}
\captionsetup{font=scriptsize}
 \caption{MNIST data}
\end{subfigure}
  \hfill{}
  \begin{subfigure}[c]{0.48\textwidth}
  \begin{tikzpicture}
\renewcommand{\axisdefaulttryminticks}{5} 
\pgfplotsset{every major grid/.style={densely dashed}}       
\tikzstyle{every axis y label}+=[yshift=-10pt] 
\tikzstyle{every axis x label}+=[yshift=5pt]
\pgfplotsset{every axis legend/.append style={cells={anchor=east},fill=none,draw=none, at={(0,1.19)}, anchor=north west, font=\tiny},legend columns=2,
    transpose legend}
\begin{axis}[
width=0.9\columnwidth,
height=0.75\columnwidth,
xlabel style={font=\tiny},
    ylabel style={font=\tiny},
    tick label style={font=\tiny},
xmin = 0,
                xmax = 3.2,
                ymax =2,
                ymin =  0.000000001,
                ymajorgrids=true,
                scaled ticks=true,
                xlabel = { Wall-clock time },
                ylabel = { Relative error },
                ymode=log
                ]
                       \addplot[mark=*,color=RED,line width=0.8pt] coordinates{
        (0,1)( 0.0784,7.6529e-02 )(0.1539,5.0536e-03)  ( 0.2302,3.0345e-04)( 0.3073,1.7792e-05 )       ( 0.3851,1.1075e-06)(  0.4614,6.2403e-08)(  0.5392,3.6105e-09) (0.6191,2.2146e-10)( 0.7009,1.2346e-11)
        };
    \addlegendentry{{ARLev}}[ font=\tiny];    \addplot[mark=+,color=GREEN!70!RED,line width=0.8pt] coordinates{
         (0,1)(0.09,1.1225e-01)(  0.1927,9.7952e-03)   (0.2824,1.1053e-03)( 0.3671,9.9400e-05)( 0.4637,7.7029e-06 ) (0.5510,6.2448e-07)(0.6384, 5.6262e-08)(0.7358, 5.8733e-09)(0.8242,5.9466e-10)
    };
     \addlegendentry{{ALev}}[ font=\tiny];
        \addplot[mark=x,color=GREEN!40!BLUE,line width=0.8pt] coordinates{
     (0.0000, 1)(0.0823, 9.2859e-02 )(0.1763, 7.5875e-03)( 0.2615, 6.1239e-04 )(0.3386, 4.5669e-05 )(0.4302, 2.9578e-06 )(0.5133, 2.3577e-07)( 0.5930, 1.7039e-08)( 0.6906, 1.3178e-09 )(0.7783, 1.1577e-10 )
    };
    \addlegendentry{{SLev}}[ font=\tiny];
    
    \addplot[mark=square*,color=RED!25!BLUE,line width=0.8pt] coordinates{
                (0.0000,1 )( 0.1067, 1.3209e-01 )( 0.2168, 9.6579e-03 )(0.3292, 5.9068e-04 )(0.4384, 3.3053e-05)(0.5516, 2.4013e-06)( 0.6630, 1.4491e-07 )(0.7711, 8.0067e-09 )(0.8806, 4.4673e-10 )
            };

        \addplot[mark=triangle*,color=RED!60!white,line width=0.8pt] coordinates{
            (0,1)(5.71004700660705,0.986412129358221)(11.35992169,0.973167226)(17.02349186,0.960256791 )(22.78676724,0.947672525)
        };
    
            \addplot[mark=star,color=BLUE!60!white,line width=0.8pt] coordinates{
                (0,1)(6.484952068,0.351265181)(12.9296442,0.335570946)( 19.39308889,0.330820916)(25.85534246,0.326025626)
            };
     
                \addplot[mark=diamond*,color=GREEN!80!white,line width=0.8pt] coordinates{
                   (0,1)( 0.4516, 1.2989e-01 )( 0.8978, 1.3104e-02 )( 1.3514, 9.5145e-04)( 1.8036, 6.3647e-05)(  2.2597, 4.1325e-06 )( 2.7156, 2.3141e-07 )( 3.1685, 1.0498e-08)( 3.6211, 5.4098e-10 )
                };
\end{axis}
\end{tikzpicture}
\captionsetup{font=scriptsize}
 \caption{CIFAR-10 data}
  \end{subfigure}

\caption{{
Convergence--complexity trade-off between various optimization methods on logistic regression for MNIST and CIFAR-10 data, with sketch size  $m=300$ for MNIST and $m=400$ for CIFAR-10 data. 
Results are obtained by averaging over $10$ independent runs (except for GD that is deterministic). 
}}

\label{fig:convergence_time}
\end{figure}
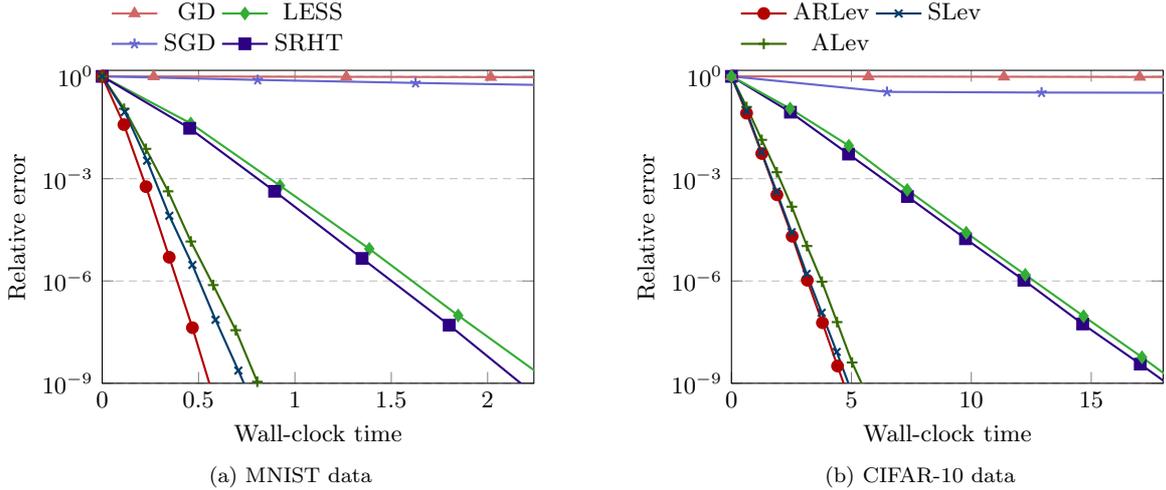

\Cref{fig:convergence_time} compares the relative errors as a function of wall-clock time across various optimization methods:
\begin{enumerate}
  \item \textbf{First-order baselines}: Gradient Descent (GD) and Stochastic GD (SGD).
  \item \textbf{De-biased SSN methods} using different sampling schemes: Approximate $\lambda$-Ridge Leverage Score (ARLev), Approximate Leverage Score (ALev), Shrinkage Leverage Score (SLev) sampling~\citep{ma2015statistical}, and SRHT (see \Cref{def:srht} above).
  \item \textbf{Newton Sketch with LESS-uniform sketch} (LESS)~\citep{derezinski2021newtonless}, which achieves significantly shorter running times compared to the original Newton Sketch with dense Gaussian sketches.
\end{enumerate}

From \Cref{fig:convergence_time}, we see that SSN methods, when properly de-biased, exhibit a significantly better \emph{convergence--complexity} trade-off than both first-order methods \emph{and} the Newton-LESS approach.
The SRHT sampling outperforms Newton-LESS but still lags slightly behind all other SSN variants in speed. Among these, the SLev scheme edges out ALev yet remains slower than ARLev.
Across both datasets tested, de-biased SSN with ARLev consistently delivers the \emph{optimal} convergence–complexity trade-off among all methods evaluated.

\section{Conclusions and Perspectives}
\label{sec:conclusion} 

In the work, we investigate the inversion bias inherent in various random sampling schemes, including uniform and non-uniform leverage-based sampling, as wel as structured random projections (e.g., the Hadamard transform-based SRHT). 
Leveraging recent advances in RMT and RandNLA, we provide a precise characterization of this inversion bias and propose corresponding de-biasing techniques. 
Notably, for approximate leverage sampling and SRHT, this de-biasing reduces to a simple scalar rescaling.
We further show that our results enable an improved SSN method, achieving local convergence rates comparable to those of Newton Sketch with dense Gaussian projections. 
Our theoretical insights are complemented by numerical results on MNIST and CIFAR-10 datasets, underscoring the practical relevance of the proposed approach.

It would be of future interest to see whether the proposed debiasing technique, when wisely combined with adaptive sampling schemes, can achieve or even improve the quadratic convergence in \citet{lacotte2021adaptive}. In addition, it is also worthwhile to  extend our debiasing framework to dependent sampling methods~\citep{cortinovis2024adaptive}, such as volume sampling, which may further improve SSN’s efficiency and convergence.

\section*{Acknowledgements}

Z.~Liao would like to acknowledge the National Natural Science Foundation of China (via fund NSFC-62206101) and the Guangdong Provincial Key Laboratory of Mathematical Foundations for Artificial Intelligence (2023B1212010001) for providing partial support.
Z.~Ling is supported by the National Natural Science Foundation of China (via NSFC-62406119) and the Natural Science Foundation of Hubei Province (2024AFB074).
M.~W.~Mahoney would like to acknowledge the DARPA, DOE, NSF, and ONR for providing partial support of this work.

\section*{Impact Statement}

This paper presents work whose goal is to advance the field of ML, RandNLA, and their connection to RMT.
The work is primarily theoretical, and we do not see any immediate negative societal impact.
If anything, providing methods that come with a smaller theory-practice gap will make it easier to identify and mitigate unintended negative impact.

\bibliography{liao}

\begin{thebibliography}{69}
\providecommand{\natexlab}[1]{#1}
\providecommand{\url}[1]{\texttt{#1}}
\expandafter\ifx\csname urlstyle\endcsname\relax
  \providecommand{\doi}[1]{doi: #1}\else
  \providecommand{\doi}{doi: \begingroup \urlstyle{rm}\Url}\fi

\bibitem[Ailon and Chazelle(2006)]{ailon2006approximate}
Nir Ailon and Bernard Chazelle.
\newblock Approximate nearest neighbors and the fast {Johnson-Lindenstrauss}
  transform.
\newblock In \emph{Proceedings of the Thirty-Eighth Annual ACM Symposium on
  Theory of Computing}, pages 557--563, 2006.

\bibitem[Anderson et~al.(2010)Anderson, Guionnet, and
  Zeitouni]{anderson2010introduction}
Greg~W. Anderson, Alice Guionnet, and Ofer Zeitouni.
\newblock \emph{{An Introduction to Random Matrices}}, volume 118 of
  \emph{Cambridge Studies in Advanced Mathematics}.
\newblock Cambridge University Press, 2010.

\bibitem[Avron et~al.(2017)Avron, Clarkson, and Woodruff]{avron2017faster}
Haim Avron, Kenneth~L Clarkson, and David~P Woodruff.
\newblock Faster kernel ridge regression using sketching and preconditioning.
\newblock \emph{SIAM Journal on Matrix Analysis and Applications}, 38\penalty0
  (4):\penalty0 1116--1138, 2017.

\bibitem[Bollapragada et~al.(2019)Bollapragada, Byrd, and
  Nocedal]{bollapragada2019exact}
Raghu Bollapragada, Richard~H Byrd, and Jorge Nocedal.
\newblock Exact and inexact subsampled {Newton} methods for optimization.
\newblock \emph{IMA Journal of Numerical Analysis}, 39\penalty0 (2):\penalty0
  545--578, 2019.

\bibitem[Bonnans et~al.(2006)Bonnans, Gilbert, Lemar{\'e}chal, and
  Sagastiz{\'a}bal]{bonnans2006numerical}
Joseph-Fr{\'e}d{\'e}ric Bonnans, Jean~Charles Gilbert, Claude Lemar{\'e}chal,
  and Claudia~A Sagastiz{\'a}bal.
\newblock \emph{{Numerical Optimization: Theoretical and Practical Aspects}}.
\newblock Springer Science \& Business Media, 2006.

\bibitem[Boyd and Vandenberghe(2004)]{boyd2004convex}
Stephen~P Boyd and Lieven Vandenberghe.
\newblock \emph{Convex {Optimization}}.
\newblock Cambridge University Press, 2004.

\bibitem[Chowdhury et~al.(2018)Chowdhury, Yang, and
  Drineas]{chowdhury2018iterative}
Agniva Chowdhury, Jiasen Yang, and Petros Drineas.
\newblock An iterative, sketching-based framework for ridge regression.
\newblock In \emph{Proceedings of the 35th International Conference on Machine
  Learning}, volume~{80} of \emph{Proceedings of Machine Learning Research},
  pages { 989--998}. PMLR, 2018.

\bibitem[Clarkson and Woodruff(2017)]{clarkson2017low}
Kenneth~L Clarkson and David~P Woodruff.
\newblock Low-rank approximation and regression in input sparsity time.
\newblock \emph{Journal of the ACM (JACM)}, 63\penalty0 (6):\penalty0 1--45,
  2017.

\bibitem[Cohen et~al.(2017)Cohen, Musco, and Musco]{cohen2017input}
Michael~B Cohen, Cameron Musco, and Christopher Musco.
\newblock Input sparsity time low-rank approximation via ridge leverage score
  sampling.
\newblock In \emph{Proceedings of the Twenty-Eighth Annual ACM-SIAM Symposium
  on Discrete Algorithms}, pages 1758--1777. SIAM, 2017.

\bibitem[Cortinovis and Kressner(2024)]{cortinovis2024adaptive}
Alice Cortinovis and Daniel Kressner.
\newblock Adaptive randomized pivoting for column subset selection, {DEIM}, and
  low-rank approximation.
\newblock \emph{arXiv preprint arXiv:2412.13992}, 2024.

\bibitem[Couillet and Debbah(2011)]{couillet2011random}
Romain Couillet and Mérouane Debbah.
\newblock \emph{{Random Matrix Methods for Wireless Communications}}.
\newblock Cambridge University Press, 2011.

\bibitem[Couillet and Liao(2022)]{couillet2022RMT4ML}
Romain Couillet and Zhenyu Liao.
\newblock \emph{{Random Matrix Methods for Machine Learning}}.
\newblock Cambridge University Press, 2022.

\bibitem[de~la Pe{\~n}a and Gin{\'e}(1999)]{de1999decoupling}
Victor de~la Pe{\~n}a and Evarist Gin{\'e}.
\newblock \emph{Decoupling: From Dependence to Independence}.
\newblock Springer Science \& Business Media, 1999.

\bibitem[Deng et~al.(2024)Deng, Huang, Ding, Zhu, Jing, and
  Zhang]{deng2024subsampling}
Jiayi Deng, Danyang Huang, Yi~Ding, Yingqiu Zhu, Bingyi Jing, and Bo~Zhang.
\newblock Subsampling spectral clustering for stochastic block models in
  large-scale networks.
\newblock \emph{{Computational Statistics \textnormal{\&} Data Analysis}},
  189:\penalty0 107835, 2024.

\bibitem[Derezi{\'n}ski and Mahoney(2021)]{DM21_NoticesAMS}
Micha{\l} Derezi{\'n}ski and Michael~W Mahoney.
\newblock Determinantal point processes in randomized numerical linear algebra.
\newblock \emph{Notices of the AMS}, 68\penalty0 (1):\penalty0 34--45, 2021.

\bibitem[Derezi{\'n}ski and Mahoney(2024)]{randnla_kdd24_TR}
Micha{\l} Derezi{\'n}ski and Michael~W. Mahoney.
\newblock Recent and upcoming developments in randomized numerical linear
  algebra for machine learning.
\newblock In \emph{Proceedings of the 30th ACM SIGKDD Conference on Knowledge
  Discovery and Data Mining}, KDD '24, page 6470–6479, New York, NY, USA,
  2024. Association for Computing Machinery.

\bibitem[Derezi{\'n}ski et~al.(2020)Derezi{\'n}ski, Liang, Liao, and
  Mahoney]{michal2020precise}
Micha{\l} Derezi{\'n}ski, Feynman~T Liang, Zhenyu Liao, and Michael~W. Mahoney.
\newblock Precise expressions for random projections: {{Low-rank}}
  approximation and randomized {{Newton}}.
\newblock In \emph{Advances in Neural Information Processing Systems},
  volume~33, pages 18272---18283. Curran Associates, Inc., 2020.

\bibitem[Derezi{\'n}ski et~al.(2021{\natexlab{a}})Derezi{\'n}ski, Lacotte,
  Pilanci, and Mahoney]{derezinski2021newtonless}
Micha{\l} Derezi{\'n}ski, Jonathan Lacotte, Mert Pilanci, and Michael~W
  Mahoney.
\newblock Newton-less: Sparsification without trade-offs for the sketched
  {Newton} update.
\newblock In \emph{Advances in Neural Information Processing Systems},
  volume~34, pages 2835--2847. Curran Associates, Inc., 2021{\natexlab{a}}.

\bibitem[Derezi{\'n}ski et~al.(2021{\natexlab{b}})Derezi{\'n}ski, Liao,
  Dobriban, and Mahoney]{derezinski2021sparse}
Micha{\l} Derezi{\'n}ski, Zhenyu Liao, Edgar Dobriban, and Michael Mahoney.
\newblock Sparse sketches with small inversion bias.
\newblock In \emph{Proceedings of Thirty Fourth Conference on Learning Theory},
  volume 134, pages 1467--1510. PMLR, 2021{\natexlab{b}}.

\bibitem[Drineas and Mahoney(2016)]{DM16_CACM}
P.~Drineas and M.~W. Mahoney.
\newblock {RandNLA}: Randomized numerical linear algebra.
\newblock \emph{Communications of the ACM}, 59:\penalty0 80--90, 2016.

\bibitem[Drineas and Mahoney(2018)]{RandNLA_PCMIchapter_chapter}
P.~Drineas and M.~W. Mahoney.
\newblock Lectures on randomized numerical linear algebra.
\newblock In M.~W. Mahoney, J.~C. Duchi, and A.~C. Gilbert, editors, \emph{The
  Mathematics of Data}, IAS/Park City Mathematics Series, pages 1--48.
  AMS/IAS/SIAM, 2018.

\bibitem[Drineas et~al.(2006{\natexlab{a}})Drineas, Kannan, and
  Mahoney]{drineas2006fast1}
Petros Drineas, Ravi Kannan, and Michael~W Mahoney.
\newblock Fast {Monte Carlo} algorithms for matrices {I}: Approximating matrix
  multiplication.
\newblock \emph{SIAM Journal on Computing}, 36\penalty0 (1):\penalty0 132--157,
  2006{\natexlab{a}}.

\bibitem[Drineas et~al.(2006{\natexlab{b}})Drineas, Mahoney, and
  Muthukrishnan]{drineas2006sampling}
Petros Drineas, Michael~W Mahoney, and Shan Muthukrishnan.
\newblock Sampling algorithms for $l_2$ regression and applications.
\newblock In \emph{Proceedings of the Seventeenth Annual ACM-SIAM Symposium on
  Discrete Algorithm}, pages 1127--1136, 2006{\natexlab{b}}.

\bibitem[Drineas et~al.(2011)Drineas, Mahoney, Muthukrishnan, and
  Sarl{\'o}s]{drineas2011faster}
Petros Drineas, Michael~W Mahoney, Shan Muthukrishnan, and Tam{\'a}s
  Sarl{\'o}s.
\newblock Faster least squares approximation.
\newblock \emph{Numerische Mathematik}, 117\penalty0 (2):\penalty0 219--249,
  2011.

\bibitem[Drineas et~al.(2012)Drineas, Magdon-Ismail, Mahoney, and
  Woodruff]{drineas2012fast}
Petros Drineas, Malik Magdon-Ismail, Michael~W Mahoney, and David~P Woodruff.
\newblock Fast approximation of matrix coherence and statistical leverage.
\newblock \emph{The Journal of Machine Learning Research}, 13\penalty0
  (1):\penalty0 3475--3506, 2012.

\bibitem[El~Alaoui and Mahoney(2015)]{alaoui2015fast}
Ahmed El~Alaoui and Michael~W Mahoney.
\newblock Fast randomized kernel ridge regression with statistical guarantees.
\newblock In \emph{Advances in Neural Information Processing Systems},
  volume~28. Curran Associates, Inc., 2015.

\bibitem[Fan and Wang(2020)]{fan2020spectra}
Zhou Fan and Zhichao Wang.
\newblock {Spectra of the Conjugate Kernel and Neural Tangent Kernel for
  linear-width neural networks}.
\newblock In \emph{Advances in Neural Information Processing Systems},
  volume~33, pages 7710--7721. Curran Associates, Inc., 2020.

\bibitem[Haff(1979)]{haff1979identity}
L.~R Haff.
\newblock An identity for the {{Wishart}} distribution with applications.
\newblock \emph{Journal of Multivariate Analysis}, 9\penalty0 (4):\penalty0
  531--544, 1979.

\bibitem[Halko et~al.(2011)Halko, Martinsson, and Tropp]{halko2011finding}
Nathan Halko, Per-Gunnar Martinsson, and Joel~A Tropp.
\newblock Finding structure with randomness: {Probabilistic} algorithms for
  constructing approximate matrix decompositions.
\newblock \emph{SIAM Review}, 53\penalty0 (2):\penalty0 217--288, 2011.

\bibitem[Johnson and Lindenstrauss(1984)]{johnson1984extensions}
William~B. Johnson and Joram Lindenstrauss.
\newblock {Extensions of Lipschitz mappings into a Hilbert space}.
\newblock \emph{Contemporary Mathematics}, pages 189--206, 1984.

\bibitem[Kalantzis et~al.(2013)Kalantzis, Bekas, Curioni, and
  Gallopoulos]{kalantzis2013Accelerating}
V.~Kalantzis, C.~Bekas, A.~Curioni, and E.~Gallopoulos.
\newblock Accelerating data uncertainty quantification by solving linear
  systems with multiple right-hand sides.
\newblock \emph{Numerical Algorithms}, 62\penalty0 (4):\penalty0 637--653,
  2013.

\bibitem[Krizhevsky(2009)]{krizhevsky2009Learning}
Alex Krizhevsky.
\newblock \emph{Learning {{Multiple Layers}} of {{Features}} from {{Tiny
  Images}}}.
\newblock PhD thesis, 2009.

\bibitem[Lacotte and Pilanci(2019)]{lacotte2019faster}
Jonathan Lacotte and Mert Pilanci.
\newblock Faster least squares optimization.
\newblock \emph{arXiv preprint arXiv:1911.02675}, 2019.

\bibitem[Lacotte and Pilanci(2022)]{lacotte2022adaptive}
Jonathan Lacotte and Mert Pilanci.
\newblock Adaptive and oblivious randomized subspace methods for
  high-dimensional optimization: {Sharp} analysis and lower bounds.
\newblock \emph{IEEE Transactions on Information Theory}, 68\penalty0
  (5):\penalty0 3281--3303, 2022.

\bibitem[Lacotte et~al.(2021)Lacotte, Wang, and Pilanci]{lacotte2021adaptive}
Jonathan Lacotte, Yifei Wang, and Mert Pilanci.
\newblock Adaptive {Newton} sketch: Linear-time optimization with quadratic
  convergence and effective {Hessian} dimensionality.
\newblock In \emph{Proceedings of the 38th International Conference on Machine
  Learning}, volume 139 of \emph{Proceedings of Machine Learning Research},
  pages 5926--5936. PMLR, 2021.

\bibitem[LeCun et~al.(1998)LeCun, Bottou, Bengio, and
  Haffner]{lecun1998gradient}
Yann LeCun, Leon Bottou, Yoshua Bengio, and Patrick Haffner.
\newblock {Gradient-based learning applied to document recognition}.
\newblock \emph{Proceedings of the IEEE}, 86\penalty0 (11):\penalty0
  2278--2324, 1998.

\bibitem[Liao and Mahoney(2021)]{liao2021hessian}
Zhenyu Liao and Michael~W Mahoney.
\newblock Hessian eigenspectra of more realistic nonlinear models.
\newblock In \emph{Advances in Neural Information Processing Systems},
  volume~34, pages 20104--20117. Curran Associates, Inc., 2021.

\bibitem[Liao et~al.(2020)Liao, Couillet, and Mahoney]{liao2020random}
Zhenyu Liao, Romain Couillet, and Michael~W Mahoney.
\newblock A random matrix analysis of random {Fourier} features: Beyond the
  {Gaussian} kernel, a precise phase transition, and the corresponding double
  descent.
\newblock In \emph{Advances in Neural Information Processing Systems},
  volume~33, pages 13939--13950. Curran Associates, Inc., 2020.

\bibitem[Liao et~al.(2021)Liao, Couillet, and Mahoney]{liao2021sparse}
Zhenyu Liao, Romain Couillet, and Michael~W. Mahoney.
\newblock Sparse quantized spectral clustering.
\newblock In \emph{International Conference on Learning Representations}, 2021.

\bibitem[Ma et~al.(2015)Ma, Mahoney, and Yu]{ma2015statistical}
Ping Ma, Michael~W Mahoney, and Bin Yu.
\newblock A statistical perspective on algorithmic leveraging.
\newblock \emph{Journal of Machine Learning Research}, 16:\penalty0 861--911,
  2015.

\bibitem[Ma et~al.(2022)Ma, Chen, Zhang, Xing, Ma, and
  Mahoney]{ma2022asymptotic}
Ping Ma, Yongkai Chen, Xinlian Zhang, Xin Xing, Jingyi Ma, and Michael~W
  Mahoney.
\newblock Asymptotic analysis of sampling estimators for randomized numerical
  linear algebra algorithms.
\newblock \emph{Journal of Machine Learning Research}, 23\penalty0
  (177):\penalty0 1--45, 2022.

\bibitem[Mahoney(2011)]{mahoney2011randomized}
Michael~W. Mahoney.
\newblock Randomized algorithms for matrices and data.
\newblock \emph{Foundations and Trends® in Machine Learning}, 3\penalty0
  (2):\penalty0 123--224, 2011.

\bibitem[Martinsson and Tropp(2020)]{martinsson2020randomized}
Per-Gunnar Martinsson and Joel~A Tropp.
\newblock Randomized numerical linear algebra: Foundations and algorithms.
\newblock \emph{Acta Numerica}, 29:\penalty0 403--572, 2020.

\bibitem[Mei and Montanari(2022)]{mei2021generalization}
Song Mei and Andrea Montanari.
\newblock The generalization error of random features regression: Precise
  asymptotics and the double descent curve.
\newblock \emph{Communications on Pure and Applied Mathematics}, 75\penalty0
  (4):\penalty0 667--766, 2022.

\bibitem[Murray et~al.(2023)Murray, Demmel, Mahoney, Erichson, Melnichenko,
  Malik, Grigori, Luszczek, Derezi{\'n}ski, Lopes, Liang, Luo, and
  Dongarra]{randlapack_book_v2_arxiv}
R.~Murray, J.~Demmel, M.~W. Mahoney, N.~B. Erichson, M.~Melnichenko, O.~A.
  Malik, L.~Grigori, P.~Luszczek, M.~Derezi{\'n}ski, M.~E. Lopes, T.~Liang,
  H.~Luo, and J.~Dongarra.
\newblock {Randomized Numerical Linear Algebra} -- a perspective on the field
  with an eye to software.
\newblock Technical Report Preprint: arXiv:2302.11474v2, 2023.

\bibitem[Paquette et~al.(2021)Paquette, Lee, Pedregosa, and
  Paquette]{paquette2021SGDa}
Courtney Paquette, Kiwon Lee, Fabian Pedregosa, and Elliot Paquette.
\newblock {{SGD}} in the large: {{Average-case}} analysis, asymptotics, and
  stepsize criticality.
\newblock In \emph{Proceedings of {{Thirty Fourth Conference}} on {{Learning
  Theory}}}, pages 3548--3626. PMLR, 2021.

\bibitem[Paquette et~al.(2023)Paquette, {van Merri{\"e}nboer}, Paquette, and
  Pedregosa]{paquette2023Haltinga}
Courtney Paquette, Bart {van Merri{\"e}nboer}, Elliot Paquette, and Fabian
  Pedregosa.
\newblock Halting time is predictable for large models: A universality property
  and average-case analysis.
\newblock \emph{Foundations of Computational Mathematics}, 23\penalty0
  (2):\penalty0 597--673, 2023.

\bibitem[Pennington and Worah(2017)]{pennington2017nonlinear}
Jeffrey Pennington and Pratik Worah.
\newblock {Nonlinear random matrix theory for deep learning}.
\newblock In \emph{Advances in Neural Information Processing Systems},
  volume~30 of \emph{NIPS'17}, pages 2637--2646. Curran Associates, Inc., 2017.

\bibitem[Pilanci and Wainwright(2016)]{pilanci2016iterative}
Mert Pilanci and Martin~J Wainwright.
\newblock Iterative {Hessian} sketch: Fast and accurate solution approximation
  for constrained least-squares.
\newblock \emph{Journal of Machine Learning Research}, 17\penalty0
  (53):\penalty0 1--38, 2016.

\bibitem[Pilanci and Wainwright(2017)]{pilanci2017Newton}
Mert Pilanci and Martin~J. Wainwright.
\newblock {Newton} sketch: {A} near linear-time optimization algorithm with
  linear-quadratic convergence.
\newblock \emph{SIAM Journal on Optimization}, 27\penalty0 (1):\penalty0
  205--245, 2017.

\bibitem[Plerou et~al.(2002)Plerou, Gopikrishnan, Rosenow, Amaral, Guhr, and
  Stanley]{plerou2002Random}
Vasiliki Plerou, Parameswaran Gopikrishnan, Bernd Rosenow, Lu{\'i}s A.~Nunes
  Amaral, Thomas Guhr, and H.~Eugene Stanley.
\newblock Random matrix approach to cross correlations in financial data.
\newblock \emph{Physical Review E}, 65\penalty0 (6):\penalty0 066126, 2002.

\bibitem[Pukelsheim(2006)]{pukelsheim2006optimal}
Friedrich Pukelsheim.
\newblock \emph{{Optimal Design of Experiments}}.
\newblock SIAM, 2006.

\bibitem[Romanov et~al.(2024)Romanov, Zhang, and Pilanci]{romanov2024newton}
Elad Romanov, Fangzhao Zhang, and Mert Pilanci.
\newblock Newton meets {Marchenko-Pastur}: Massively parallel second-order
  optimization with {Hessian} sketching and debiasing.
\newblock \emph{arXiv preprint arXiv:2410.01374}, 2024.

\bibitem[Roosta-Khorasani and Mahoney(2019)]{roosta2019subsampled}
Farbod Roosta-Khorasani and Michael~W. Mahoney.
\newblock Sub-sampled {Newton} methods.
\newblock \emph{Mathematical Programming}, 174\penalty0 (1-2):\penalty0
  293--326, 2019.

\bibitem[Sarl\'{o}s(2006)]{Sarlos06}
T.~Sarl\'{o}s.
\newblock Improved approximation algorithms for large matrices via random
  projections.
\newblock In \emph{Proceedings of the 47th Annual IEEE Symposium on Foundations
  of Computer Science}, pages 143--152, 2006.

\bibitem[Silverstein and Bai(1995)]{silverstein1995empirical}
Jack~W. Silverstein and Zhidong Bai.
\newblock On the empirical distribution of eigenvalues of a class of large
  dimensional random matrices.
\newblock \emph{Journal of Multivariate Analysis}, 54\penalty0 (2):\penalty0
  175--192, 1995.

\bibitem[Tropp(2011)]{tropp2011improved}
Joel~A Tropp.
\newblock Improved analysis of the subsampled randomized {Hadamard} transform.
\newblock \emph{Advances in Adaptive Data Analysis}, 3\penalty0
  (01n02):\penalty0 115--126, 2011.

\bibitem[Tropp(2015)]{tropp2015introduction}
Joel~A Tropp.
\newblock An introduction to matrix concentration inequalities.
\newblock \emph{Foundations and Trends{\textregistered} in Machine Learning},
  8\penalty0 (1-2):\penalty0 1--230, 2015.

\bibitem[Vershynin(2018)]{vershynin2018high}
Roman Vershynin.
\newblock \emph{{High-Dimensional Probability: An Introduction with
  Applications in Data Science}}.
\newblock Cambridge Series in Statistical and Probabilistic Mathematics.
  Cambridge University Press, 2018.

\bibitem[Wang and Ma(2021)]{wang2021optimal}
Haiying Wang and Yanyuan Ma.
\newblock Optimal subsampling for quantile regression in big data.
\newblock \emph{Biometrika}, 108\penalty0 (1):\penalty0 99--112, 2021.

\bibitem[Wang et~al.(2018)Wang, Zhu, and Ma]{wang2018optimal}
HaiYing Wang, Rong Zhu, and Ping Ma.
\newblock Optimal subsampling for large sample logistic regression.
\newblock \emph{Journal of the American Statistical Association}, 113\penalty0
  (522):\penalty0 829--844, 2018.

\bibitem[Woodruff(2014)]{david2014sketching}
David~P. Woodruff.
\newblock Sketching as a tool for numerical linear algebra.
\newblock \emph{Foundations and Trends® in Theoretical Computer Science},
  10\penalty0 (1–2):\penalty0 1--157, 2014.

\bibitem[Xu et~al.(2016)Xu, Yang, Roosta-Khorasani, Ré, and
  Mahoney]{xu2016subsampled}
Peng Xu, Jiyan Yang, Farbod Roosta-Khorasani, Christopher Ré, and Michael~W.
  Mahoney.
\newblock Sub-sampled {Newton} methods with non-uniform sampling.
\newblock In \emph{Advances in Neural Information Processing Systems},
  volume~29 of \emph{NIPS‘16}, pages 3000--3008. Curran Associates, Inc.,
  2016.

\bibitem[Xu et~al.(2020)Xu, Roosta-Khorasani, and Mahoney]{xu2020newton}
Peng Xu, Farbod Roosta-Khorasani, and Michael~W. Mahoney.
\newblock {Newton-type methods for non-convex optimization under inexact
  Hessian information}.
\newblock \emph{Mathematical Programming}, 184\penalty0 (1-2):\penalty0 35--70,
  2020.

\bibitem[Yao et~al.(2018)Yao, Xu, Roosta-Khorasani, and
  Mahoney]{yao2018inexact}
Zhewei Yao, Peng Xu, Farbod Roosta-Khorasani, and Michael~W Mahoney.
\newblock Inexact non-convex {Newton}-type methods.
\newblock \emph{arXiv preprint arXiv:1802.06925}, 2018.

\bibitem[Ye et~al.(2021)Ye, Luo, and Zhang]{ye2021approximate}
Haishan Ye, Luo Luo, and Zhihua Zhang.
\newblock Approximate {Newton} methods.
\newblock \emph{Journal of Machine Learning Research}, 22\penalty0
  (66):\penalty0 1--41, 2021.

\bibitem[Yu et~al.(2022)Yu, Wang, Ai, and Zhang]{yu2022optimal}
Jun Yu, HaiYing Wang, Mingyao Ai, and Huiming Zhang.
\newblock Optimal distributed subsampling for maximum quasi-likelihood
  estimators with massive data.
\newblock \emph{Journal of the American Statistical Association}, 117\penalty0
  (537):\penalty0 265--276, 2022.

\bibitem[Zhan(2001)]{zhan2001singular}
Xingzhi Zhan.
\newblock Singular values of differences of positive semidefinite matrices.
\newblock \emph{SIAM Journal on Matrix Analysis and Applications}, 22\penalty0
  (3):\penalty0 819--823, 2001.

\bibitem[Zhang and Pilanci(2023)]{pmlr-v202-zhang23ah}
Fangzhao Zhang and Mert Pilanci.
\newblock Optimal shrinkage for distributed second-order optimization.
\newblock In \emph{Proceedings of the 40th International Conference on Machine
  Learning}, volume 202 of \emph{Proceedings of Machine Learning Research},
  pages 41523--41549. PMLR, 23--29 Jul 2023.

\end{thebibliography}
\bibliographystyle{plainnat}

\newpage
\appendix

\begin{center}
  \textbf{\large Supplementary Material of} \\ 
  \textbf{Fundamental Bias in Inverting Random Sampling Matrices\\ with Application to Sub-sampled Newton}
\end{center}

The technical appendices of this paper are organized as follows.
\begin{itemize}
    \item We recall a few technical lemmas that will be used in our proofs in \Cref{sec:lemmas}.
    \item The proof of \Cref{lem:sub_embed} on the subspace embedding property for random sampling is given in \Cref{sec:proof_sub_embedd}.
    \item The proof of \Cref{prop:coarse-RS} on coarse-grained debiasing of random sampling is given in \Cref{sec:proof_coarse-RS}.
    \item The proof of \Cref{theo:inverse-bias} on fine-grained analysis of inversion bias for random sampling is given in \Cref{sec:proof_of_theo:inverse-bias}.
    \item The proof of \Cref{prop:debias} on fine-grained debiasing for random sampling is given in \Cref{sec:proof_prop_debias}.
    \item The proof concerning results on the application to de-biased SSN in \Cref{sec:application} is given in \Cref{sec:proof_application}. 
    \item Implementation details for the numerical results in \Cref{sec:num} are given in~\cref{sec:imple_detail_nmerical_exper}.
\end{itemize}

\section{Useful Lemmas}
\label{sec:lemmas}

In this section, we introduce a few technical lemmas to be used in subsequent sections.

\begin{lemma}[Singular value bounds of symmetric p.s.d.\@ matrices,~{\citet[Theorem~2.1]{zhan2001singular}}]\label{lemm:zhao2001}
For real symmetric p.s.d.\@ matrices $\J,\K \in \mathbb{R}^{s\times s}$ having (ordered) singular values $\sigma_1(\J)\geq\sigma_2(\J)\geq \ldots \geq \sigma_s(\J)$ and $\sigma_1(\K)\geq\sigma_2(\K)\geq \ldots \geq \sigma_s(\K)$, the singular values of the difference $\J-\K$ are bounded as
\begin{align}
    \sigma_j(\J-\K)\leq     \sigma_j \begin{pmatrix}
        \J & 0 \\ 0 &  \K
    \end{pmatrix}, \quad{j=1,\ldots, s}.\nonumber
\end{align}
In particular,  we obtain
\begin{align}
\|\J-\K\|\leq\max\{\|\J\|,\|\K\|\}.\nonumber
\end{align}
\end{lemma}

\begin{lemma}[Intrinsic matrix Bernstein,~{\citet[Theorem~7.3.1]{tropp2015introduction}}] \label{lemm:troop2015}
Let $ \X_1, \ldots, \X_n$ be $n$ independent real symmetric random matrices such that $\EE[\X_i]= \zo$, $\max_{1\leq i \leq n}\|\X_i\|\leq \rho_1$, and $\sum^n_{i=1}\EE[\X^2_i]\preceq \B$ for some symmetric p.s.d.\@ matrix $\B$. 
Then, for any $\epsilon\geq \|\B\|^{1/2}+\rho_1/3$, we have
\begin{align*}
    \Pr \left( \left\|\sum^n_{i=1}\X_i \right\| \geq\epsilon \right)\leq \frac{4 \tr \B}{ \|\B\| } \exp \left(-\frac{\epsilon^2/2}{\|\B\|+\rho_1\epsilon/3} \right).
\end{align*}
\end{lemma}
 
\begin{lemma}[Sherman--Morrison formula]\label{lemm:sherman-morrison}
    For an invertible matrix $\A\in \RR^{n\times n}$ and two vectors $\uu, \vv \in \RR^n$, $\A+\uu\vv^\top$ is invertible if and only if $1+ \vv^\top\A^{-1}\uu\neq 0$ and 
    \begin{align}
        (\A+ \uu\vv^\top)^{-1}=\A^{-1} -\frac{\A^{-1}\uu\vv^\top\A^{-1}}{1+ \vv^\top\A^{-1}\uu}. \nonumber
    \end{align}
    Besides, it also follows that
    \begin{align}
        (\A+ \uu\vv^\top)^{-1}\uu=\frac{\A^{-1}\uu}{1+ \vv^\top\A^{-1}\uu}.\nonumber
    \end{align}
\end{lemma}

\begin{lemma}[Martingale Rosenthal inequality~\citep{de1999decoupling}]\label{lemm:martingale_rosenthal_sharp}
There exists a universal constant $K_p$, with $0 < K_p < \infty$, such that if
$\{d_i\}_{i=1}^n$ is a mean-zero martingale difference sequence  with respect to the increasing $\sigma$-field $\mathcal F_i$, then for all
$p \ge 2$,
\begin{equation}\label{eq:martingale_rosenthal_sharp}
\EE\left[\left|\sum_{i=1}^n d_i\right|^p\right]
\leq
K_p^{\,p}\left(\frac{p}{\log p}\right)^{p}
\left(
\EE\left[\left(
\sum_{i=1}^n \EE[d_i^{2}\mid \mathcal{F}_{i-1}]
\right)^{p/2}\right]
+
\sum_{i=1}^n \EE[|d_i|^{p}]
\right).
\end{equation}
\end{lemma}

We will use \Cref{lemm:martingale_rosenthal_sharp} to establish concentration results for resolvent/inverse matrices. 
In that context, the increasing $\sigma$-field $\mathcal F_i$ is defined with respect to the independent trials of random sampling (i.e., the rows of the random sampling matrix $\S \in \RR^{m \times n}$ in \Cref{def:RS}, see \Cref{lemma:max_expec} in \Cref{subsec:discussion_theo:inverse-bias} and the discussion thereafter.

\begin{theorem}[Hanson-Wright inequality,~\citet{vershynin2018high}]\label{theo: hanson_wright_inequality}
Assume $\x = (x_1, \ldots, x_n) \in \RR^n$ be a random vector with independent, mean-zero, sub-gaussian  variables. 
For a matrix $\M\in\RR^{n\times n}$ and any  $t \geq 0$, then we have
\begin{align*}
    \Pr \left(\left| \x^\top \M \x - \EE[\x^\top \M \x] \right| \geq t \right ) \leq 2 \exp \left[ -c \min\left( \frac{t^2}{K^4 \|\M\|_F^2}, \frac{t}{K^2 \|\M\|} \right) \right],
\end{align*}
where $\|\cdot\|_F$ denotes the Frobenius norm, $K = \max_{1\leq i\leq n} \inf \{s > 0 : \EE[\exp(x_i^2 / s ^2) ]\leq 2\}$, and $c$  is a universal constant. 
\end{theorem}

\section{Proof of \Cref{lem:sub_embed}}
\label{sec:proof_sub_embedd}

In this section, we present the proof of \Cref{lem:sub_embed}.
Recall that our objective is, for $\A\in \RR^{n\times d}$ of rank $d$ with $n\geq d$, p.s.d.\@ $\C\in\RR^{d\times d}$, $\S$ some random sampling matrix with number of trials $m$ and importance sampling distribution $\{ \pi_i \}_{i=1}^n$ as in \Cref{def:RS}, $d_{\eff}=\tr( \A_\C^\top \A_\C )$ the effective dimension of $\A$ given $\C$ with $\A_\C \equiv \A(\A^\top \A+\C)^{-1/2}$ as in \Cref{def:lev}, and max importance sampling approximation  factor $\rho_{\max}$ as in \Cref{def:approx_factor}, to show that for $m\geq C\rho_{\max} d_{\eff} \log(d_{\eff}/\delta )/\epsilon^2$, $ \A_\C^\top\S^\top\S\A_\C$ is an $(\epsilon,\delta)$-approximation of $\A_\C^\top\A_\C$.
That is
\begin{equation}
  (1+\epsilon)^{-1} \A_\C^\top \A_\C \preceq \A_\C^\top\S^\top\S\A_\C \preceq(1+\epsilon) \A_\C^\top \A_\C,
\end{equation}
holds with probability at least $1 - \delta$.

First note that, for $\A_\C \equiv \A(\A^\top \A+\C)^{-1/2} \in \RR^{n \times d}$, we have
\begin{align*}
  \A_{\C}^\top\S^\top\S\A_{\C}-\A_{\C}^\top\A_{\C} = \sum_{s=1}^{m} \left( \frac{{\ba_{\C}}_{i_s} {\ba_{\C}}^\top_{i_s} }{m\pi_{i_s}}  -\frac{\A_{\C}^\top\A_{\C}}{m} \right) \equiv \sum_{s=1}^{m}\F_s,
\end{align*}
where ${\ba^\top_{\C}}_i \in \RR^{d}$ denotes the $i^{th}$ row of $\A_\C$.
Given $\A_{\C}$, it follows from \Cref{def:RS} that the indices $i_s$ are i.i.d.\@ drawn with replacement from the index set $\{ 1, \ldots, n \}$.
We would like to apply the Intrinsic matrix Bernstein, \Cref{lemm:troop2015}, to bound the sum of random matrices $\A_{\C}^\top\S^\top\S\A_{\C}-\A_{\C}^\top\A_{\C} = \sum_{s=1}^{m}\F_s$.
To this end, note first that given $\A_\C$ we have
\begin{align} \label{eq:expec_F_t}
  \EE[ \F_s ] = \zo_d.
\end{align}
It then remains to bound (i) the spectral norm $\| \F_s \|$ and (ii) the sum $\sum^m_{s=1}\EE[\F^2_s]$ in the sense of p.s.d.\@ matrices.
By Lemma \ref{lemm:zhao2001} and the triangle inequality, for $s=1,\ldots, m$, we get
\begin{align}
    \|\F_s\|&\leq\mathop{\rm{max}}\limits_{1\leq i \leq n} \left\{ \left\|\frac{ {\ba_{\C}}_{i} {\ba_{\C}}_{i}^\top }{m \pi_{i}} \right\|, \frac{1}{m} \right\} =\frac{1}{m}\mathop{\rm{max}}\limits_{1\leq i \leq n} \left\{ \frac{\|{\ba_{\C}}_{i}\|^2}{\pi_{i}}, 1 \right\}  =\frac{\rho_{\max}  d_{\eff}}{m}  \equiv\rho_1,\nonumber
\end{align}
where we recall $\| \A_\C^\top \A_\C \| \leq 1$ and $\rho_{\max} = \max_{1 \leq i \leq n} \|{\ba_{\C}}_{i}\|^2/(\pi_i d_{\eff}) \geq 1$ is the max importance sampling approximation factor in \Cref{def:approx_factor}.

Further note that, per its definition $\F_s = \frac{{\ba_{\C}}_{i_s} {\ba_{\C}}^\top_{i_s} }{m\pi_{i_s}}  -\frac{\A_{\C}^\top\A_{\C}}{m}$, we have 
\begin{align}\label{eq:ft=mt}
  \left(\F_s+\frac{\A_{\C}^\top\A_{\C}}{m} \right)^2= \frac{ \| {\ba_{\C}}_{i_s} \|^2 \cdot {\ba_{\C}}_{i_s} {\ba_{\C}}^\top_{i_s} }{m^2\pi_{i_s}^2} = \F^2_s + \F_s\frac{\A_{\C}^\top\A_{\C}}{m}+\frac{\A_{\C}^\top\A_{\C}}{m}\F_s+\frac{(\A_{\C}^\top\A_{\C})^2}{m^2}.
\end{align}
Taking expectations on both sides of \eqref{eq:ft=mt} and applying \eqref{eq:expec_F_t}, we get
\begin{align*}
    \EE[\F^2_s] + \frac{(\A_{\C}^\top\A_{\C})^2}{m^2}&= \sum_{i=1}^{n}\pi_i\frac{ \|{\ba_{\C}}_{i}\|^2 \cdot {\ba_{\C}}_{i}{\ba_{\C}}_{i}^\top}{m^2\pi^2_i}\preceq \frac{\rho_{\max} d_{\eff}}{m^2}\A_{\C}^\top\A_{\C}.
\end{align*}
Thus, 
\begin{align} 
        \sum_{s=1}^{m}\EE[\F^2_s]& \preceq  \frac{\rho_{\max} d_{\eff}}{m}\A_{\C}^\top\A_{\C}
         \equiv \P,\nonumber
\end{align}
for which we have $\|\P\|=\frac{\rho_{\max} d_{\eff}}{m}\|\A_{\C}^\top\A_{\C}\|$ and $\tr(\P)=\frac{\rho_{\max} d^2_{\eff}}{m}$, so that $\frac{ \tr \P }{ \| \P \|} =  \frac{d_{\eff}}{\|\A_{\C}^\top\A_{\C}\| } $.

Recall that $\|\A_{\C}^\top\A_{\C}\| \leq 1$, applying \Cref{lemm:troop2015}, we obtain
\begin{align}
&1-\Pr \left(\frac{1}{1+\epsilon}\A_{\C}^{\top}\A_{\C}\preceq\A_{\C}^\top\S^\top\S\A_{\C}\preceq (1+\epsilon)\A_{\C}^\top\A_{\C} \right)  \nonumber\\
    &\leq   \Pr \left(\|\A_{\C}^\top\S^\top\S\A_{\C}-\A_{\C}^\top\A_{\C}\|>\epsilon \right)= \Pr \left( \left\|\sum_{s=1}^{m}\F_s \right\|>\epsilon \right) \notag\\
        &\leq 4 \|\A_{\C}^\top\A_{\C}\|^{-1}d_{\eff}\exp \left(-\frac{\epsilon^2/2}{\frac{\rho_{\max}  d_{\eff}}{m}\|\A_{\C}^\top\A_{\C}\|+\frac{\rho_{\max}  d_{\eff}}{m}\epsilon/3 } \right)\notag\\
        &=4 \|\A_{\C}^\top\A_{\C}\|^{-1}d_{\eff}\exp \left(-\frac{\epsilon^2}{\frac{\rho_{\max}  d_{\eff}}{m}( 2\|\A_{\C}^\top\A_{\C}\|+2\epsilon/3)} \right). \nonumber
\end{align}
Clearly,
 $\Pr(\frac{1}{1+\epsilon}\A_{\C}^{\top}\A_{\C}\preceq\A_{\C}^\top\S^\top\S\A_{\C}\preceq (1+\epsilon)\A_{\C}^\top\A_{\C})\leq1-\delta$ holds  if
\begin{align} \label{eq:error=delta}
4\|\A_{\C}^\top\A_{\C}\|^{-1}d_{\eff}\exp \left(-\frac{\epsilon^2}{\frac{\rho_{\max}  d_{\eff}}{m}( 2\|\A_{\C}^\top\A_{\C}\|+2\epsilon/3)} \right) \leq\delta,
\end{align}
that can be rewritten as
\begin{align}\label{prooflem5.9}
m\geq \frac{{\rho_{\max}  d_{\eff}}( 2\|\A_{\C}^\top\A_{\C}\|+2\epsilon/3)}{\epsilon^2}\log \left(\frac{4\|\A_{\C}^\top\A_{\C}\|^{-1}d_{\eff}}{\delta} \right).
\end{align}
Since $\epsilon\leq1$ and $\|\A_{\C}^\top\A_{\C}\|\leq 1$, \eqref{prooflem5.9} holds if $m\geq  \frac{8{\rho  d_{\eff}}}{3\epsilon^2}\log(\frac{4\|\A_{\C}^\top\A_{\C}\|^{-1}d_{\eff}}{\delta})$.

Lastly, let us check the condition $\epsilon\geq \|\P\|^{1/2}+\rho_1/3$ in \Cref{lemm:troop2015} is satisfied. 
Solving \eqref{eq:error=delta} for $\epsilon$, we get
\begin{align*}
& \quad \quad \ 4\|\A_{\C}^\top\A_{\C}\|^{-1}d_{\eff}\mathrm{exp} \left(-\frac{\epsilon^2}{\frac{\rho_{\max} d_{\eff}}{m}( 2\|\A_{\C}^\top\A_{\C}\|+2\epsilon/3)} \right)\leq\delta\\
&\Longrightarrow 
\epsilon^2-\frac{2\rho_{\max}  d_{\eff}}{3m}\log (4\|\A_{\C}^\top\A_{\C}\|^{-1}d_{\eff}/\delta) \epsilon-6\frac{\rho_{\max}  d_{\eff}\|\A_{\C}^\top\A_{\C}\|}{3m}\log(4\|\A_{\C}^\top\A_{\C}\|^{-1}d_{\eff}/\delta) \geq0 \\
&\Longrightarrow \epsilon^2-2v \epsilon-6\|\A_{\C}^\top\A_{\C}\|v\geq0 \\
&\Longrightarrow \epsilon\geq {v+\sqrt{v^2+6\|\A_{\C}^\top\A_{\C}\|v}} ,
\end{align*}
where $v=\frac{\rho_{\max}  d_{\eff}}{3m}\log(\frac{4\|\A_{\C}^\top\A_{\C}\|^{-1}d_{\eff}}{\delta}) $.
As $\delta<1$ and $4\|\A_{\C}^\top\A_{\C}\|^{-1}d_{\eff}>e$, we have $\log(\frac{4\|\A_{\C}^\top\A_{\C}\|^{-1}d_{\eff}}{\delta})\geq1$. 
As such, we get $v\geq\rho_1/3 $ and $6\|\A_{\C}^\top\A_{\C}\|v\geq \|\P\|$, so that $ \epsilon\geq \|\P\|^{1/2}+\rho_1/3 $ holds.
This concludes the proof of \Cref{lem:sub_embed}.
\qedwhite

Given the subspace embedding result in \Cref{lem:sub_embed}, it can be shown that the inversion $(\A^\top \S^\top \S\A+\C)^{-1}$ also satisfies a similar subspace embedding property and is close to the (population) inversion $(\A^\top \A + \C)^{-1}$. 
This is given in the following result.
\begin{lemma}\label{lemm:sub_embed_inversion}
Under the settings and notations of \Cref{lem:sub_embed},   
we have that $(\A^\top\S^\top \S\A+\C)^{-1}$ is an $(\epsilon,\delta)$-approximation of $(\A^\top\A+\C)^{-1}$. 
\end{lemma}

\begin{proof}[Proof of \Cref{lemm:sub_embed_inversion}]
Denote $\H=\A^\top\A+\C$ (so that $\| \H^{-\frac{1}{2}}\C\H^{-\frac{1}{2}} \| \leq 1$), it then follows from \Cref{lem:sub_embed} that $\A_\C^\top\S^\top \S\A_\C$ is an $(\epsilon,\delta)$-approximation of $\A_\C^\top\A_\C$, that is,
\begin{align}
& ~~~~\frac{1}{1+\epsilon}\A_\C^\top\A_\C\preceq\A_\C^\top\S^\top \S\A_\C\preceq(1+\epsilon)\A_\C^\top\A_\C\nonumber\\
&  \Rightarrow \frac{1}{1+\epsilon}(\A_\C^\top\A_\C+\H^{-\frac{1}{2}}\C\H^{-\frac{1}{2}})\preceq\A_\C^\top\S^\top \S\A_\C+\H^{-\frac{1}{2}}\C\H^{-\frac{1}{2}}\preceq(1+\epsilon)(\A_\C^\top\A_\C+\H^{-\frac{1}{2}}\C\H^{-\frac{1}{2}})\nonumber\\
& \Rightarrow \frac{1}{1+\epsilon}\H\preceq\A^\top\S^\top\S\A+\C\preceq(1+\epsilon)\H.\nonumber\\
  & \Rightarrow \frac{1}{1+\epsilon}\H^{-1}\preceq(\A^\top\S^\top\S\A+\C)^{-1}\preceq(1+\epsilon)\H^{-1},\nonumber
\end{align}
where we recall the definition $\A_\C = \A (\A^\top \A + \C)^{-1/2}$.
This concludes the proof of \Cref{lemm:sub_embed_inversion}.
\end{proof}

\section{Proof of \Cref{prop:coarse-RS} }
\label{sec:proof_coarse-RS}

The proof of \Cref{prop:coarse-RS} follows roughly the same line as that of~\citet[Theorem~6]{derezinski2021newtonless}~and~\citet[Theorem~11]{derezinski2021sparse}.
It is provided here for completeness.

Here, we focus on the differences from the proof in~\citet{derezinski2021newtonless,derezinski2021sparse} by considering now $\S$ under study is a random sampling matrix as in \Cref{def:RS}. 

To start with, we introduce the following condition that is crucial to the proof in~\citet{derezinski2021newtonless, derezinski2021sparse} and of our \Cref{prop:coarse-RS} here.

\begin{condition}[Concentration property of random vector $\s$]\label{con:property_S}
    Given $\V\in \RR^{n\times d}$, the $n$-dimensional random vector $\s$ satisfies $\Var[\s^\top \V\B\V^\top \s]\leq \alpha\cdot \tr(\V\B^2\V^\top)$ for all p.s.d.\@ matrices $\B \in \RR^{d \times d}$ and some $\alpha>0$.
\end{condition}

The proof of \Cref{prop:coarse-RS} then comes in the following two steps:
\begin{enumerate}
  \item construct a high probability event $\zeta$ (via subspace embedding-type results in \Cref{lem:sub_embed}), on which the inverse $(\A^\top \S^\top \S \A + \C)^{-1}$ exists (in particular for $\C = \zo_d$); and
  \item conditioned on that event $\zeta$, bound the spectral norm $ \| \I_d-\EE_{\zeta} [\tilde\Q]\H \|$ via ``leave-one-out'' type analysis, for $\H=\A^\top\A+\C$ and $\tilde\Q=(\gamma\A^\top \S^\top\S \A +\C)^{-1}$ with $\gamma=\frac{m}{m-d_{\eff}}$.
\end{enumerate}

Let us start with the first step, for random sampling in \Cref{def:RS}, denote $\x^\top_s=\ee^\top_{i_s}/\sqrt{\pi_{i_s}}\A \in \RR^d$ the $i^{th}$ row of the sketch $\tilde \A$, so that $\EE[\x_s{\x}^{\top}_s]=\A^\top\A$.
Denote
\begin{equation}\label{eq:def_H_Q_proof_prop:coarse-RS}
  \H=\A^\top\A+\C, \quad \tilde\Q=(\gamma\A^\top \S^\top\S \A +\C)^{-1}= \left(\sum^{m}_{s=1}\frac{\gamma}{m}\x_s{\x}^{\top}_s+\C \right)^{-1}, \quad \gamma=\frac{m}{m-d_{\eff}},
\end{equation}
and
\begin{equation*}
  \tilde\Q_{-s}= \left(\sum_{j\neq s}\frac{\gamma}{m}\x_j{\x}^{\top}_j +\C \right)^{-1},
\end{equation*} 
that is \emph{independent} of $\x_s$.

Without loss of generality, assume that $t=m/3$ is an integer, and define the following events:
\begin{align}\label{eq:events}
    \zeta_j:\sum^{tj}_{s=t(j-1)+1}\frac{1}{t}\x_s{\x}^{\top}_s\succeq \frac{1}{2}   \A^\top\A ,~~~j=1,2,3,~~~ \zeta=\bigcap^3_{j=1}\zeta_j.
\end{align}
Recall that $\gamma=\frac{m}{m-d_{\eff}} > 1$, the events $\zeta_j$ imply
\begin{align}
 \sum^{tj}_{s=t(j-1)+1}\frac{\gamma}{t}\x_s{\x}^{\top}_s   \succeq \frac{1}{2}   \A^\top\A, ~~~j=1,2,3.\nonumber
\end{align}
Here, the event $\zeta_j$ presents that the (weighted) average of the rank-one matrices $\x_s\x_s^\top$ over the corresponding $j$-th third of indices $1,\ldots,n$ forms a sketch (of size $t = m/3$) of $\A^\top \A$ that is a ``lower'' $1/2$-spectral-approximation of $\A^\top\A$ in the sense of \Cref{def:rela_error_approxi}.

In the case of random sampling in \Cref{def:RS}, the events $\zeta_1$, $\zeta_2$, and $\zeta_3$ are independent.  
As such, for each $s\in \{1,\ldots, m\}$, there exists an index $j=j(s)\in\{1,2,3\}$ such that
\begin{enumerate}
    \item $  \zeta_j$ is independent of $\x_{s}$; and 
    \item conditioned on $  \zeta_j$ one has $\tilde \Q\preceq\tilde \Q_{-s}\preceq6\H^{-1}$, which, for $m>2 d_{\eff}$, further leads to $\gamma \H^{1/2}\tilde \Q_{-s} \H^{1/2} \preceq 12 \I_d$.
\end{enumerate}

Following the decomposition of $\EE_{\zeta} [\tilde\Q]$ in \citet{derezinski2021newtonless,derezinski2021sparse} with $\tilde \gamma_s=1+\frac{\gamma}{m}\x_s^\top\tilde\Q_{-s}\x_s$, we write
\begin{equation}\label{eq:def_tZ_prop:coarse-RS}
\I_d-\EE_{\zeta} [\tilde\Q]\H =\underbrace{\EE_{\zeta}[\tilde\Q_{-s}(\x_s\x_s^\top-\A^\top\A)]}_{\tilde\Z_0} +\underbrace{\EE_{\zeta}[\tilde\Q_{-s}-\tilde\Q]\A^\top\A}_{\tilde\Z_1}
+\underbrace{\EE_{\zeta}\left[\left(\frac{\gamma}{\tilde \gamma_s}-1\right)\tilde\Q_{-s}\x_s\x_s^\top    \right]}_{\tilde\Z_2}.
\end{equation}

Recall that $\tilde \Q$ in \eqref{eq:def_H_Q_proof_prop:coarse-RS} is p.s.d.\@ symmetric, to establish the result in \Cref{prop:coarse-RS}, it suffices to bound the spectral norm
\begin{align}\label{eq:tilde_hqh-I}
    \|\I_d-\H^{\frac{1}{2}}\EE_{\zeta} [\tilde\Q]\H^{\frac{1}{2}}\|
    \leq \|\H^{\frac{1}{2}}\tilde\Z_0\H^{-\frac{1}{2}}\|
    +\|\H^{\frac{1}{2}}\tilde\Z_1\H^{-\frac{1}{2}}\|+\|\H^{\frac{1}{2}}\tilde\Z_2\H^{-\frac{1}{2}}\|,
\end{align}
for $\tilde \Z_0, \tilde \Z_1, \tilde \Z_2$ defined in \eqref{eq:def_tZ_prop:coarse-RS}.

Without loss of generality, assume that the events $\zeta_1$ and $\zeta_2$ are independent of $\x_{s}$, and set $\zeta^{'} =\zeta_1\bigcap \zeta_2$  as well as $\delta_3=\Pr(\neg\zeta_3)$. 
To bound \eqref{eq:tilde_hqh-I}, we first recall the following results from the proof of \citet{derezinski2021newtonless,derezinski2021sparse}.

\begin{lemma}
Under the settings and notations of \Cref{prop:coarse-RS}, we have,
\begin{enumerate}
  \item for a p.s.d.\@ random matrix $\M$ (or a non-negative random variable) living in the probability space of $\S$, 
\begin{align}
    \EE_{\zeta}[\M]=\frac{\EE[(\prod^3_{j=1}\mathbf{1}_{\zeta_j})\M]}{\Pr(\zeta)} \preceq \frac{1}{1-\delta}  \EE[\mathbf{1}_{\zeta^{'}}\M] \preceq 2 \EE_{\zeta^{'}}[\M], \label{eq:condition_expec_three}
\end{align}
where $\mathbf{1}_{\zeta_j}$ is the indicator of the event $\zeta_j$.
  \item $\|\H^{1/2}\tilde\Z_1\H^{-1/2}\|=O(\frac{1}{m})$.
\end{enumerate}
\end{lemma}

Next, we bound the terms invoking $\tilde \Z_0$ and $\tilde \Z_2$ in \eqref{eq:tilde_hqh-I}, particularly by emphasizing the difference between our proof of \Cref{prop:coarse-RS} here from that of \citet{derezinski2021newtonless,derezinski2021sparse}. 
By $\delta_3\leq \frac{1}{2}$ and
\begin{align*}
  \tilde  \Z_0= -\frac{1}{1-\delta_3}\EE_{\zeta'}[\tilde\Q_{-s}(\x_s\x_s^\top-\A^\top\A)\cdot \mathbf{1}_{\neg\zeta_3}],
\end{align*}
it follows that
\begin{align}
    \|\H^{\frac{1}{2}}\tilde\Z_0\H^{-\frac{1}{2}}\| \leq 12 \EE_{\zeta'}[(1+\x_s^\top\H^{-1}\x_s)\cdot\mathbf{1}_ {\neg\zeta_3}]. \nonumber
\end{align}

We now need to bound the (conditional) expectation $\EE_{\zeta'}[(1+\x_s^\top\H^{-1}\x_s)\cdot\mathbf{1} \neg\zeta_3]$.
To do this, we resort to \Cref{con:property_S} with an appropriate choice of $\alpha$.
Note that
\begin{equation*}
  \EE[\x_s^\top\H^{-1}\x_s] = \tr (\H^{-1} \A^\top \A) = d_{\eff},
\end{equation*}
per \Cref{def:lev}, it follows that
\begin{align*}
\Var[\x_s^\top\H^{-1}\x_s]&=\Var[\ee^\top_{i_s}/\sqrt{\pi_{i_s}}\A\H^{-1}\A^\top\ee_{i_s}/\sqrt{\pi_{i_s}}]\nonumber\\
&=\EE[(\x_s^\top\H^{-1}\x_s)^2]-(\EE[\x_s^\top\H^{-1}\x_s])^2 = \sum^n_{i=1}\frac{(\ba_i^\top\H^{-1}\ba_i)^2}{\pi_i} -d_{\eff}^2\nonumber\\
   &\leq \rho_{\max}  d^2_{\eff}- d^2_{\eff} = O(\rho_{\max} d^2_{\eff}),
\end{align*}
for max  importance  sampling approximation factor $\rho_{\max}$ in \Cref{def:approx_factor}, which implies that $\alpha$ in \Cref{con:property_S} satisfies
\begin{equation*}
  \alpha =  O( \rho_{\max} d_{\eff}).
\end{equation*}
Further, by Chebyshev's inequality, one has, for $x\geq 2d_{\eff}$, that $\Pr(\x_s^\top\H^{-1}\x_s\geq x ~|~\zeta^{'}) \leq \alpha d_{\eff}/x^2$.
This further leads to, for $\delta_3\leq 1/m^3$,
\begin{align*}
  \|\H^{\frac{1}{2}}\tilde\Z_0\H^{-\frac{1}{2}}\|&\leq 12   
  \int^\infty_0 \Pr(\x_s^\top\H^{-1}\x_s\cdot \mathbf{1}_{\neg\zeta_3}\geq x~|~\zeta')dx +O\Big(\frac{1}{m^3}\Big)\nonumber\\
    &\leq 24m^2\delta_3+12\int^\infty_{2m^2}\Pr(\x_s^\top\H^{-1}\x_s\geq x~|~\zeta')dx +O\Big(\frac{1}{m^3}\Big)\nonumber\\
    &\leq O\Big(\frac{1}{m}\Big) +\frac{ \alpha d_{\eff}}{m^2}=O \left(\frac{ \alpha \sqrt{d_{\eff}}}{m} \right).
\end{align*}

We then move on to bound the last term concerning $\tilde{\Z}_2$ in \eqref{eq:tilde_hqh-I}. 
Applying Cauchy-Schwarz inequality twice, we get
\begin{align}
    \|\H^{\frac{1}{2}}\tilde\Z_2\H^{-\frac{1}{2}}\| &\leq \underbrace{\sqrt{\EE_{\zeta} \left[(\tilde\gamma_s-\gamma)^2  \right]}}_{\tilde  T_1} \cdot \underbrace{ \sup\limits_{\|\uu\|=1}\sqrt[4]{\EE_{\zeta} \left[(\uu^\top \H^{\frac{1}{2}}\tilde \Q_{-s}\x_s)^4\right]} }_{\tilde  T_2}\cdot \underbrace{ \sup\limits_{\|\uu\|=1}\sqrt[4]{\EE_{\zeta} \left[(\uu^\top \H^{-\frac{1}{2}}\x_s)^4\right]} }_{\tilde  T_3}.\label{eq:tilde_z_2}
\end{align}
We start by bounding the term $\tilde T_3$ using \Cref{con:property_S} with $\B=\A\H^{-1/2}\uu\uu^\top\H^{-1/2}\A^\top$. Noting that  $\tr(\B)=\uu^\top\H^{-1/2}\A^\top\A\H^{-1/2}\uu \leq 1$, we get
\begin{align}
    \EE_{\zeta} \left[(\uu^\top \H^{-\frac{1}{2}}\x_s)^4\right]&\leq 2 \EE_{\zeta^{'}} \left[(\uu^\top \H^{\frac{1}{2}}\x_s)^4\right]=2\EE[(\x_s^\top\H^{-\frac{1}{2}}\uu\uu^\top\H^{-\frac{1}{2}}\x_s)^2] \nonumber\\
    &=2 \Var_{\zeta^{'}}[(\x_s^\top\H^{-\frac{1}{2}}\uu\uu^\top\H^{-\frac{1}{2}}\x_s)^2]+2( \EE_{\zeta^{'}} [\x_s^\top\H^{-\frac{1}{2}}\uu\uu^\top\H^{-\frac{1}{2}}\x_s])^2\nonumber\\
    &\leq 2\sum^n_{i=1}\frac{(\ba_i^\top\H^{-\frac{1}{2}}\uu\uu^\top\H^{-\frac{1}{2}}\ba_i)^2}{\pi_i}+2(\tr(\B))^2\nonumber\\
    &\leq 2 \rho_{\max} d_{\eff}\sum^n_{i=1}\uu^\top\H^{-\frac{1}{2}}\ba_i\ba_i^\top\H^{-\frac{1}{2}}\uu+2(\tr(\B))^2\nonumber\\
    &\leq 2 \rho_{\max} d_{\eff}\tr(\B)+2(\tr(\B))^2\leq 2(\rho d_{\eff}+1) = O(\alpha+1), \nonumber
\end{align}
for $\alpha=O(\rho_{\max} d_{\eff})$ in \Cref{con:property_S}.
This results in $\tilde T_3=O(\sqrt[4]{\alpha+1})$.  
Similarly, we bound $\tilde T_2$ by taking $\B=\A\tilde \Q_{-s}\H^{1/2}\uu\uu^\top\H^{1/2}\tilde \Q_{-s}\A^\top$ in \Cref{con:property_S}. 
It follows $\tr(\B)\leq \uu^\top (\H^{1/2}\tilde \Q_{-s}\H^{1/2})^2 \uu \leq 6^2$ and $\tr(\B^2)\leq 6^4$, so that $\tilde T_2=O(\sqrt[4]{\alpha+1})$. 

It thus remains to bound the first term $\tilde T_1$ in \eqref{eq:tilde_z_2}. 
Noting $\bar\gamma=\EE_{\zeta^{'}}[\tilde \gamma_s] = 1+\frac{\gamma}{ m}\tr(\EE_{\zeta^{'}}[\tilde\Q_{-s}]\A^\top\A)$, we write
\begin{align}
  \EE_{\zeta} \left[(\tilde\gamma_s-\gamma)^2  \right] &\leq 2(\gamma-\bar\gamma)^2+\frac{2\gamma^2}{m^2}\EE_{\zeta^{'}}[(\tr(\tilde\Q_{-s}-\EE_{\zeta^{'}}[\tilde\Q_{-s}])\A^\top\A)^2] \nonumber\\
  & +\frac{2\gamma^2}{m^2}\EE_{\zeta^{'}}[(\tr(\A\tilde\Q_{-s}\A^\top)-\x_s^\top\tilde\Q_{-s}\x_s)^2].\nonumber
\end{align}
Analogously as above, letting $\B=\A\tilde\Q_{-s}\A^\top$ so that $\tr(\B^2)\leq 36 d_{\eff}$ in \Cref{con:property_S}, we have, for $m\geq 2d_{\eff}$ that
\begin{align}
  &\frac{2\gamma^2}{m^2}\EE_{\zeta^{'}}\left[\left(\tr(\A\tilde\Q_{-s}\A^\top)-\x_s^\top\tilde\Q_{-s}\x_s\right)^2\right]=  \frac{2\gamma^2}{m^2}\EE_{\zeta^{'}}\left[\sum^n_{i=1}\pi_i(\tr(\A\Q_{-s}\A^\top)-\frac{\ba_i^\top\tilde\Q_{-s}\ba_i}{\pi_i})^2\right] \nonumber\\
  &= \frac{2\gamma^2}{m^2}\EE_{\zeta^{'}}\left[\left(\tr(\A\Q_{-s}\A^\top)\right)^2\right]+ \frac{2\gamma^2}{m^2}\EE_{\zeta^{'}}\left[\sum^n_{i=1}\frac{(\ba_i^\top\tilde\Q_{-s}\ba_i)^2}{\pi_i}\right]- \frac{4\gamma^2}{m^2}\EE_{\zeta^{'}}\left[\left(\tr(\A\Q_{-s}\A^\top)\right)^2\right]\nonumber\\
  &\leq
  \frac{72\gamma^2\rho_{\max} d^2_{\eff}}{m^2}- \frac{2\gamma^2}{m^2}\EE_{\zeta^{'}}\left[\left(\tr(\A\Q_{-s}\A^\top)\right)^2\right]\nonumber\\
  & \leq \frac{72\gamma^2\rho_{\max} d^2_{\eff}}{m^2}\leq \frac{288\rho_{\max} d^2_{\eff}}{m^2}= O \left( \frac{\alpha d_{\eff}}{m^2} \right), \nonumber
\end{align}
with again $\alpha=O(\rho_{\max} d_{\eff})$ in \Cref{con:property_S}.

This, following the line of arguments in~\citet{derezinski2021newtonless}, further leads to
\begin{equation*}
 \EE_{\zeta} \left[(\tilde\gamma_s-\bar\gamma)^2  \right]\leq O \left( \frac{\alpha d_{\eff}}{m^2} \right),
\end{equation*} 
which, together with $|\gamma-\bar\gamma|=O(\sqrt{\alpha d_{\eff}}/m)$, yields that $\tilde T_1=O(\sqrt{\alpha d_{\eff} }/m)$. 

This allows us to conclude that
\begin{equation*}
  \|\H^{\frac{1}{2}}\tilde\Z_2\H^{-\frac{1}{2}}\|\leq \tilde T_1\cdot\tilde T_2\cdot\tilde T_3=O \left(\frac{\alpha\sqrt{ d_{\eff}}}{m} \right).
\end{equation*}
Putting everything together, we conclude that
\begin{equation*}
  \|\I_d-\H^{\frac{1}{2}}\EE_{\zeta} [\tilde\Q]\H^{\frac{1}{2}}\| = O \left(\frac{\alpha\sqrt{ d_{\eff}}}{m} \right),
\end{equation*}
with $ \alpha=O(\rho_{\max} d_{\eff})$ in \Cref{con:property_S}.
This concludes the proof of \Cref{prop:coarse-RS}.
\qedwhite

\section{Proof and Discussions of \Cref{theo:inverse-bias}}
\label{sec:proof_of_theo:inverse-bias}

In this section, we start by presenting in \Cref{subsec:RMT_intuition_theo:inverse-bias} the intuition for the self-consistent equation in \eqref{eq:def_Dii} of \Cref{theo:inverse-bias}.
The detailed proof of \Cref{theo:inverse-bias} is given in \Cref{subsec:detailed_proof_theo:inverse-bias}.
In \Cref{subsec:discussion_theo:inverse-bias}, we provide discussions and auxiliary results on \Cref{theo:inverse-bias}.

\subsection{RMT Intuition on the Self-consistent Equation in \Cref{theo:inverse-bias}}
\label{subsec:RMT_intuition_theo:inverse-bias}

Here, we present a heuristic derivation of the self-consistent equation in \eqref{eq:def_Dii} of \Cref{theo:inverse-bias}.
Let us recall some notations from \Cref{theo:inverse-bias}.
Let $\x^\top_s=\ee^\top_{i_s}/\sqrt{\pi_{i_s}}\A$ as in \Cref{sec:proof_coarse-RS}.
For the ease of further use, we  denote 
\begin{equation*}
  \Q=(\A^\top\S^\top\S \A +\C)^{-1}= \left(\frac{1}{m} \sum^{m}_{s=1} \x_s{\x}^{\top}_s + \C \right)^{-1}.
\end{equation*}
 and $ \Q_{-s}=(\sum_{j\neq s}\frac{1}{m} \x_j{\x}^{\top}_j + \C)^{-1}$, for which we get
\begin{equation*}
  \sum^{m}_{s=1}\frac{1}{m} \EE[\x_s{\x}^{\top}_s]=\sum^n_{i=1} {\ba}_i{\ba}_i^\top=\A^\top\A.
\end{equation*}

Then, we follow the \emph{deterministic equivalent} framework (see~\citet[Chapter~2]{couillet2022RMT4ML} for an introduction) and show that $\|\EE[ \Q]-\widetilde\H^{-1}\|\simeq 0$, for $\widetilde{\H}= \A^\top\D\A+\C$, where $\D\in \RR^{n\times n}$ is the diagonal matrix in \Cref{theo:inverse-bias}. 
Note that
 \begin{align} 
    \|\EE[ \Q]-\widetilde\H^{-1}\|= \|\EE[\Q]\A^\top  \D\A\widetilde\H^{-1}-\EE[\Q\A^\top \S^\top\S \A ]\widetilde\H^{-1}\| \simeq 0,\nonumber
 \end{align}
and using Sherman-Morrison formula in \Cref{lemm:sherman-morrison}, we further ascertain that
\begin{align*}
\EE[ \Q\A^\top \S^\top\S \A \widetilde\H^{-1}]&=\sum^m_{s=1}\EE \left[\frac{\frac{1}{m} \Q_{-s}\x_s\x_s^\top\widetilde\H^{-1}}{1+\x_s^\top \Q_{-s}\x_s/m}\right]=\sum^n_{i=1}\EE\left[\frac{ \Q_{-s}\ba_i\ba_i^\top\widetilde\H^{-1}}{1+\ba_i^\top \Q_{-s}\ba_i/m\pi_{i}}\right] \\ 
&\simeq \sum^n_{i=1}\EE\left[\frac{ \Q_{-s}\ba_i\ba_i^\top\widetilde\H^{-1}}{1+\ba_i^\top \widetilde\H^{-1}\ba_i/m\pi_{i}}\right].
\end{align*}
Using the rank-one perturbation lemma in~\citet[Lemma~2.6]{silverstein1995empirical}, we obtain
\begin{align*}
\EE[ \Q\A^\top \S^\top\S \A \widetilde\H^{-1}] &\simeq \sum^n_{i=1}\EE\left[\frac{ \Q\ba_i\ba_i^\top\widetilde\H^{-1}}{1+\ba_i^\top\widetilde\H^{-1}\ba_i/m\pi_{i}}\right]=\EE\left[\Q\sum^n_{i=1}\frac{ \ba_i\ba_i^\top}{1+\ba_i^\top\widetilde\H^{-1}\ba_i/m\pi_{i}}\widetilde\H^{-1}\right]\\ 
&=\EE\left[\Q\right]\A^\top\D\A\widetilde\H^{-1},\nonumber
\end{align*}
so that $\D=\diag\left\{\frac{m\pi_i}{m\pi_i+\ba_i^\top\widetilde\H^{-1}\ba_i}\right\}^n_{i=1}$.
This leads to the self-consistent equation in \eqref{eq:def_Dii} of \Cref{theo:inverse-bias}.

\subsection{Detailed Proof of \Cref{theo:inverse-bias}}
\label{subsec:detailed_proof_theo:inverse-bias}

As outlined in \Cref{sec:proof_coarse-RS}, the proof of \Cref{theo:inverse-bias} also comes in the following two steps:
\begin{enumerate}
  \item construct a high probability event $\zeta$ as in \eqref{eq:events}; and
  \item conditioned on that event $\zeta$, bound the spectral norm $  \| \I_d-\EE_{\zeta} [\Q]\widetilde\H \|$ using ``leave-one-out'' analysis.
\end{enumerate}
Furthermore, for each $s\in \{1,\ldots, m\}$, there also exists an index $j=j(s)\in\{1,2,3\}$ such that,  conditioned on $  \zeta_j$,  we have $\Q\preceq  \Q_{-s}\preceq6\H^{-1}$.
  
  To complete the proof of  \Cref{theo:inverse-bias}, we first rewrite
  \begin{align*}
  \|\I_d-\widetilde{\H}^{\frac{1}{2}}\EE_{\zeta}[\Q ]\widetilde{\H}^{\frac{1}{2}}\|&=\|\widetilde{\H}^{\frac{1}{2}}(\EE_{\zeta}[\Q ]\widetilde{\H}-\I_d)\widetilde{\H}^{-\frac{1}{2}}\|.
\end{align*}
  
Taking $\gamma_s=1+\frac{1}{m}\x_s^\top\Q_{-s}\x_s$, $s=1,\ldots,m$, and $\tilde D_{s}=\frac{1}{1+\frac{1}{m}\x_s^\top\tilde\H^{-1}\x_s}$, we then get
\begin{align}
\EE_{\zeta}[\Q ]\widetilde{\H}-\I_d&=(\EE_{\zeta}[\Q ]- \widetilde{\H}^{-1})\widetilde{\H}=\EE_{\zeta}\left[\Q\left(\A^\top\D\A- \A^\top  \S^\top\S\A\right)\widetilde{\H}^{-1}\right] \widetilde{\H} \nonumber\\
&=\EE_{\zeta}\left[\Q\left( \A^\top\D\A- \A^\top  \S^\top\S\A\right)\right] \nonumber\\
&=\underbrace{\EE_{\zeta}[\Q-\Q_{-s}]\A^\top\D\A}_{\Z_1}+\underbrace{\EE_{\zeta}[\Q_{-s}(\A^\top\D\A-\tilde D_{s}\x_s\x^\top_s)]}_{\Z_2} + \underbrace{\EE_{\zeta}[\Q_{-s}(\tilde D_{s}-\frac{1}{\gamma_s})\x_s\x^\top_s]}_{\Z_3},
\nonumber 
\end{align}
which yields 
\begin{align}\label{eq:bound_rew}
  \|\I_d-\widetilde{\H}^{\frac{1}{2}}\EE_{\zeta}[\Q ]\widetilde{\H}^{\frac{1}{2}}\|&=\|\widetilde{\H}^{\frac{1}{2}}(\Z_1+\Z_2+\Z_3)\widetilde{\H}^{-\frac{1}{2}}\|  \nonumber\\
  &\leq \|\widetilde{\H}^{\frac{1}{2}}\Z_1\widetilde{\H}^{-\frac{1}{2}}\|  +\|\widetilde{\H}^{\frac{1}{2}}\Z_2\widetilde{\H}^{-\frac{1}{2}}\|+\|\widetilde{\H}^{\frac{1}{2}}\Z_3\widetilde{\H}^{-\frac{1}{2}}\|. 
\end{align}
Then, we bound the first term  $\|\widetilde{\H}^{1/2}\Z_1\widetilde{\H}^{-1/2}\|$. Together with  the fact that the event $\zeta^{'}$ is independent of $\x_{s}$, and  using the Sherman-Morrison formula in \Cref{lemm:sherman-morrison} and \eqref{eq:condition_expec_three},  we get 
\begin{align}
    \EE_{\zeta}[\Q_{-s}-\Q]&=   \EE_{\zeta}[\frac{1}{m\gamma_s}\Q_{-s}\x_s\x^\top_s\Q_{-s}] \preceq  2\EE_{\zeta^{'}}[\frac{1}{m\gamma_s}\Q_{-s}\x_s\x^\top_s\Q_{-s}]  \nonumber\\
    &=\frac{2}{m}\EE_{\zeta^{'}}[\Q_{-s}\x_s\x^\top_s\Q_{-s}]  \preceq \ \frac{2}{m}\EE_{\zeta^{'}}[\Q_{-s}\sum^n_{j=1}\ba_j\ba_j^\top\Q_{-s}]\nonumber\\
    &=\frac{2}{m}\EE_{\zeta^{'}}[\Q_{-s}\A^\top\A\Q_{-s}]\preceq \frac{2}{m}\EE_{\zeta^{'}}[\Q_{-s}\H\Q_{-s}],\nonumber 
\end{align}
which, by incorporating  the fact that  $\widetilde{\H}\preceq \H$ and  the event $\zeta$ indicates $\H^{1/2}\Q_{-s}\H^{1/2}\preceq 6\I_d$,   leads to 
\begin{align}
   \|\widetilde{\H}^{\frac{1}{2}}\Z_1\widetilde{\H}^{-\frac{1}{2}}\|&\leq \|\widetilde{\H}^{\frac{1}{2}}\EE_{\zeta}[\Q-\Q_{-s}]\widetilde{\H}^{\frac{1}{2}}\widetilde{\H}^{-\frac{1}{2}}\A^\top\D\A\widetilde{\H}^{-\frac{1}{2}}\| \nonumber\\
   &\leq \|\widetilde{\H}^{\frac{1}{2}}\H^{-\frac{1}{2}}\|\|\H^{\frac{1}{2}}\EE_{\zeta}[\Q-\Q_{-s}]\H^{\frac{1}{2}}\| \| \H^{-\frac{1}{2}}\widetilde{\H}^{\frac{1}{2}}\|
  \|\widetilde{\H}^{-\frac{1}{2}}\A^\top\D\A\widetilde{\H}^{-\frac{1}{2}}\| \nonumber\\
   &\leq \|\widetilde{\H}^{\frac{1}{2}}\H^{-\frac{1}{2}}\|^2 \| \H^{\frac{1}{2}}\EE_{\zeta}[\Q-\Q_{-s}] \H^{\frac{1}{2}}\|\leq  \frac{2}{m}\|\EE_{\zeta^{'}}[ \H^{\frac{1}{2}}\Q_{-s} \H^{\frac{1}{2}} \H^{\frac{1}{2}}\Q_{-s} \H^{\frac{1}{2}}]\|\nonumber\\
      & \leq\frac{72}{m}. \label{eq:z_1}
\end{align}

Further, we move on to  bound the second  term in \eqref{eq:bound_rew}.  Recalling the  assumption   regarding  $\zeta_1$, $\zeta_2$, $\zeta^{'}$ and $\delta_3$   from \Cref{sec:proof_coarse-RS}, we   have 
 \begin{align}
      \EE_{\zeta}[\Q_{-s}(\A^\top\D\A-\tilde D_{s}\x_s\x^\top_s)]&=\frac{1}{1-\delta_3}( \EE_{\zeta'}[\Q_{-s}(\A^\top\D\A-\tilde D_{s}\x_s\x^\top_s)]-\EE_{\zeta'}[\Q_{-s}(\A^\top\D\A-\tilde D_{s}\x_s\x^\top_s)\cdot\mathbf{1}_{\neg\zeta_3}])\nonumber\\
      &=-\frac{1}{1-\delta_3}\EE_{\zeta'}[\Q_{-s}(\A^\top\D\A-\tilde D_{s}\x_s\x^\top_s)\cdot\mathbf{1}_{\neg\zeta_3}].\nonumber
 \end{align}
Then, it follows that
 \begin{align}
       \|\widetilde{\H}^{\frac{1}{2}}\Z_2\widetilde{\H}^{-\frac{1}{2}}\|&\leq 2\|\widetilde\H^{\frac{1}{2}}\EE_{\zeta'}[\Q_{-s}(\A^\top\D\A-\tilde D_{s}\x_s\x^\top_s) \cdot\mathbf{1} \neg\zeta_3] \widetilde\H^{-\frac{1}{2}}\|\nonumber\\
      & \leq 2\|\EE_{\zeta'}[\widetilde\H^{\frac{1}{2}}\Q_{-s}\widetilde\H^{\frac{1}{2}}\widetilde\H^{-\frac{1}{2}}(\A^\top\D\A-\tilde D_{s}\x_s\x^\top_s) \cdot\mathbf{1} \neg\zeta_3] \widetilde\H^{-\frac{1}{2}}\|\nonumber\\
       &\leq 12\|\widetilde\H^{\frac{1}{2}}\H^{-1}\widetilde\H^{\frac{1}{2}}\|\EE_{\zeta'}[\|\widetilde\H^{-\frac{1}{2}}(\A^\top\D\A-\tilde D_{s}\x_s\x^\top_s) \cdot\mathbf{1} \neg\zeta_3\widetilde\H^{-\frac{1}{2}}\|] \nonumber\\
      & \leq 12\EE_{\zeta'}[\|\widetilde\H^{-\frac{1}{2}}(\A^\top\D\A-\tilde D_{s}\x_s\x^\top_s) \cdot\mathbf{1} \neg\zeta_3 \widetilde\H^{-\frac{1}{2}}\|]\nonumber\\
      & \leq 12\EE_{\zeta'}[(1+\|\widetilde\H^{-\frac{1}{2}}\tilde D_{s}\x_s\x^\top_s\widetilde\H^{-\frac{1}{2}}\|)\cdot\mathbf{1} \neg\zeta_3 ]\nonumber\\
      & \leq 12\EE_{\zeta'}[(1+\tilde D_{s}\x_s^\top\widetilde\H^{-1}\x_s)\cdot\mathbf{1} \neg\zeta_3 ]=12\delta_3+12\EE_{\zeta'}[\tilde D_{s}\x_s^\top\widetilde\H^{-1}\x_s\cdot\mathbf{1} \neg\zeta_3 ].\nonumber
 \end{align}
It follows from \Cref{lem:range_D} and $m\geq 2\rho_{\max} d_{\eff}$ that
\begin{align}\label{eq:h_wilde_h_h}
   \|\H^{\frac{1}{2}}\widetilde{\H}^{-1}\H^{\frac{1}{2}}\|&\leq \frac{m+2\rho_{\max} d_{\eff}}{m}\|\H^{\frac{1}{2}}\H^{-1}\H^{\frac{1}{2}}\|=\frac{m+2\rho_{\max} d_{\eff}}{m} \leq 2,
\end{align}
so that
\begin{align}
 \Var_{\zeta'}[\tilde D_{s}\x_s^\top\widetilde\H^{-1}\x_s] &\leq \EE_{\zeta'}[(\tilde D_{s}\x_s^\top\widetilde\H^{-1}\x_s)^2]=\sum^n_{j=1}\pi_j\left(\frac{m\ba_j^\top\widetilde\H^{-1}\ba_j}{m\pi_j+\ba_j^\top\widetilde\H^{-1}\ba_j}\right)^2\nonumber\\
& =\sum^n_{j=1}\frac{D^2_{jj}(\ba_j^\top\widetilde\H^{-1}\ba_j)^2}{\pi_j}\leq\sum^n_{j=1}\frac{\|\H^{\frac{1}{2}}\widetilde{\H}^{-1}\H^{\frac{1}{2}}\|^2(\ba_j^\top\H^{-1}\ba_j)^2}{\pi_j}\nonumber\\
&\leq4\rho_{\max} d_{\eff}\sum^n_{j=1}\ba_j^\top\H^{-1}\ba_j\leq4\rho_{\max} d^2_{\eff},\nonumber
\end{align}
and 
\begin{align}
     \EE_{\zeta'}[\tilde D_{s}\x_s^\top\widetilde\H^{-1}\x_s]&=\sum^n_{j=1}\frac{m\pi_j}{m\pi_j+\ba_j^\top\widetilde\H^{-1}\ba_j}\ba_j^\top\widetilde\H^{-1}\ba_j\leq  \sum^n_{j=1}\|\H^{\frac{1}{2}}\widetilde{\H}^{-1}\H^{\frac{1}{2}}\|\ba_j^\top\H^{-1}\ba_j \leq 2d_{\eff}.\nonumber
 \end{align}
Using Chebyshev's inequality, we get, for $x> 2d_{\eff}$,
\begin{align}
    \Pr(\tilde D_{s}\x_s^\top\widetilde\H^{-1}\x_s\geq x~|~\zeta')\leq\frac{4\rho_{\max} d^2_{\eff}}{x^2}. \nonumber
\end{align}
This, together with  $\delta_3\leq \delta\leq \frac{1}{m^3}$ and $m>\rho_{\max} d_{\eff}$ further leads to
\begin{align*}
   \|\widetilde{\H}^{\frac{1}{2}}\Z_2\widetilde{\H}^{-\frac{1}{2}}\|&\leq O\Big(\frac{1}{m^3}\Big)+  12\int^\infty_0 \Pr(\tilde D_{s}\x_s^\top\widetilde\H^{-1}\x_s\cdot \mathbf{1}_{\neg\zeta_3}\geq x~|~\zeta')dx \\
    &\leq O\Big(\frac{1}{m^3}\Big)+24m^2\delta_3+12\int^\infty_{2m^2}\Pr(\tilde D_{s}\x_s^\top\widetilde\H^{-1}\x_s\geq x~|~\zeta')dx \\
    &\leq O\Big(\frac{1}{m^3}\Big)+\frac{24}{m}+ 48\rho_{\max} d^2_{\eff}\int^\infty_{2m^2}\frac{1}{x^2}dx\leq O\Big(\frac{1}{m}\Big)+\frac{24\rho_{\max} d^2_{\eff}}{m^2} \\ 
    &= O\left(\sqrt{\frac{\rho_{\max}^3 d^3_{\eff}}{m^3}}\right).
\end{align*}

Using \eqref{eq:condition_expec_three} and \eqref{eq:h_wilde_h_h}, we bound the last term in \eqref{eq:bound_rew} as follows:
\begin{align}
    \|\widetilde{\H}^{\frac{1}{2}}\Z_3\widetilde{\H}^{-\frac{1}{2}}\|&\leq   \sup\limits_{\|\uu\|=1,\|\vv\|=1}
     \EE_{\zeta} \left[|\tilde D_{s}-\gamma_s^{-1}||\uu^\top\widetilde\H^{\frac{1}{2}} \Q_{-s}\x_s\x^\top_s\widetilde\H^{-\frac{1}{2}}\vv | \right]\nonumber\\
     &\leq   \sup\limits_{\|\uu\|=1,\|\vv\|=1}
     2\EE_{\zeta^{'}} \left[|\tilde D_{s}-\gamma_s^{-1}||\uu^\top\widetilde\H^{\frac{1}{2}} \Q_{-s}\x_s\x^\top_s\widetilde\H^{-\frac{1}{2}}\vv |  \right]\nonumber\\
      &=   \sup\limits_{\|\uu\|=1,\|\vv\|=1}
     2\EE_{\zeta^{'}} \left[\left|\frac{\frac{1}{m}\x_s^\top\Q_{-s}\x_s-\frac{1}{m}\x_s^\top\widetilde\H^{-1}\x_s}{\tilde D^{-1}_{s}\gamma_s}\right|\left|\uu^\top\widetilde\H^{\frac{1}{2}} \Q_{-s}\x_s\x^\top_s\widetilde\H^{-\frac{1}{2}}\vv\right| \right]\nonumber\\
     &\leq   \sup\limits_{\|\uu\|=1,\|\vv\|=1}
    \EE_{\zeta^{'}} \left[\left|\frac{\ba_{i_s}^\top\Q_{-s}\ba_{i_s}-\ba_{i_s}^\top\widetilde\H^{-1}\ba_{i_s}}{m\pi_{i_s}}\right| (\uu^\top\widetilde\H^{\frac{1}{2}} \Q_{-s}\x_s\x_s\Q_{-s}\widetilde\H^{\frac{1}{2}}\uu+ \vv^\top \widetilde\H^{-\frac{1}{2}}\x_s\x^\top_s \widetilde\H^{-\frac{1}{2}}\vv)
     \right]\nonumber\\
    &=   \sup\limits_{\|\uu\|=1}
     \EE_{\zeta^{'}} \left[\sum^n_{j=1}\left|\frac{\ba_{j}^\top\Q_{-s}\ba_{j}-\ba_{j}^\top\widetilde\H^{-1}\ba_{j}}{m\pi_{j}}\right|\uu^\top\widetilde\H^{\frac{1}{2}}  \Q_{-s}\ba_{j}\ba_{j}^\top \Q_{-s}\widetilde\H^{\frac{1}{2}}\uu
     \right]\nonumber\\  
    &+ \sup\limits_{\|\vv\|=1}
     \EE_{\zeta^{'}} \left[\sum^n_{j=1}\left|\frac{\ba_{j}^\top\Q_{-s}\ba_{j}-\ba_{j}^\top\widetilde\H^{-1}\ba_{j}}{m\pi_{j}}\right|\vv^\top \widetilde\H^{-\frac{1}{2}}\ba_{j}\ba_{j}^\top \widetilde\H^{-\frac{1}{2}}\vv
     \right]\nonumber\\  
     &\leq   \sup\limits_{\|\uu\|=1}
     \EE_{\zeta^{'}} \left[\max\limits_{1\leq j\leq n}\left|\frac{\ba_{j}^\top\Q_{-s}\ba_{j}-\ba_{j}^\top\widetilde\H^{-1}\ba_{j}}{m\pi_{j}}\right|\sum^n_{j=1}\uu^\top\widetilde\H^{\frac{1}{2}}  \Q_{-s}\ba_{j}\ba_{j}^\top \Q_{-s}\widetilde\H^{\frac{1}{2}}\uu
     \right]\nonumber\\  
    &+ \sup\limits_{\|\vv\|=1}
     \EE_{\zeta^{'}} \left[\max\limits_{1\leq j\leq n}\left|\frac{\ba_{j}^\top\Q_{-s}\ba_{j}-\ba_{j}^\top\widetilde\H^{-1}\ba_{j}}{m\pi_{j}}\right|\sum^n_{j=1}\vv^\top \widetilde\H^{-\frac{1}{2}}\ba_{j}\ba_{j}^\top \widetilde\H^{-\frac{1}{2}}\vv
     \right]\nonumber\\  
      &\leq   \sup\limits_{\|\uu\|=1}
     \EE_{\zeta^{'}} \left[\max\limits_{1\leq j\leq n}\left|\frac{\ba_{j}^\top\Q_{-s}\ba_{j}-\ba_{j}^\top\widetilde\H^{-1}\ba_{j}}{m\pi_{j}}\right|\uu^\top\widetilde\H^{\frac{1}{2}}  \Q_{-s}\A^\top \A\Q_{-s}\widetilde\H^{\frac{1}{2}}\uu
     \right]\nonumber\\  
    &+ \sup\limits_{\|\vv\|=1}
     \EE_{\zeta^{'}} \left[\max\limits_{1\leq j\leq n}\left|\frac{\ba_{j}^\top\Q_{-s}\ba_{j}-\ba_{j}^\top\widetilde\H^{-1}\ba_{j}}{m\pi_{j}}\right|\vv^\top \widetilde\H^{-\frac{1}{2}}\A^\top \A \widetilde\H^{-\frac{1}{2}}\vv
     \right]\nonumber\\ 
           &\leq  
     \EE_{\zeta^{'}} \left[\max\limits_{1\leq j\leq n}\left|\frac{\ba_{j}^\top\Q_{-s}\ba_{j}-\ba_{j}^\top\widetilde\H^{-1}\ba_{j}}{m\pi_{j}}\right|\|\widetilde\H^{\frac{1}{2}}\H^{-\frac{1}{2}}\|^2\|\H^{\frac{1}{2}}  \Q_{-s}\H^{\frac{1}{2}}\|^2\|\H^{-\frac{1}{2}}\A^\top \A\H^{-\frac{1}{2}}\|
     \right]\nonumber\\  
    &+
     \EE_{\zeta^{'}} \left[\max\limits_{1\leq j\leq n}\left|\frac{\ba_{j}^\top\Q_{-s}\ba_{j}-\ba_{j}^\top\widetilde\H^{-1}\ba_{j}}{m\pi_{j}}\right|\|\widetilde\H^{-\frac{1}{2}}\H^{\frac{1}{2}}\|^2\|\H^{-\frac{1}{2}}\A^\top \A\H^{-\frac{1}{2}}\|
     \right]\nonumber\\ 
     &\overset{(a)}{\leq}
    \EE_{\zeta^{'}} \left[ \max\limits_{1\leq j\leq n}38\left|\frac{\ba_{j}^\top\Q_{-s}\ba_{j}-\ba_{j}^\top\widetilde\H^{-1}\ba_{j}}{m\pi_{j}}\right|
     \right]
     \leq \underbrace{ 38\EE_{\zeta^{'}} \left[ \max\limits_{1\leq j\leq n}\left|\frac{\ba_{j}^\top\Q_{-s}\ba_{j}-\EE_{\zeta^{'}}[\ba_{j}^\top\Q_{-s}\ba_{j}]}{m\pi_{j}}\right|
     \right]}_{ M_1} \nonumber\\  
     & +
    \underbrace{\max\limits_{1\leq j\leq n} 38\frac{|\EE_{\zeta^{'}}[\ba_{j}^\top\Q_{-s}\ba_{j}]-\ba_{j}^\top\widetilde\H^{-1}\ba_{j}|}{m\pi_{j}}
     }_{ M_2},\label{eq: z_3}
\end{align}
where in $(a)$ we use the fact that $\A^\top\A\preceq \H$, $\widetilde{\H}\preceq \H$ and  the event $\zeta$ indicates $\H^{1/2}\Q_{-s}\H^{1/2}\preceq 6\I_d$,  along with \eqref{eq:h_wilde_h_h}.
Then, by \Cref{lemma:max_expec}, we bound the term $ M_1$ in \eqref{eq: z_3} as
\begin{align}
  M_1\leq   O\left(\frac{\log n}{\log (\log n)}\sqrt{\frac{\rho_{\max}^3 d^3_{\eff}}{m^3}}\right).\nonumber 
 \end{align}

Now, it remains to bound the second term $M_{2}$ in \eqref{eq: z_3}.  
Using again \eqref{eq:h_wilde_h_h}, we get
\begin{align}
  & | \EE_{\zeta'}[\ba_j^\top\Q_{-s}\ba_j]-\ba_j^\top\widetilde\H^{-1}\ba_j|=|\ba_j^\top \widetilde\H^{-\frac{1}{2}}\widetilde\H^{\frac{1}{2}}\EE_{\zeta'}[\Q_{-s}-\widetilde\H^{-1}]\widetilde\H^{\frac{1}{2}}\widetilde\H^{-\frac{1}{2}}\ba_j|\nonumber\\
   &\leq \ba_j^\top \widetilde\H^{-1}\ba_j\|\widetilde\H^{\frac{1}{2}}(\EE_{\zeta'}[\Q_{-s}]-\widetilde\H^{-1})\widetilde\H^{\frac{1}{2}}\|\leq 2\ba_j^\top \H^{-1}\ba_j\|\widetilde\H^{\frac{1}{2}}(\EE_{\zeta'}[\Q_{-s}]-\widetilde\H^{-1})\widetilde\H^{\frac{1}{2}}\|\nonumber\\
   &\leq 2\ba_j^\top \H^{-1}\ba_j\|\widetilde\H^{\frac{1}{2}}((\EE_{\zeta'}-\EE_{\zeta})[\Q_{-s}]+\EE_{\zeta}[\Q_{-s}-\Q]+\EE_{\zeta}[\Q]-\widetilde\H^{-1})\widetilde\H^{\frac{1}{2}}\|\nonumber\\
     &\leq 2\ba_j^\top \H^{-1}\ba_j\left(\underbrace{\|\widetilde\H^{\frac{1}{2}}(\EE_{\zeta'}-\EE_{\zeta})[\Q_{-s}]\widetilde\H^{\frac{1}{2}}\|}_{G_1}+\underbrace{\|\widetilde\H^{\frac{1}{2}}\EE_{\zeta}[\Q_{-s}-\Q]\widetilde\H^{\frac{1}{2}}\|}_{G_2}+\|\widetilde\H^{\frac{1}{2}}(\EE_{\zeta}[\Q]\widetilde\H-\I_d)\widetilde\H^{-\frac{1}{2}}\|\right). \label{eq: M_2}
\end{align}
Noting $\delta_3<\frac{1}{m^3}$ and   
\begin{align}
   \EE_{\zeta}[\Q_{-s} ]=
   \frac{1}{1-\delta_3}(\EE_{\zeta^{'}}[\Q_{-s} ]-\delta_3\EE_{\zeta^{'}}[ \Q_{-s}|\neg\delta_3 ]), \nonumber
\end{align}
we have, for the term $G_1$ defined in \eqref{eq: M_2}, that
\begin{align}
    G_1&=\frac{\delta_3}{1-\delta_3}\|\widetilde\H^{\frac{1}{2}}(\EE_{\zeta^{'}}[\Q_{-s} ]-\EE_{\zeta^{'}}[ \Q_{-s}|\neg\delta_3 ])\widetilde\H^{\frac{1}{2}}\|\nonumber\\
    &\leq 2\delta_3(\|\widetilde\H^{\frac{1}{2}}\EE_{\zeta^{'}}[\Q_{-s}] \widetilde\H^{\frac{1}{2}}\|+\|\widetilde\H^{\frac{1}{2}}\EE_{\zeta^{'}}[ \Q_{-s}|\neg\delta_3 ]\widetilde\H^{\frac{1}{2}}\|)\leq 24\delta_3\leq\frac{24}{m^3}.\nonumber
\end{align}
For the term $G_2$ in \eqref{eq: M_2}, it follows from \eqref{eq:z_1} that $G_2 = O(\frac{1}{m})$. 
We thus have
\begin{align}
 &M_2\leq 76 \max\limits_{1\leq j\leq n}\frac{\ba_j^\top \H^{-1}\ba_j}{m\pi_{j}}\left(O\left(\frac{1}{m^3}\right)+ O\left(\frac{1}{m}\right)+ \|\widetilde\H^{\frac{1}{2}}\Z_1\widetilde\H^{-\frac{1}{2}}\|+ \|\widetilde\H^{\frac{1}{2}}\Z_2\widetilde\H^{-\frac{1}{2}}\|+ \|\widetilde\H^{\frac{1}{2}}\Z_3\widetilde\H^{-\frac{1}{2}}\|\right)\nonumber\\
&\leq  \frac{76\rho_{\max} d_{\eff}}{m}\left(O\left(\frac{1}{m^3}\right)+ O\left(\frac{1}{m}\right)+ \|\widetilde\H^{\frac{1}{2}}\Z_1\widetilde\H^{-\frac{1}{2}}\|+ \|\widetilde\H^{\frac{1}{2}}\Z_2\widetilde\H^{-\frac{1}{2}}\|+ \|\widetilde\H^{\frac{1}{2}}\Z_3\widetilde\H^{-\frac{1}{2}}\|\right)\nonumber\\
&\leq \frac{76\rho_{\max} d_{\eff}}{m}\left(O\left(\frac{1}{m^3}\right))+ O\left(\frac{1}{m}\right)+O\left(\sqrt{\frac{\rho_{\max}^3 d^3_{\eff}}{m^3}}\right)
+O\left(\frac{\log n}{\log (\log n)}\sqrt{\frac{\rho_{\max}^3 d^3_{\eff}}{m^3}}\right)+ M_2\right),\nonumber
\end{align}
which, for $m>76\rho_{\max} d_{\eff}$, yields that
\begin{align}
 M_2 = O\left(\frac{\log n}{\log (\log n)}\sqrt{\frac{\rho_{\max}^3 d^3_{\eff}}{m^3}\frac{76^2\rho_{\max}^2 d^2_{\eff}}{(m-76 \rho_{\max} d_{\eff})^2}}\right)=O \left(\frac{\log n}{\log (\log n)}\sqrt{\frac{\rho_{\max}^3 d^3_{\eff}} {m^3}} \right).\nonumber 
\end{align}
Putting everything together, we conclude that
\begin{align*}
   \Big \|\widetilde\H^{\frac{1}{2}}\EE_{\zeta}[\Q ]\widetilde\H^{\frac{1}{2}}-\I_d\Big\|=O \left(\frac{\log n}{\log (\log n)}\sqrt{\frac{\rho_{\max}^3 d^3_{\eff}} {m^3}} \right).
\end{align*}
This concludes the proof of \Cref{theo:inverse-bias}.
\qedwhite

\subsection{Discussions on \Cref{theo:inverse-bias} and Auxiliary Results}
\label{subsec:discussion_theo:inverse-bias}

\begin{lemma}\label{lemma:max_expec}
Under the settings and notations of \Cref{theo:inverse-bias},  then we have 
\begin{align*}
\EE_{\zeta^{'}} \left[ \max\limits_{1\leq j\leq n}\left|\frac{\ba_{j}^\top\Q_{-s}\ba_{j}-\EE_{\zeta^{'}}[\ba_{j}^\top\Q_{-s}\ba_{j}]}{m\pi_{j}}\right|
     \right]= O\left(\frac{\log n}{\log (\log n)} \sqrt{\frac{\rho_{\max}^{3}d^{3}_{\eff}}{m^{3}}} \right).
\end{align*}
\end{lemma}

\begin{proof}

 Using $\max\limits_{1\leq i\leq n}z_i\leq \left(\sum^n_{i=1} z_i^s\right)^{1/s}$ with any variable $z_i>0$,   along with  Jensen’s inequality, we get, for $q\geq 2$
\begin{align*}
& \EE_{\zeta^{'}}\left[\max\limits_{1\leq j\leq n}\frac{|\ba_j^\top\Q_{-s}\ba_j-\EE_{\zeta^{'}}[\ba_j^\top\Q_{-s}\ba_j]|}{m\pi_j}\right]\leq  \EE_{\zeta^{'}}\left[\left(\sum^n_{j=1}\left|\frac{ \ba_j^\top\Q_{-s}\ba_j-\EE_{\zeta^{'}}[\ba_j^\top\Q_{-s}\ba_j]}{m\pi_j}\right|^q\right)^{1/q}\right]\\
 &\leq \left(\sum^n_{j=1}  \EE_{\zeta^{'}}\left[\left|\frac{ \ba_j^\top\Q_{-s}\ba_j-\EE_{\zeta^{'}}[\ba_j^\top\Q_{-s}\ba_j]}{m\pi_j}\right|^q\right]\right)^{1/q}.
\end{align*}
Now, it suffices to bound each term$\EE_{\zeta^{'}}\left[\left|\frac{ \ba_j^\top\Q_{-s}\ba_j-\EE_{\zeta^{'}}[\ba_j^\top\Q_{-s}\ba_j]}{m\pi_j}\right|^q\right]$.
    Use  $\Q_{-sl}$ to denote the matrix $(\A^\top \mathbf{S}_{-sl}^\top\mathbf{S}_{-sl}\A+\C)^{-1}$ where $\mathbf{S}_{-sl}$ is the matrix $\mathbf{S}$ without the $s^{th}$ and $l^{th}$ rows; let that, for each pair $s$, $l$, one of $\zeta_1$, $\zeta_2$ is independent of both $\x_{s}$ and $\x_{l}$. Without loss of generality, we assume  that it is $\zeta_1$.    
Also, let $\EE_{\zeta^{'},l}[\cdot]$ be the expectation conditioned on $\zeta^{'}$ and  the $\sigma$-field  $\mathcal F_l$ generated by the basic events  $\x_{1},\ldots, \x_{l}.$ 
To apply the Martingale Rosenthal  inequality in \Cref{lemm:martingale_rosenthal_sharp}, we rewrite $ \ba_j^\top\Q_{-s}\ba_j-\EE_{\zeta^{'}}[\ba_j^\top\Q_{-s}\ba_j] $ as the martingale difference sequence
\begin{align}
& \frac{ \ba_{j}^\top\Q_{-s}\ba_{j}-\EE_{\zeta^{'}}[\ba_{j}^\top\Q_{-s}\ba_{j}]  }{m\pi_j}     
    = \frac{\EE_{\zeta^{'},m} [\ba_{j}^\top\Q_{-s}\ba_{j}]   -\EE_{\zeta^{'},0} [\ba_{j}^\top\Q_{-s}\ba_{j}]  }{m\pi_j}=\sum^m_{l=1}(Z_l-Z_{l-1})\nonumber\\
    &=\sum^m_{l=1}\frac{ ((\EE_{\zeta^{'},l}-\EE_{\zeta^{'},l-1} )[\ba_{j}^\top\Q_{-s}\ba_{j}-\ba_{j}^\top\Q_{-sl}\ba_{j}]  }{m\pi_j}    +  \frac{(\EE_{\zeta^{'},l}-\EE_{\zeta^{'},l-1} )[\ba_{j}^\top\Q_{-sl}\ba_{j}] )}{m\pi_j}
    \nonumber\\
    &=- \sum^m_{l=1}(\phi_l+\varphi_l), \nonumber
\end{align}
where
\begin{align*}
   \phi_l&= \frac{(\EE_{\zeta^{'},l}-\EE_{\zeta^{'},l-1})[\ba_{j}^\top\Q_{-sl}\ba_{j}-\ba_{j}^\top\Q_{-s}\ba_{j}] }{m\pi_j}\\
  \varphi_l&= -\frac{(\EE_{\zeta^{'},l}-\EE_{\zeta^{'},l-1}) [\ba_{j}^\top\Q_{-sl}\ba_{j}]}{m\pi_j}.
\end{align*}
Following \Cref{lemm:martingale_rosenthal_sharp},  we get
\begin{align*}
  &  \EE_{\zeta^{'}}\left[\left| \frac{\ba_{j}^\top\Q_{-s}\ba_{j}-\EE_{\zeta^{'}}[\ba_{j}^\top\Q_{-s}\ba_{j}]}{m\pi_j}\right|^q\right]\leq     \EE_{\zeta^{'}}\left[\left| - \sum^m_{l=1}(\phi_l+\varphi_l)\right|^q\right]\\
    &\leq K_q^q \left(\frac{q}{\log q}\right)^{q}
\left(
\underbrace{\EE_{\zeta^{'}}\left[\left(
\sum_{l=1}^m \EE_{\zeta^{'},l-1}[(\phi_l+\varphi_l)^2]
\right)^{q/2}\right]}_{R_1}
+
\underbrace{\sum_{l=1}^m \EE_{\zeta^{'}}[|\phi_l+\varphi_l|^{q}]}_{R_2}
\right).
\end{align*}
In the following, we first bound the term $R_2$. Using $(a+b)^s\leq 2^{s-1}(a^s+b^s)$ with  $a>0,~b>0$, gives
\begin{align*}
 \EE_{\zeta^{'}}[|\phi_l+\varphi_l|^{q}]\leq   2^{q-1} \EE_{\zeta^{'}}\left[|\phi_l|^{q}\right]+2^{q-1}\EE_{\zeta^{'}}\left[|\varphi_l|^{q}\right]
\end{align*}
which leads to
\begin{align*}
    R_2\leq   2^{q-1} \sum^m_{l=1}\EE_{\zeta^{'}}\left[|\phi_l|^{q}\right]+2^{q-1} \sum^m_{l=1}\EE_{\zeta^{'}}\left[|\varphi_l|^{q}\right].
\end{align*}
Next, we bound the term   $ \sum^m_{l=1}\EE_{\zeta^{'}}\left[|\phi_l|^{q}\right]$.
 We rewrite 
\begin{align}
    \ba_{j}^\top\Q_{-sl}\ba_{j}-\ba_{j}^\top\Q_{-s}\ba_{j}=\frac{1}{m}\frac{\ba_{j}^\top\Q_{-sl}\x_l\x_l^\top\Q_{-sl}\ba_{j}}{1+\frac{1}{m}\x_l^\top\Q_{-sl}\x_l}\leq \frac{1}{m}\ba_{j}^\top\Q_{-sl}\x_l\x_l^\top\Q_{-sl}\ba_{j}.\label{eq:q_sl_q_s_quadratic} 
\end{align}
\begin{lemma}\label{lemm:expec_psd_s}
Under the settings and notations of \Cref{theo:inverse-bias}, we have, for a p.s.d.\@ random matrix $\M$ (or a non-negative random variable) living in the probability space of $\S$, 
\begin{align}
    \EE_{\zeta'}[\M]=\frac{\EE[(\prod^2_{j=1}\mathbf{1}_{\zeta_j})\M]}{\Pr(\zeta')} \preceq \frac{1}{1-\delta}  \EE[\mathbf{1}_{\zeta_{1}}\M] \preceq 2 \EE_{\zeta_1}[\M], \label{eq:condition_expec_two}
\end{align}
where $\mathbf{1}_{\zeta_j}$ is the indicator of the event $\zeta_j$.
\end{lemma}
Then, together with the  Jensen’s inequality,  we have, for $q\geq 2$, 
\begin{align*}
 & \EE_{\zeta^{'}}[|\phi_l|^{q}]\leq 2^{q-1}\EE_{\zeta^{'}}\left[\left(\frac{\EE_{\zeta^{'},l}[\ba_{j}^\top\Q_{-sl}\x_l\x_l^\top\Q_{-sl}\ba_{j}] }{m^2\pi_j}\right)^q\right]\\
 &+2^{q-1} \EE_{\zeta^{'}}\left[\left(\frac{\EE_{\zeta^{'},l-1}[\ba_{j}^\top\Q_{-sl}\x_l\x_l^\top\Q_{-sl}\ba_{j}] }{m^2\pi_j}\right)^q\right]\\
 &\leq 2^{q-1}\EE_{\zeta^{'}}\left[\left(\frac{\ba_{j}^\top\Q_{-sl}\x_l\x_l^\top\Q_{-sl}\ba_{j} }{m^2\pi_j}\right)^q\right]+ 2^{q-1}\EE_{\zeta^{'}}\left[\left(\frac{\ba_{j}^\top\Q_{-sl}\x_l\x_l^\top\Q_{-sl}\ba_{j}}{m^2\pi_j}\right)^q\right]\\
 &=2^{q}\EE_{\zeta^{'}}\left[\left(\frac{\ba_{j}^\top\Q_{-sl}\x_l\x_l^\top\Q_{-sl}\ba_{j} }{m^2\pi_j}\right)^q\right]\overset{(a)}{\leq } 2^{q+1}\EE_{\zeta_1}\left[\left(\frac{\ba_{j}^\top\Q_{-sl}\x_l\x_l^\top\Q_{-sl}\ba_{j} }{m^2\pi_j}\right)^q\right],
\end{align*}
where in $(a)$ we apply \Cref{lemm:expec_psd_s}.
Combining 
\begin{align*}
 \frac{\ba_{j}^\top\Q_{-sl}\ba_i\ba_i^\top\Q_{-sl}\ba_{j} }{m^2\pi_j\pi_i}\leq \frac{6^2\rho_{\max}^2d^2_{\eff}}{m^2},
\end{align*}
and $\Q_{-sl}\preceq 6\H^{-1}$, 
it follows that 
\begin{align}\label{eq:q_th_fini_momen}
&\EE_{\zeta_1}\left[\left(\frac{\ba_{j}^\top\Q_{-sl}\x_l\x_l^\top\Q_{-sl}\ba_{j} }{m^2\pi_j}\right)^{q}\right]=\sum^n_{i=1}\EE_{\zeta_1}\left[\pi_i\left(\frac{\ba_{j}^\top\Q_{-sl}\ba_i\ba_i^\top\Q_{-sl}\ba_{j} }{m^2\pi_j\pi_i}\right)^{q}\right]\nonumber\\
&\leq \frac{(6\rho_{\max} d_{\eff})^{2q-2}}{m^{2q-2}}\sum^n_{i=1}\EE_{\zeta_1}\left[\frac{\ba_{j}^\top\Q_{-sl}\ba_i\ba_i^\top\Q_{-sl}\ba_{j} }{m^2\pi_j}\right]\nonumber\\
&\leq \frac{(6\rho_{\max} d_{\eff})^{2q-2}}{m^{2q-2}}\EE_{\zeta_1}\left[\frac{\ba_{j}^\top\Q_{-sl}\A^\top\A\Q_{-sl}\ba_{j} }{m^2\pi_j}\right]\leq  \frac{36(6\rho_{\max} d_{\eff})^{2q-2}}{m^{2q-2}}\cdot \frac{\ba_j^\top\H^{-1}\ba_j}{m^2\pi_j}\nonumber\\
&\leq  \frac{6^{2q}(\rho_{\max} d_{\eff})^{2q-1}}{m^{2q}}.
\end{align}
This implies 
\begin{align*}
  \sum_{l=1}^m \EE_{\zeta^{'}}[|\phi_l|^{q}]\leq 2^{q+1}\cdot  \frac{6^{2q}(\rho_{\max} d_{\eff})^{2q-1}}{m^{2q-1}}=  \frac{2^{3q+1}\cdot 3^{2q}(\rho_{\max} d_{\eff})^{2q-1}}{m^{2q-1}}.
\end{align*}

Subsequently, we proceed to bound the term  $|\varphi_l|$. 
Considering that $\zeta_1$ is independent of $\x_{l}$, it follows that $\EE_{\zeta_1,l}[\ba_j^\top\Q_{-sl}\ba_j]=\EE_{\zeta_1,l-1}[\ba_j^\top\Q_{-sl}\ba_j]$. 
Then, noting that, for $\delta_2=\Pr(\neg\zeta_2)<\frac{1}{m}$,
\begin{align}
    \EE_{\zeta^{'},l}[\ba_j^\top\Q_{-sl}\ba_j]&=\frac{1}{1-\delta_2}   \EE_{\zeta_1,l}[\ba_j^\top\Q_{-sl}\ba_j]
   -\frac{\delta_2}{1-\delta_2}\EE_{\zeta_1,l}[\ba_j^\top\Q_{-sl}\ba_j|\neg\zeta_2], \nonumber 
\end{align}
and
\begin{align}
   \EE_{\zeta^{'},l-1}[\ba_j^\top\Q_{-sl}\ba_j]
    &  =\frac{1}{1-\delta_2}   \EE_{\zeta_1,l-1}[\ba_j^\top\Q_{-sl}\ba_j]
   -\frac{\delta_2}{1-\delta_2}\EE_{\zeta_1,l-1}[\ba_j^\top\Q_{-sl}\ba_j|\neg\zeta_2], \nonumber 
\end{align}
so that 
\begin{align}\label{eq:varphi_l}
    |\varphi_l|&=\frac{|(\EE_{\zeta^{'},l}-\EE_{\zeta^{'},l-1})[\ba_j^\top\Q_{-sl}\ba_j]
    |}{m\pi_j}\leq \frac{\delta_2}{1-\delta_2}\cdot \frac{|(\EE_{\zeta_1,l}-\EE_{\zeta_1,l-1})[\ba_j^\top\Q_{-sl}\ba_j |\neg\zeta_2]}{m\pi_j}
 | \nonumber\\
    &\leq
    \frac{ 2\delta_2\EE_{\zeta_1}[\ba_j^\top\Q_{-sl}\ba_j|\neg\zeta_2]}{m\pi_j}\leq \frac{12\delta_2\ba_j^\top\H^{-1}\ba_j}{m\pi_j}<\frac{12\rho_{\max}d_{\eff}}{m^2},
\end{align}
which yields 
\begin{align*}
 \sum^m_{l=1}\EE_{\zeta^{'}}\left[|\varphi_l|^{q}\right]\leq    \left(\frac{12\rho_{\max} d_{\eff}}{m^{2-1/q}}\right)^{q}.
\end{align*}
We thus get
\begin{align*}
    R_2&\leq 2^{q-1}\cdot \frac{2^{3q+1}\cdot 3^{2q}(\rho_{\max} d_{\eff})^{2q-1}}{m^{2q-1}}+2^{q-1}\cdot\left(\frac{12\rho_{\max} d_{\eff}}{m^{2-1/q}}\right)^{q}\\
   &\leq  2^{4q}\cdot 3^{2q}\cdot\left( \frac{(\rho_{\max} d_{\eff})^{2-1/q}}{m^{2-1/q}}\right)^{q}+ 2^{3q-1}\cdot 3^{q}\cdot\left(\frac{\rho_{\max} d_{\eff}}{m^{2-1/q}}\right)^{q} \leq  2^{4q+1}\cdot 3^{2q}\left( \frac{(\rho_{\max} d_{\eff})^{2-1/q}}{m^{2-1/q}}\right)^{q}.
\end{align*}

Now, it remains to bound the first term $R_1$. Considering 
\begin{align*}
    \EE_{\zeta^{'},l-1}[(\phi_l+\varphi_l)^2]\leq 2 \EE_{\zeta^{'},l-1}[\phi_l^2]+2 \EE_{\zeta^{'},l-1}[\varphi_l^2],
\end{align*}
and  taking $(a+b)^s\leq 2^{s-1}(a^s+b^s)$ with $a>0,~b>0$ again,
we first rewrite 
\begin{align*}
     R_1&\leq 2^{q/2-1}\cdot 2^{q/2}  \EE_{\zeta^{'}}\left[\left( \sum^m_{l=1} \EE_{\zeta^{'},l-1}[\varphi_l^2]\right)^{q/2}\right]+2^{q/2-1}\cdot 2^{q/2}\EE_{\zeta^{'}}\left[\left( \sum^m_{l=1} \EE_{\zeta^{'},l-1}[\phi_l^2]\right)^{q/2}\right]\\
     &=2^{q-1}\cdot\EE_{\zeta^{'}}\left[\left( \sum^m_{l=1} \EE_{\zeta^{'},l-1}[\varphi_l^2]\right)^{q/2}\right]+2^{q-1}\cdot\EE_{\zeta^{'}}\left[\left( \sum^m_{l=1} \EE_{\zeta^{'},l-1}[\phi_l^2]\right)^{q/2}\right].
\end{align*}
Using \eqref{eq:varphi_l} again, we have 
\begin{align*}
 \EE_{\zeta^{'}}\left[\left( \sum^m_{l=1} \EE_{\zeta^{'},l-1}[\varphi_l^2]\right)^{q/2}\right]\leq    \left(\frac{12^2\rho_{\max}^2d_{\eff}^2}{m^{3}}\right)^{q/2}.
\end{align*}
Recalling \eqref{eq:q_sl_q_s_quadratic} and \eqref{eq:q_th_fini_momen},
 we further  derive 
\begin{align*}
 & \EE_{\zeta^{'},l-1}[\phi_l^2]\leq 2\EE_{\zeta^{'},l-1}\left[\frac{(\EE_{\zeta^{'},l}[\ba_{j}^\top\Q_{-sl}\x_l\x_l^\top\Q_{-sl}\ba_{j}] )^2}{m^4\pi^2_j}\right]+2 \EE_{\zeta^{'},l-1}\left[\frac{(\EE_{\zeta^{'},l-1}[\ba_{j}^\top\Q_{-sl}\x_l\x_l^\top\Q_{-sl}\ba_{j}] )^2}{m^4\pi^2_j}\right]\\
 &\overset{(a)}{\leq} 2\EE_{\zeta^{'},l-1}\left[\frac{(\ba_{j}^\top\Q_{-sl}\x_l\x_l^\top\Q_{-sl}\ba_{j} )^2}{m^4\pi^2_j}\right]+2\EE_{\zeta^{'},l-1}\left[\frac{(\ba_{j}^\top\Q_{-sl}\x_l\x_l^\top\Q_{-sl}\ba_{j} )^2}{m^4\pi^2_j}\right]\\
 &=4\EE_{\zeta^{'},l-1}\left[\frac{(\ba_{j}^\top\Q_{-sl}\x_l\x_l^\top\Q_{-sl}\ba_{j} )^2}{m^4\pi^2_j}\right]\leq  8\EE_{\zeta_1,l-1}\left[\frac{(\ba_{j}^\top\Q_{-sl}\x_l\x_l^\top\Q_{-sl}\ba_{j} )^2}{m^4\pi^2_j}\right]\\
 &=8\sum^n_{i=1}\EE_{\zeta_1,l-1}\left[\frac{(\ba_{j}^\top\Q_{-sl}\ba_i\ba_i^\top\Q_{-sl}\ba_{j} )^2}{m^4\pi^2_j\pi_i}\right]\leq  8\cdot \frac{6^{4}\rho_{\max}^3d_{\eff}^3}{m^{4}}= \frac{2^7\cdot 3^4\rho_{\max}^3d_{\eff}^3}{m^{4}},
\end{align*}
where in $(a)$, we use the   Jensen’s inequality.
We thus attain
\begin{align*}
 \EE_{\zeta^{'}}\left[\left( \sum^m_{l=1} \EE_{\zeta^{'},l-1}[\phi_l^2]\right)^{q/2}\right]\leq    \left( \frac{2^7\cdot 3^4\rho_{\max}^3d_{\eff}^3}{m^{3}}\right)^{q/2}.
\end{align*}
This further results in
\begin{align*}
    R_1
    &\leq 2^{q-1} \cdot \left(\frac{12^2\rho_{\max}^2d^2_{\eff}}{m^{3}}\right)^{q/2}+2^{q-1} \cdot  \left( \frac{2^7\cdot 3^4\rho_{\max}^3d_{\eff}^3}{m^{3}}\right)^{q/2}\\
    &\leq 2\cdot 2^{q-1} \cdot 2^{7q/2}\cdot 3^{2q}  \cdot  \left( \frac{\rho_{\max}^3d_{\eff}^3}{m^{3}}\right)^{q/2}= 2^{9q/2} \cdot 3^{2q}  \cdot \left( \frac{\rho_{\max}^3d_{\eff}^3}{m^{3}}\right)^{q/2}.
\end{align*}
Putting the above together, for $2-1/q\geq 3/2$, we conclude 
that 
\begin{align*}
& \EE_{\zeta^{'}}\left[\left| \frac{\ba_{j}^\top\Q_{-s}\ba_{j}-\EE_{\zeta^{'}}[\ba_{j}^\top\Q_{-s}\ba_{j}]}{m\pi_j}\right|^q\right]\leq K_q^q \left(\frac{q}{\log q}\right)^{q}2^{9q/2}\cdot 3^{2q}  \cdot  \left( \frac{\rho_{\max}^3d_{\eff}^3}{m^{3}}\right)^{q/2}\\
&+ K_q^q \left(\frac{q}{\log q}\right)^{q} 2^{4q+1}\cdot 3^{2q}\cdot\left( \frac{(\rho_{\max} d_{\eff})^{2-1/q}}{m^{2-1/q}}\right)^{q}\leq  K_q^q \left(\frac{q}{\log q}\right)^{q}2^{9q/2+2}\cdot 3^{2q}  \cdot  \left( \frac{\rho_{\max}^3d_{\eff}^3}{m^{3}}\right)^{q/2}.
\end{align*} 
Letting  $q=\log n$ such that $n^{1/q}=O(1)$, for $K_q$ a constant, we further obtain
\begin{align*}
&  \EE_{\zeta^{'}}\left[\max\limits_{1\leq j\leq n}\frac{|\ba_j^\top\Q_{-s}\ba_j-\EE_{\zeta^{'}}[\ba_j^\top\Q_{-s}\ba_j]|}{m\pi_j}\right]\\
  &\leq{\left(n  K_q^q \left(\frac{q}{\log q}\right)^{q}2^{9q/2+2}\cdot 3^{2q}  \cdot \left( \frac{\rho_{\max}^3d_{\eff}^3}{m^{3}}\right)^{q/2}\right)}^{1/q}\\
  &= 9\cdot2^{9/2}K_q \cdot   (  4 n)^{1/q}\cdot \frac{q}{\log q}\cdot  \sqrt{\frac{\rho_{\max}^3d_{\eff}^3}{m^{3}}}=O\left(\frac{q}{\log q} \cdot \sqrt{\frac{\rho_{\max}^3d_{\eff}^3}{m^{3}}} \right).
\end{align*}
We thus complete the proof.
\end{proof}

Below is the proof of the result in \Cref{footnote:D}.
\begin{lemma}[On the self-consistent $\D$]\label{lem:range_D}
For a given matrix $\A\in \RR^{n\times d}$,
let $\S$ be a random sampling  matrix with number of trials $m$ and importance sampling distribution $\{ \pi_i \}_{i=1}^n$ as in \Cref{def:RS}, and let $\C\in \RR^{d\times d}$ be a p.s.d.\@ matrix and $d_{\eff}= \sum_{i=1}^n \ell^{\C}_i$ with  leverage score $ \ell^{\C}_i $ as in \Cref{def:lev}. We have
\begin{align}
    \frac{m}{m+2\rho_{\max} d_{\eff}}\I_n \preceq \D \preceq \frac{m}{m+\rho_{\min}d_{\eff}}\I_n,  \nonumber
\end{align}
in the sense of p.s.d.\@ matrices, where we recall $\rho_{\max}=\max_{1\leq i\leq n} \ell_i^\C/(\pi_id_{\eff})$ and $\rho_{\min}=\min_{1\leq i\leq n} \ell_i^\C/(\pi_id_{\eff})$.
\end{lemma}
\begin{proof}[Proof of \Cref{lem:range_D}]
For each $D_{ii}$, $i=1,\ldots, n$, it follows from its definition in \Cref{theo:inverse-bias} that
 \begin{align}
   D_{ii}=  \frac{m\pi_i}{ m\pi_i + \ba_i^\top (\A^\top \D \A + \C)^{-1} \ba_i }\geq\frac{m\pi_i}{ m\pi_i + D^{-1}_{\min}\ba_i^\top (\A^\top  \A+\C)^{-1} \ba_i },\nonumber
 \end{align}
 where $D_{\min}=\min_{1 \leq i \leq n} D_{ii} \leq 1$. 
 Without loss of generality, let $D_{\min}=D_{nn}$ and use the fact that $m\geq C \rho_{\max} d_{\eff} \geq C \ell_n^\C/\pi_n =  C\ba_n^\top (\A^\top  \A+\C)^{-1} \ba_n /\pi_n$, we obtain a lower bound on $D_{nn}$ as
 \begin{align}
    D_{nn}&=  \frac{m\pi_n}{ m\pi_n + \ba_n^\top (\A^\top \D \A + \C)^{-1} \ba_n }\geq\frac{m\pi_n}{ m\pi_n + D^{-1}_{nn}\ba_n^\top (\A^\top  \A+\C)^{-1} \ba_n}\nonumber\\
     &\geq \frac{m\pi_n}{ m\pi_n + D^{-1}_{nn}C^{-1}m\pi_n}=\frac{C}{C +  D^{-1}_{nn} },\nonumber
 \end{align}
 so that $D_{nn} \geq \frac{C-1}{C} \equiv \Delta > 1/2$, for some $C > 2$. 
 Then, it follows that, for $i = 1, \ldots, n$, 
 \begin{align}
     D_{ii}\geq  \frac{m\pi_i}{ m\pi_i +\Delta^{-1} \ba_i^\top (\A^\top  \A+\C)^{-1} \ba_i }\geq  \frac{m}{ m+\Delta^{-1} \ba_i^\top (\A^\top  \A+\C)^{-1} \ba_i/\pi_i }
     \geq \frac{m}{ m +\Delta^{-1}  \rho_{\max} d_{\eff}}.\nonumber
 \end{align}
 This together with $C>2$ results in
 \begin{align}
     D_{ii}\geq  \frac{m}{ m +2  \rho_{\max} d_{\eff}}.\nonumber
 \end{align}
On the other hand,  we have
\begin{align}
    D_{ii}&=  \frac{m}{ m+ \ba_i^\top (\A^\top \D \A + \C)^{-1} \ba_i/\pi_i }\leq\frac{m}{ m + \ba_i^\top (\A^\top  \A+\C)^{-1} \ba_i /\pi_i }\leq \frac{m}{ m +\rho_{\min} d_{\eff} }. \nonumber
\end{align}
Consequently, we have that $\frac{m}{ m +2  \rho_{\max} d_{\eff}}\I_n\preceq\D\preceq \frac{m}{ m +\rho_{\min} d_{\eff}}\I_n$, this concludes the proof of \Cref{lem:range_D}.
\end{proof}

\section{Proof of \Cref{prop:debias}} 
\label{sec:proof_prop_debias}

In this section, we start by presenting the RMT intuition for \Cref{prop:debias}, and in particular how the \emph{exact} leverage scores $\ell_i^\C$ comes into play in the analysis and debiasing, in \Cref{subsec:RMT_intuition_prop:debias}.
The detailed proof of \Cref{prop:debias} is then given in \Cref{subsec:detailed_proof_prop:debias}; and finally 
we provide discussions and auxiliary results on \Cref{prop:debias} in \Cref{subsec:discussion_prop:debias}.

\subsection{RMT Intuition on \Cref{prop:debias}}
\label{subsec:RMT_intuition_prop:debias}

Here, we present the heuristic derivation of \Cref{prop:debias}.
First, let us recall some notations from \Cref{prop:debias}:
let $\x^\top_s=\ee^\top_{i_s}/\sqrt{\pi_{i_s}}\A$ as in \Cref{sec:proof_coarse-RS}, 
for the ease of further use, let
\begin{align}  
  \check{\S}^\top \check{\S}=\sum^m_{s=1}F_{i_si_s}~\cdot~\frac{\ee_{i_s}\ee^\top_{i_s} }{m\pi_{i_s}},\nonumber
\end{align} 
for some \emph{deterministic} $F_{ii}$ to be specified, and
\begin{equation*}
  \check{\Q}=(\A^\top \check{\S}^\top \check{\S} \A +\C)^{-1}= \left(\frac{1}{m} \sum^{m}_{s=1} F_{i_si_s} \x_s{\x}^{\top}_s + \C \right)^{-1},
\end{equation*}
and similarly $\check{\Q}_{-s}=( \frac1m \sum_{l\neq s} F_{i_li_l}\x_l\x_l^\top+\C)^{-1}$, for which we have
\begin{equation*}
  \frac{1}{m} \sum^{m}_{s=1} \EE[F_{i_si_s}\x_s{\x}^{\top}_s]=\sum^n_{i=1}  F_{ii}{\ba}_i{\ba}_i^\top.
\end{equation*}

Our objective here is to find $\check{\Q}$ (and $F_{ii}$) such that $\|\EE[\check{\Q}]-\H^{-1}\|\simeq 0$, for $\H^{-1}= (\A^\top  \A + \C)^{-1}$. 
This is
 \begin{align*} 
    \EE[\check{\Q}]-\H^{-1} = \EE[\check{\Q}]\A^\top  \A\H^{-1}-\EE[\check{\Q}\A^\top   \check{\S}^\top \check{\S}\A ]\H^{-1} \simeq 0.
 \end{align*}
Using the Sherman-Morrison formula in \Cref{lemm:sherman-morrison}, we further ascertain that
 \begin{align}
 \EE[\check{\Q}\A^\top   \check{\S}^\top \check{\S}\A \H^{-1}]&=\EE \left[\frac{\check{\Q}_{-s}F_{i_si_s}\A^\top\ee_{i_s}\ee_{i_s}^\top/\pi_{i_s}\A\H^{-1}}{1+F_{i_si_s}\ee_{i_s}^\top\A\check{\Q}_{-s}\A^\top\ee_{i_s}/m\pi_{i_s}}\right]=\sum^n_{i=1}\EE\left[\frac{\check{\Q}_{-s}F_{ii}\A^\top\ee_{i}\ee_{i}^\top\A\H^{-1}}{1+F_{ii}\ee_{i}^\top\A\check{\Q}_{-s}\A^\top\ee_{i}/m\pi_{i}}\right]\nonumber\\
 & \simeq\sum^n_{i=1}\EE\left[\frac{\check{\Q}_{-s}F_{ii}\A^\top\ee_{i}\ee_{i}^\top\A\H^{-1}}{1+F_{ii}\ee_{i}^\top\A\H^{-1}\A^\top\ee_{i}/m\pi_{i}}\right],\nonumber
 \end{align}
 where we observe that the \emph{exact} leverage score $\ee_{i}^\top\A\H^{-1}\A^\top\ee_{i} = \mathbf{a}_i^\top (\A^\top  \A + \C)^{-1} \mathbf{a}_i = \ell_i^\C$ given $\C$ as in \Cref{def:lev} naturally appears in the denominator from this leave-one-out derivation.

This, together with the rank-one perturbation lemma in~\citet[Lemma~2.6]{silverstein1995empirical}, gives that 
\begin{align}
\EE[\check{\Q}\A^\top   \check{\S}^\top \check{\S}\A \H^{-1}]\simeq \EE[  \check{\Q}]  \A^\top\sum^n_{i=1}\frac{F_{ii}\ee_{i}\ee_{i}^\top}{1+F_{ii}\ell^{\C}_i/m\pi_{i}}\A\H^{-1}= \EE[  \check{\Q}]  \A^\top\A\H^{-1},  \nonumber 
\end{align}
where we take the \emph{debiasing factor} $F_{ii}= m\pi_{i}/( m\pi_{i} - \ell_{i}^\C )$ such that $F_{ii}/(1+F_{ii}\ell^{\C}_i/m\pi_{i})=1$.
This thus leads to
 \begin{align}
  \check\S =\diag \left\{\sqrt{m/(m-\ell_{i_s}^\C/\pi_{i_s})} \right\}^m_{s=1}\cdot\S, \quad i_s\in \{1,\ldots,n\}.\nonumber
\end{align}

\subsection{Detailed Proof of \Cref{prop:debias}}
\label{subsec:detailed_proof_prop:debias}

The proof of \Cref{prop:debias} follows the same line of arguments as in that of \Cref{theo:inverse-bias}.  
In the following, we presented the proof for completeness.
We specifically highlight the difference in the proof of \Cref{prop:debias} from that of \Cref{theo:inverse-bias}, where $\Q$ is replaced by $\check\Q$.

The proof of \Cref{prop:debias} also comes in the following two steps:
\begin{enumerate}
  \item construct a high-probability event $\zeta$, as in \eqref{eq:events}; and
  \item conditional on the event $\zeta$, derive a bound for the spectral norm $\| \I_d - \EE_{\zeta} [\check\Q] \H \|$ using again the ``leave-one-out'' type analysis.
\end{enumerate}

For the second term above, we have, for  $m>2\rho_{\max} d_{\eff}$ that
\begin{align}\label{eq:F_ii}
    F_{i_si_s}<\frac{m\pi_{i_s}}{ m\pi_{i_s} -2^{-1} m\pi_{i_s} }<2, ~s= 1,\ldots,m.
\end{align} 
This yields that, for given $s\in \{1,\ldots,m\}$, there exists an index $j = j(s) \in \{1, 2, 3\}$ such that $\check\Q \preceq \check\Q_{-s} \preceq 6\H^{-1}$ holds on the event $\zeta_j$.

Denote $\check \gamma_s=1+\frac{1}{m}F_{i_si_s}\x_s^\top  \check{\Q}_{-s}\x_s$, we then  obtain 
\begin{align}
\EE_{\zeta}[  \check{\Q} ]\H-\I_d&=(\EE_{\zeta}[  \check{\Q} ]- \H^{-1})\H=\EE_{\zeta}\left[  \check{\Q}\left(\A^\top\A-\sum^m_{s=1}\frac{1}{m}F_{i_si_s}\x_s\x^\top_s\right)\H^{-1}\right] \H \nonumber\\
&=\EE_{\zeta}\left[  \check{\Q}\left( \A^\top\A-\sum^m_{s=1}\frac{1}{m}F_{i_si_s}\x_s\x^\top_s\right)\right]=\EE_{\zeta}[  \check{\Q}\A^\top\A]-\EE_{\zeta}\left[\frac{1}{\check \gamma_s}  \check{\Q}_{-s}F_{i_si_s}\x_s\x^\top_s\right] \nonumber\\
&=\underbrace{\EE_{\zeta}[  \check{\Q}-  \check{\Q}_{-s}]\A^\top\A}_{\check{\Z}_1}+\underbrace{\EE_{\zeta}[  \check{\Q}_{-s}(\A^\top\A-\x_s\x^\top_s)]}_{\check{\Z}_2}  +\underbrace{\EE_{\zeta}\left[\left(1-\frac{F_{i_si_s}}{\check \gamma_s}\right)  \check{\Q}_{-s}\x_s\x^\top_s    \right]}_{\check{\Z}_3},  \nonumber 
\end{align}
which leads to 
\begin{align}\label{eq:hatbound_rew}
\|\I_d-\H^{\frac{1}{2}}\EE_{\zeta}[  \check{\Q} ]\H^{\frac{1}{2}}\|=\|\H^{\frac{1}{2}}(\EE_{\zeta}[  \check{\Q} ]\H-\I_d)\H^{-\frac{1}{2}}\|&\leq \|\H^{\frac{1}{2}}\check{\Z}_1\H^{-\frac{1}{2}}\|  +\|\H^{\frac{1}{2}}\check{\Z}_2\H^{-\frac{1}{2}}\|+\|\H^{\frac{1}{2}}\check{\Z}_3\H^{-\frac{1}{2}}\|.
\end{align}

We first bound the first term $\|\H^{1/2}\check{\Z}_1\H^{-1/2}\|$ in \eqref{eq:hatbound_rew}.
Adapting the bound for $\EE_{\zeta}[\Q_{-s}-\Q]$ and $\|\widetilde{\H}^{1/2}\Z_1\widetilde{\H}^{-1/2}\|$ in the proof of \Cref{theo:inverse-bias} in \Cref{subsec:detailed_proof_theo:inverse-bias}, and using \eqref{eq:F_ii} along with the fact that $\H^{1/2}  \check{\Q}_{-s}\H^{1/2}\preceq 6\I_d$ when conditioned on $\zeta^{'}$, we get 
\begin{align}
\EE_{\zeta}[  \check{\Q}_{-s}-  \check{\Q}]\preceq  \EE_{\zeta} \left[\frac{F_{i_si_s}}{\check \gamma_s m}  \check{\Q}_{-s}\x_s\x^\top_s  \check{\Q}_{-s} \right] \preceq \frac{4}{m}\EE_{\zeta^{'}}[  \check{\Q}_{-s}\H  \check{\Q}_{-s}],\label{eq:Q_diffe_hat}
\end{align} 
and 
\begin{align}
\|\H^{\frac{1}{2}}\check{\Z}_1\H^{-\frac{1}{2}}\|= \|\H^{\frac{1}{2}}\EE_{\zeta}[ \check{\Q}_{-s}- \check{\Q}]\H^{\frac{1}{2}}\H^{-\frac{1}{2}}\A^\top\A\H^{-\frac{1}{2}}\|=O\left(\frac{1}{m}\right).\nonumber
\end{align}

 Next, we bound the second term $\|\H^{1/2}\check{\Z}_2\H^{-1/2}\|$ in \eqref{eq:hatbound_rew}. 
 Note that $\EE_{\zeta'}[\x_s^\top\H^{-1}\x_s]=d_{\eff}$ and 
\begin{align}
 \Var_{\zeta'}[\x_s^\top\H^{-1}\x_s] \leq \EE_{\zeta'}[(\x_s^\top\H^{-1}\x_s)^2]=\sum^n_{j=1}\pi_j\frac{(\ba_j^\top\H^{-1}\ba_j)^2}{\pi^2_j}=\sum^n_{j=1}\frac{(\ba_j^\top\H^{-1}\ba_j)^2}{\pi_j}\leq\rho_{\max} d^2_{\eff}.\nonumber
\end{align}
Then, using Chebyshev's inequality, we have, for $x\geq 2d_{\eff}$ that
\begin{align}
\Pr(\x_s^\top\H^{-1}\x_s\geq x~|~\zeta')\leq \frac{\rho_{\max} d^2_{\eff}}{x^2}. \nonumber
\end{align}
Analogous  to the bound on $\EE_{\zeta}[\Q_{-s}(\A^\top\D\A-\tilde D_{s}\x_s\x^\top_s)]$ and $ \|\widetilde{\H}^{1/2}\Z_2\widetilde{\H}^{-1/2}\|$ in \Cref{subsec:detailed_proof_theo:inverse-bias},  we get
 \begin{align} 
      \EE_{\zeta}[ \check{\Q}_{-s}(\A^\top\A-\x_s\x^\top_s)]=-\frac{1}{1-\delta_3}\EE_{\zeta'}[ \check{\Q}_{-s}(\A^\top\A-\x_s\x^\top_s)\cdot\mathbf{1}_{\neg\zeta_3}],\nonumber
 \end{align}
and 
\begin{align*}
\|\H^{\frac{1}{2}}\check{\Z}_2\H^{-\frac{1}{2}}\|&\leq 12 \delta_3+12 \EE_{\zeta'}[\x_s^\top\H^{-1}\x_s\cdot\mathbf{1}
\neg\zeta_3]=12 \delta_3+12 \int^\infty_0 \Pr(\x_s^\top\H^{-1}\x_s\cdot \mathbf{1}_{\neg\zeta_3}\geq x~|~\zeta')dx \\ 
&=O\left(\sqrt{\frac{\rho_{\max}^3 d^3_{\eff}}{m^3}}\right).
 \end{align*}

Now, we move on to bound the last term $\|\H^{1/2}\check{\Z}_3\H^{-1/2}\|$ in \eqref{eq:hatbound_rew}. 
Considering \eqref{eq:F_ii},  it follows that
\begin{align}\label{eq:1-f_ii}
 \left| 1-\frac{F_{i_si_s}}{\check \gamma_s}\right|&=\left| \frac{\ba_{i_s}^\top \check{\Q}_{-s}\ba_{i_s}-\ell^{\C}_{i_s}}{m\pi_{i_s}-\ell^{\C}_{i_s}+\ba_{i_s}^\top \check{\Q}_{-s}\ba_{i_s}}\right|\leq\left| \frac{\ba_{i_s}^\top \check{\Q}_{-s}\ba_{i_s}-\ell^{\C}_{i_s}}{m\pi_{i_s}-\ell^{\C}_{i_s}}\right| =F_{i_si_s}\left| \frac{\ba_{i_s}^\top \check{\Q}_{-s}\ba_{i_s}-\ell^{\C}_{i_s}}{m\pi_{i_s}}\right|\nonumber\\
 &\leq2\left| \frac{\ba_{i_s}^\top \check{\Q}_{-s}\ba_{i_s}-\ell^{\C}_{i_s}}{m\pi_{i_s}}\right|= 2\left| \frac{\ba_{i_s}^\top \check{\Q}_{-s}\ba_{i_s}-\ba_{i_s}^\top\H^{-1}\ba_{i_s}}{m\pi_{i_s}}\right|.
\end{align}
Following the bound on $ \|\widetilde{\H}^{1/2}\Z_3\widetilde{\H}^{-1/2}\|$ in \Cref{subsec:detailed_proof_theo:inverse-bias}, and recalling \eqref{eq:F_ii}, we further have
\begin{align}
   & \|\H^{\frac{1}{2}}\check{\Z}_3\H^{-\frac{1}{2}}\|\leq   \sup\limits_{\|\uu\|=1,\|\vv\|=1}
     \EE_{\zeta} \left[\left| 1-\frac{F_{i_si_s}}{\check \gamma_s} \right|\left|\uu^\top\H^{\frac{1}{2}}   \check{\Q}_{-s}\x_s\x^\top_s \H^{-\frac{1}{2}}\vv \right| \right] \nonumber \\ 
     &\leq \EE_{\zeta^{'}} \left[ \max\limits_{1\leq j\leq n}74\left|\frac{\ba_{j}^\top \check{\Q}_{-s}\ba_{j}-\ba_{j}^\top\H^{-1}\ba_{j}}{m\pi_{j}}\right|
     \right]\leq \underbrace{  \EE_{\zeta^{'}} \left[ \max\limits_{1\leq j\leq n}74\left|\frac{\ba_{j}^\top \check{\Q}_{-s}\ba_{j}-\EE_{\zeta^{'}}[\ba_{j}^\top \check{\Q}_{-s}\ba_{j}]}{m\pi_{j}}\right|
     \right]}_{\check{M}_1}\nonumber\\  
     & +
    \underbrace{ \max\limits_{1\leq j\leq n} \frac{74\left|\EE_{\zeta^{'}}[\ba_{j}^\top \check{\Q}_{-s}\ba_{j}]-\ba_{j}^\top\H^{-1}\ba_{j}\right|}{m\pi_{j}}
     }_{\check{M}_2}. \label{eq: hat_z_3}
\end{align}
Subsequently, like the bounds for $ M_1$ and $M_2$ in \Cref{subsec:detailed_proof_theo:inverse-bias}, and using \eqref{eq:F_ii}, we obtain the following bounds for $\check{M}_1$ and~$\check{M}_2$:
 \begin{align*}
  \check{M}_1 = O\left(\frac{\log n}{\log (\log n)}\sqrt{\frac{\rho_{\max}^3 d^3_{\eff}}{m^3}}\right),
 \end{align*}
 and
 \begin{align*}
  \check{M}_2  &\leq \max\limits_{1\leq j\leq n}\frac{ 74 \ba_j^\top \H^{-1}\ba_j}{m\pi_{j}}(\|\H^{\frac{1}{2}}(\EE_{\zeta'}-\EE_{\zeta})[ \check{\Q}_{-s}]\H^{\frac{1}{2}}\|+\|\H^{\frac{1}{2}}\EE_{\zeta}[ \check{\Q}_{-s}- \check{\Q}]\H^{\frac{1}{2}}\|+\|\H^{\frac{1}{2}}(\EE_{\zeta}[ \check{\Q}]\H-\I_d)\H^{-\frac{1}{2}}\|)\\
  &= O\left(\frac{\log n}{\log (\log n)}\sqrt{\frac{\rho_{\max}^3 d^3_{\eff}} {m^3}}\right).\nonumber 
 \end{align*}
Combining the above, we conclude that 
\begin{align*}
    \Big\|\I_d-\H^{\frac{1}{2}}\EE_{\zeta}[ \check{\Q} ]\H^{\frac{1}{2}}\Big\| = O\left(\frac{\log n}{\log (\log n)}\sqrt{\frac{\rho_{\max}^3 d^3_{\eff}} {m^3}}\right),
\end{align*}
 and thus complete the proof of \Cref{prop:debias}.
\qedwhite

\subsection{Discussions on \Cref{prop:debias} and Auxiliary Results}
\label{subsec:discussion_prop:debias}

Here, we present some auxiliary results in addition to the fine-grained debiasing results in \Cref{prop:debias}.
Precisely, we show in \Cref{subsec:proof_of_approximate_lev_check_S} that by substituting exact leverage scores with some (computationally efficient) approximations leads to a controlled inversion bias.    
We then provide the proof of~\Cref{coro:inver_bias_constant_debias} in~\Cref{rem:proof_inver_bias_constant_debias}. 
A ``counter-example'' to \citet[Theorem~10]{derezinski2021sparse} that uses scalar debiasing to achieve zero inversion bias in the case of \emph{exact} leverage score sampling is given in \Cref{subsec:exact_lev_zero_bias}.
Finally, we establish the proof of~\Cref{coro:debiasing_srht} in~\Cref{rem:proof of SRHT}.

\subsubsection{Discussions on Debiasing $\check{\S}$ Using Approximate Leverage Scores} \label{subsec:proof_of_approximate_lev_check_S}

\begin{corollary}[Fine-grained debiasing $\check{\S}$ using approximate leverage scores]\label{coro:proof_of_approximate_lev}\normalfont
Under the setting and notations of \Cref{theo:inverse-bias}, if one uses approximate leverage scores $\check\ell^{\C}_i$  in the debiasing matrix, instead of the exact leverage scores $\ell^{\C}_i$ in \eqref{eq:debias_check_S} in \Cref{prop:debias}, that is
\begin{equation*}
    \check\S =\diag \left\{\sqrt{m/(m-\check\ell_{i_s}^\C/\pi_{i_s})} \right\}^m_{s=1}\cdot\S, \quad (1 - \omega) \ell_i^\C \leq  \check \ell^{\C}_{i} \leq (1+ \omega) \ell_i^\C, \quad \omega\in[0,1].
\end{equation*}
Then, there exists $C > 0$ independent of $n,d_{\eff}$ such that for $m\geq C\rho_{\max} d_{\eff} (\log(d_{\eff}/\delta )+\max\left\{\omega/\epsilon, (\log n)^{2/3} /(\epsilon\log (\log n))^{2/3} \right\})$ with $\delta\leq m^{-3}$, $(\A^\top \check\S^\top \check\S \A + \C )^{-1}$ is an $(\epsilon,\delta)$-unbiased estimator of $(\A^\top  \A + \C)^{-1}$.
\end{corollary}

\begin{proof}[Proof of~\Cref{coro:proof_of_approximate_lev}]
The proof of \Cref{coro:proof_of_approximate_lev} largely follows the proof of \Cref{prop:debias}, and is presented here for completeness. 
In the following, we emphasize the difference from the proof of \Cref{prop:debias}, particularly regarding the matrix $\check\S$,  which is constructed using the approximate leverage scores denoted by $\check \ell^{\C}_i$ as
 \begin{align*}
  \check\S =\diag \left\{\sqrt{m/(m-\check \ell^{\C}_{i_s}/\pi_{i_s}}) \right\}^m_{s=1}\cdot\S, \quad i_s\in \{1,\ldots,n\}.
\end{align*}
This implies, in the proof of \Cref{prop:debias} that
\begin{align}\label{eq:F_ii_approxi_lev}
    F_{i_si_s}= \frac{m\pi_{i_s}}{ m\pi_{i_s} - \check\ell_{i_s}^\C}.
\end{align}
Note that in this setting, we have, for $\omega\in[0,1]$ that $| \check \ell^{\C}_{i}-\ell^\C_{i}|\leq \omega\ell^\C_{i} $ and $\check \ell^{\C}_{i} \leq (1+\omega)\ell^\C_{i}\leq 2 \ell^\C_{i}$.
As such, for $m \geq C \rho_{\max} d_{\eff} \geq C \ell^\C_{i}/\pi_{i_s}$ with $C> 4$, we have
 \begin{align}
    F_{i_si_s} \leq  \frac{m\pi_{i_s}}{ m\pi_{i_s} - 2 \ell^\C_{i}}<2.\nonumber
\end{align}
Thus, adapting the proof of~\Cref{prop:debias}, we get the same bounds on $\|\H^{1/2}\check{\Z}_1\H^{-1/2}\|$ and $ \|\H^{1/2}\check{\Z}_2\H^{-1/2}\|$ as in  \Cref{subsec:detailed_proof_prop:debias}, that is
\begin{align*}
    \|\H^{\frac{1}{2}}\check{\Z}_1\H^{-\frac{1}{2}}\|= O\left(\frac{ 1}{m}\right),~    \|\H^{\frac{1}{2}}\check{\Z}_2\H^{-\frac{1}{2}}\|= O\left(\sqrt{\frac{\rho_{\max}^3 d^3_{\eff}}{m^3}}\right).
\end{align*}

Next,  we establish a bound for the term $ \|\H^{1/2}\check{\Z}_3\H^{-1/2}\|$ in the setting of \Cref{coro:proof_of_approximate_lev}.
Using \eqref{eq:F_ii_approxi_lev},  we obtain
    \begin{align}
   \left| 1-\frac{F_{i_si_s}}{\check \gamma_s}\right|&=\left | \frac{\ba_{i_s}^\top \check{\Q}_{-s}\ba_{i_s}-\check \ell^{\C}_{i_s}}{m\pi_{i_s}-\check \ell^{\C}_{i_s}+\ba_{i_s}^\top \check{\Q}_{-s}\ba_{i_s}}\right|\leq \left| \frac{\ba_{i_s}^\top \check{\Q}_{-s}\ba_{i_s}-\check \ell^{\C}_{i_s}}{(1-C^{-1})m\pi_{i_s}}\right|\leq 2\left| \frac{\ba_{i_s}^\top \check{\Q}_{-s}\ba_{i_s}-\check\ell^{\C}_{i_s}}{m\pi_{i_s}}\right|\nonumber\\
      &\leq 2\left| \frac{\ba_{i_s}^\top \check{\Q}_{-s}\ba_{i_s}-\ell^{\C}_{i_s}}{m\pi_{i_s}}\right|+2\left| \frac{\ell^{\C}_{i_s}-\check \ell^{\C}_{i_s}}{m\pi_{i_s}}\right|\leq2\left| \frac{\ba_{i_s}^\top \check{\Q}_{-s}\ba_{i_s}-\ell^{\C}_{i_s}}{m\pi_{i_s}}\right|+\frac{2\omega\rho_{\max} d_{\eff}}{m}.\nonumber
\end{align}
This, together with \eqref{eq: hat_z_3} yields that
\begin{align}
    \|\H^{\frac{1}{2}}\check{\Z}_3\H^{-\frac{1}{2}}\|
     = O\left(\frac{\omega\rho_{\max} d_{\eff}}{m}+\frac{\log n}{\log (\log n)}\sqrt{\frac{\rho_{\max}^3 d^3_{\eff}}{m^3}}\right).\nonumber
\end{align}     
Putting these bounds together, we conclude the proof of \Cref{coro:proof_of_approximate_lev}.
\end{proof}

\subsubsection{Proof of \Cref{coro:inver_bias_constant_debias}} \label{rem:proof_inver_bias_constant_debias}
The proof~\Cref{coro:inver_bias_constant_debias} largely mirrors that of \Cref{prop:debias}, and is included here for thoroughness. 
For $m > C d_{\eff}$ with $C> 2$,  we have
\begin{align}
   F_{i_si_s}=\frac{m}{m-d_{\eff}} =\frac{m\pi_{i_s}}{m\pi_{i_s}-d_{\eff}\pi_{i_s}}\leq \frac{m\pi_{i_s}}{(1-C^{-1})m\pi_{i_s}}<2.  \nonumber
\end{align}
Then, recalling $\pi_i\in [\ell^{\C}_{i}/(d_{\eff} \rho_{\max}), \ell^{\C}_{i}/(d_{\eff}\rho_{\min})]$ such that $|\pi_i- \ell_i^{\C}/d_{\eff}|\leq \epsilon_{\rho}\ell_i^{\C}/d_{\eff}$  with $\epsilon_{\rho}=\max\{(\rho_{\min}^{-1}-1),(1 - \rho_{\max}^{-1})\}$, and \eqref{eq:1-f_ii} can be rewritten as
    \begin{align}
   \left| 1-\frac{F_{i_si_s}}{\check \gamma_s}\right|&=\left | \frac{\ba_{i_s}^\top \check{\Q}_{-s}\ba_{i_s}-d_{\eff}\pi_{i_s}}{m\pi_{i_s}-d_{\eff}\pi_{i_s}+\ba_{i_s}^\top \check{\Q}_{-s}\ba_{i_s}}\right|\leq \left| \frac{\ba_{i_s}^\top \check{\Q}_{-s}\ba_{i_s}-d_{\eff}\pi_{i_s}}{(1-C^{-1})m\pi_{i_s}}\right|\leq  2\left| \frac{\ba_{i_s}^\top \check{\Q}_{-s}\ba_{i_s}-\ell^{\C}_{i_s}}{m\pi_{i_s}}\right|+ 2\left| \frac{\ell^{\C}_{i_s}-d_{\eff}\pi_{i_s}}{m\pi_{i_s}}\right|\nonumber\\
      &\leq 2\left| \frac{\ba_{i_s}^\top \check{\Q}_{-s}\ba_{i_s}-\ell^{\C}_{i_s}}{m\pi_{i_s}}\right|+ 2d_{\eff}\left| \frac{\ell^{\C}_{i_s}/d_{\eff}-\pi_{i_s}}{m\pi_{i_s}}\right|\leq 2\left| \frac{\ba_{i_s}^\top \check{\Q}_{-s}\ba_{i_s}-\ell^{\C}_{i_s}}{m\pi_{i_s}}\right|+2\epsilon_{\rho} d_{\eff} \frac{\ell^{\C}_{i_s}/d_{\eff}}{m\pi_{i_s}}\nonumber\\
      &\leq 2\left| \frac{\ba_{i_s}^\top \check{\Q}_{-s}\ba_{i_s}-\ell^{\C}_{i_s}}{m\pi_{i_s}}\right|+\frac{2\epsilon_{\rho} \rho_{\max} d_{\eff}}{m}.\nonumber
\end{align}
This, combined with the bound in~\eqref{eq: hat_z_3}, gives that 
\begin{align*}
  \|\H^{\frac{1}{2}}\check{\Z}_3\H^{-\frac{1}{2}}\|
     & = O\left(\frac{\epsilon_{\rho} \rho_{\max} d_{\eff}}{m}+  \frac{\log n}{\log (\log n)}\sqrt{\frac{\rho_{\max}^3 d^3_{\eff}}{m^3}}\right).   
\end{align*}
Recall the proof of \Cref{prop:debias} that
\begin{align*}
    \|\H^{\frac{1}{2}}\check{\Z}_1\H^{-\frac{1}{2}}\|= O\left(\frac{ 1}{m}\right),~    \|\H^{\frac{1}{2}}\check{\Z}_2\H^{-\frac{1}{2}}\|= O\left(\sqrt{\frac{\rho_{\max}^3 d^3_{\eff}}{m^3}}\right).
\end{align*}
Putting these together, we conclude the proof of \Cref{coro:inver_bias_constant_debias}. 

In particular, taking $\rho_{\max} = 3/2$ and $\rho_{\min}=1/2$, we have $\epsilon_{\rho}=1$ and thus an inversion bias of order $O(d_{\eff}/m )$.
\qedwhite

\subsubsection{A ``Counterexample'' to {\citet[Theorem~10]{derezinski2021sparse}}} \label{subsec:exact_lev_zero_bias}
\begin{corollary}[``Counterexample'' to the lower bound in {\citet[Theorem~10]{derezinski2021sparse}} using exact leverage sampling]\label{coro: exact_lev_zero_bias}
For any $n \geq 2d \geq 4$ , there exists $\A\in \RR^{n \times d}$, an \emph{exact} leverage score sampling matrix  $\S\in \RR^{m \times n}$ in \Cref{def:lev} and a high probability event $\zeta$ (that ensures the invertibility of $ \A^\top \S^\top \S \A$), such that when conditioned on $\zeta$, $(\gamma \A^\top \S^\top \S \A+\C)^{-1}$ is an unbiased estimator of $(\A^\top \A +\C)^{-1}$ for $\gamma=\frac{m}{d}\EE_{\zeta}[1/b]$, $b$ distributed as ${\rm Binomial}(m, 1/d)$, and $\C=\zo_d$.
\end{corollary}

\begin{proof}[Proof of~\Cref{coro: exact_lev_zero_bias}]
We begin by recalling the setup of the matrix $\A$ in~\citet[Theorem~10]{derezinski2021sparse}. 
Without loss of generality, assume  $n = 2d$ (otherwise, we pad $\A$ with zeros). The matrix $\A$  consists of $n = 2d$ scaled standard basis vectors, where consecutive rows are defined as $\ba^{\top}_{2(i-1)+1} =\ba^{\top}_{2(i-1)+2} =\frac{1}{\sqrt{2}} \ee^{\top}_i$ for $i \geq 2$, and the first two rows are $\ba^{\top}_1 =\frac{1}{ \sqrt{4}} \ee^{\top}_1$,  $  \ba^{\top}_2 =\frac{3}{ \sqrt{4}}\ee^{\top}_1$. 
This is
\begin{align*}
    \A = \begin{bmatrix}
\frac{1}{ \sqrt{4}} & 0 & \dots & 0 \\
\frac{3}{\sqrt{4}} & 0 & \dots & 0 \\
0 &\frac{1}{ \sqrt{2}} & \dots & 0 \\
0 &\frac{1}{ \sqrt{2}} & \dots & 0 \\
\vdots & \vdots & \ddots & \vdots \\
0 & 0 & \dots & \frac{1}{ \sqrt{2}}\\
0 & 0 & \dots & \frac{1}{ \sqrt{2}}
\end{bmatrix}.
\end{align*}

In this case we have $\A^{\top}\A=\I_d$ and $\C=\zo_d$, so that the leverage scores satisfy $\ell_{1}^\C=\frac{1}{4}$, $\ell_{2}^\C=\frac{3}{4}$, and $\ell_{2(i-1)+1}^\C=\ell_{2(i-1)+2}^\C=\frac{1}{2}$ for $i\geq 2$. 
The leverage score sampling distribution $\{\pi_i\}_{i=1}^n$ as in \Cref{def:lev}, used to construct the sampling matrix $\S$ of size $m\geq d$, is thus given by
\begin{align*}
\pi_i
= 
    \begin{cases}
\frac{1}{4d}, & \text{for } i = 1, \\
\frac{3}{4d}, & \text{for } i = 2, \\
\frac{1}{2d}, & \text{otherwise}.
\end{cases}
\end{align*}
For any $\gamma>0$, the matrix  $\gamma \A^\top \S^\top \S \A$   is diagonal, with diagonal entries  given by
  \begin{align*}
   \left[ \gamma  \A^\top \S^\top \S \A\right]_{ii}=\frac{\gamma db_i }{m},~\mbox{for}~i=1,\ldots, d,
  \end{align*}
where $b_i$'s are all identically (but not independently) distributed as ${\rm Binomial}(m, 1/d)$. 
Then, conditioning on \emph{any} high probability event $\zeta$ that ensures the invertibility of $  \A^\top \S^\top \S \A$,  it follows that    
\begin{align*}
 \EE_{\zeta}\left[\left[ \left(\gamma  \A^\top \S^\top \S \A\right)^{-1}\right]_{ii}\right] =  \frac{m}{\gamma d} \EE_{\zeta}\left[\frac{1}{b_i}\right],~\mbox{for}~i=1,\ldots, d,
  \end{align*}
so that
\begin{align*}
\EE_{\zeta}\left[ \left(\gamma  \A^\top \S^\top \S \A\right)^{-1}\right]=\left(  \A^\top \A\right)^{-1}    
\end{align*} 
with $ \gamma=  \frac{m}{d}\EE_{\zeta}\left[1/b\right] $, where  $b$ is any one of the aforementioned $\{b_i\}^d_{i=1}$. 
This completes the proof of \Cref{coro: exact_lev_zero_bias}.
\end{proof}

\subsubsection{Proof of~\Cref{coro:debiasing_srht}}
\label{rem:proof of SRHT}

The proof of \Cref{coro:debiasing_srht} comes in the following two steps:
\begin{enumerate}
\item construct two \emph{independent} high probability events: $\zeta_{\omega}$ on which the randomized Walsh--Hadamard transform $\H_n \D_n \A$ of $\A$ as in \Cref{def:srht} has approximately uniform leverages scores; and $\zeta$ as in~\eqref{eq:events} to subspace embedding, but for $\x^\top_s=\ee^\top_{i_s}\H_n \D_n\A/\sqrt{n\pi_{i_s}} \in \RR^d $ the $i_s^{th}$ row of $\tilde{\A}_{\rm SRHT}$; and 
\item conditioned on the event $\zeta\bigcap\zeta_{\omega}$, apply the bounds in~\Cref{rem:proof_inver_bias_constant_debias} to obtain inversion bias for $(\frac{m}{m-d_{\eff}}\tilde{\A}_{\rm SRHT}^\top \tilde{\A}_{\rm SRHT} + \C )^{-1}$, where $d_{\eff}=\sum_{i=1}^n \ell^{\C}_i $ is the effective dimension of $\A$ given $\C$.
\end{enumerate}

We start by recalling some notations from (the proof of) \Cref{prop:debias}. 
Denote  $\A_\C=\A\H^{-1/2}$ with $\H=\A^\top\A+\C$ and let $\ell_i^{\C}(\H_n\D_n\A/\sqrt{n}) =\|\ee_i^\top\H_n\D_n\A_\C\|^2/n$ be the  $i^{th}$ leverage score of $\H_n\D_n\A/\sqrt{n}$. 
Define the following event $\zeta_{\omega}$:
\begin{align*}
\zeta_{\omega}: \left| \ell^{\C}_i(\H_n\D_n\A/\sqrt{n}) -\frac{d_{\eff}}{n}\right| \leq \omega,\quad \forall 1\leq i\leq n.
\end{align*}
Note that this is equivalent to, for uniform sampling with $\pi_i=1/n$ that
\begin{align*}
 \left| \frac{\ell^{\C}_{i}(\H_n\D_n\A/\sqrt{n}) -d_{\eff}\pi_{i}}{m\pi_{i}}\right| \leq\frac{n}{m}\omega,\quad\forall 1 \leq i\leq n.
\end{align*}
Next, we  verify that the event $\zeta_{\omega}$ holds with a controlled probability. 
Precisely, 
some $C>0$, it follows from \Cref{lemm: HD_balance_row_norms} below and $\|\A_\C\|_F^2=d_{\eff}$ that
\begin{align}\label{eq:pr_lev_srht}
\Pr\left( \left|\ell_i^{\C}(\H_n\D_n\A/\sqrt{n})-\frac{d_{\eff}}{n} \right|\geq \omega,1\leq i\leq n\right)\leq \frac{\delta}{2},
\end{align}
with $\omega\geq \max\{C\sqrt{d_{\eff}\log(n/\delta)}/n,C\log(n/\delta)/n\}$, which is equivalent to
\begin{align*}
 \left| \frac{\ell^{\C}_{i}(\H_n\D_n\A/\sqrt{n}) -d_{\eff}\pi_{i}}{m\pi_{i}}\right| \leq \max\left\{\frac{C\sqrt{d_{\eff}\log(n/\delta)}}{m},\frac{C\log(n/\delta)}{m}\right\},~\text{for}~1\leq i\leq n.
\end{align*}
Further recall from \citet[Lemma 1]{lacotte2021adaptive} that a sampling size  $m\geq C\rho_{\max}(d_{\eff}+\log(1/(\epsilon\delta))\log(d_{\eff}/\delta)/\epsilon^2)$, $\rho_{\max}\equiv \max_{1\leq i \leq n}n\ell^{\C}_i (\H_n \D_n \A/\sqrt{n})/d_{\eff}$ and $\epsilon\in(0, 1/2]$ suffices to ensure that $\H^{-1/2}\tilde{\A}_{\rm SRHT}^\top \tilde{\A}_{\rm SRHT}\H^{-1/2}$ is an $(\epsilon,\delta)$-approximation of $\A_{\C}^\top\A_{\C}$.  
This leads to  $\Pr(\zeta)\leq \delta/2$. 
Combining the above results, we obtain $\Pr(\zeta \bigcap \zeta_{\omega})\leq \delta$.     
This, together with
\begin{align*}
     &\|\H^{\frac{1}{2}}\check{\Z}_3\H^{-\frac{1}{2}}\|
      = O\left(  \frac{\log n}{\log (\log n)}\sqrt{\frac{\rho_{\max}^3 d^3_{\eff}}{m^3}}\right)+ O\left(\max\limits_{1\leq i\leq n} \left| \frac{\ell^{\C}_{i}(\H_n\D_n\A/\sqrt{n}) -d_{\eff}\pi_{i}}{m\pi_{i}}\right|\right), \\
     & \|\H^{\frac{1}{2}}\check{\Z}_1\H^{-\frac{1}{2}}\|= O\left(\frac{ 1}{m}\right),~    \|\H^{\frac{1}{2}}\check{\Z}_2\H^{-\frac{1}{2}}\|= O\left(\sqrt{\frac{\rho_{\max}^3 d^3_{\eff}}{m^3}}\right).
\end{align*}
in~\Cref{rem:proof_inver_bias_constant_debias}, concludes the proof of \Cref{coro:debiasing_srht}.
\qedwhite

\begin{remark}[SRHT inversion bias with different sampling sizes]\normalfont\label{rem:sub_emb_srht}
Recalling \eqref{eq:pr_lev_srht},  we observe that
\begin{equation*}
     \rho_{\max} = \max_{1\leq i \leq n}n\ell^{\C}_i (\H_n \D_n \A/\sqrt{n})/d_{\eff} \leq 2
\end{equation*}
holds with the probability at least $1-\delta/2$. 
Consequently, choosing $m\geq C(d_{\eff}+\log(1/(\epsilon\delta))\log(d_{\eff}/\delta)/\epsilon^2)$ suffices to ensure that $\H^{-1/2}\tilde{\A}_{\rm SRHT}^\top \tilde{\A}_{\rm SRHT}\H^{-1/2}$ is an $(\epsilon,\delta)$-approximation (or subspace embedding) of $\A_{\C}^\top\A_{\C}$.
On the other hand, note from \Cref{coro:debiasing_srht} that
\begin{enumerate}
    \item by taking $\nu = 0$, a number of $m=\Theta(d_{\eff})$ samples with $ n \exp(-d_{\eff}) < \delta < m^{-3}$, we have that $(\frac{m}{m-d_{\eff}}\tilde{\A}_{\rm SRHT}^\top \tilde{\A}_{\rm SRHT} + \C )^{-1}$ is an $(O(1),\delta)$-unbiased estimator of $(\A^\top  \A + \C)^{-1}$: This, in particular, agrees with the above subspace embedding condition up to logarithmic factor.
    \item taking $\nu > 1$ in \Cref{coro:debiasing_srht} allows one to (further) reduce the inversion bias (to a level that is significantly smaller than  $O(1)$) by increasing the sample size $m = \Theta(d_{\eff}^{1+\nu})$.
\end{enumerate}
\end{remark}

\begin{lemma}[Row norms]\label{lemm: HD_balance_row_norms} 
Let $\H_n\in \RR^{n \times n}$ be the Walsh--Hadamard matrix of size $n\geq 4$ as in~\Cref{def:srht}  and $\D_n =\diag(\bupsilon)\in \RR^{n \times n}$ with $\bupsilon  \in \RR^{n}$ a Rademacher random vector.
Then, we have, for a matrix $\X\in\RR^{n\times d}$ with $n\geq d$ and $  t\geq \max\{\|\X\X^\top\|_F\sqrt{\log(2n/\delta)/(cn^2)},\|\X\|^2\log(2n/\delta)/(cn)\}$ that
\begin{align*}
    \Pr\left( \left|\frac{\|\ee_i^\top\H_n\D_n\X\|^2}{n}-\frac{\|\X\|_F^2}{n}  \right|\geq t,1\leq i\leq n\right)\leq \delta,
\end{align*}
where $\|\cdot\|_F$ denotes the Frobenius norm, and $c>0$ is a universal constant. 
\end{lemma}

\begin{proof}[Proof of~\Cref{lemm: HD_balance_row_norms}]
Recall from \Cref{def:srht} that both $\H_n$ and $\D_n$ are orthogonal matrices such that $\H_n^\top \D_n^2\H_n/n=\I_n$.  
Fix a row index $i \in\{1,\ldots,n\}$ and consider
\begin{align*}
  \frac{\|\ee_i^\top\H_n\D_n\X\|^2}{n}=\|\bupsilon^\top\E\X\|^2,
\end{align*}
where  $\E=\diag(\ee_i^\top\H_n/\sqrt{n})$ is a diagonal matrix formed from the $i^{th}$ row of $\H_n/\sqrt{n}$.  
Observe that $\E^2=\frac{1}{n}\I_n$, we further have
\begin{align*}  
    \EE[\bupsilon^{\top}\E\X\X^\top\E\bupsilon]=\tr(\EE[\X^\top\E\bupsilon\bupsilon^{\top}\E\X])=\tr(\X^\top\E^2\X)=\frac{\|\X\|^2_F}{n}.
\end{align*}
Note that 
\begin{align*}
& \|\E\X\X^\top\E\|=\|\X^\top\E^2\X\|=\frac{\|\X\|^2}{n},\\
 &\|\E\X\X^\top\E\|^2_F=\tr(\E\X\X^\top\E^2\X\X^\top\E)=\tr(\X\X^\top\E^2\X\X^\top\E^2)=\frac{\|\X\X^\top\|^2_F}{n^2},
\end{align*}
and 
\begin{equation*}
    K= \max\limits_{1 \leq i \leq n} \inf \{s > 0 : \EE[\exp(\upsilon_i^2 / s ^2) ]\leq 2\}=(\log (2))^{-1/2}>1,
\end{equation*}
where $\upsilon_i$ is the $i^{th}$ variable of $\bupsilon$.
Apply Hanson-Wright inequality in~\Cref{theo: hanson_wright_inequality} with $ t\geq \max\{\log(2)\|\X\X^\top\|_F\sqrt{\log(2n/\delta)/(cn^2)},\log(2)\|\X\|^2\log(2n/\delta)/(cn)\}$,
 for each $i=1,\ldots,n$, we have 
\begin{align*}
    \Pr\left(\left|\frac{\|\ee_i^\top\H_n\D_n\X\|^2}{n}-\frac{\|\X\|_F^2}{n} \right|   \geq t \right)\leq \frac{\delta}{n}.
\end{align*}
Taking a union bound over these $n$ events, we conclude the proof of \Cref{lemm: HD_balance_row_norms}.
\end{proof}

\section{Proof of the Results in \Cref{sec:application}}\label{sec:proof_application}

In this section, we present proof of the results in \Cref{sec:application}.
Precisely, to establish convergence guarantees for the bias-corrected SSN iteration in \eqref{eq:newton-sub-sampled}, we first introduce in \Cref{prop:sec.moment} a fine-grained non-asymptotic bound on second moment of the normalized sub-sampled inverse matrix in \Cref{prop:sec.moment} and present its proof in \Cref{sec:proof_sec_moment}. 
Stemming from this second inverse moment analysis, we also provide non-asymptotic bounds on the second inverse moment for the exact and/or approximate leverage sampling and SRHT in \Cref{coro:sec.moment_lev} and \Cref{coro:sec.moment_SRHT}, respectively, under scalar debiasing.
Then, in \Cref{subsec:proof_convergence}, we provide the detailed proof of \Cref{theom:RS-convergence}. 
Finally, further discussions and additional results related to \Cref{theom:RS-convergence} are provided in \Cref{subsec:additional_SSN}.

\begin{proposition}[Fine-grained analysis of second inverse moment]\label{prop:sec.moment} 
For a given  matrix $\A\in \RR^{n\times d}$, let $\S$ be a random sampling  matrix  with number of trials $m$ and importance sampling distribution $\{ \pi_i \}_{i=1}^n$ as in \Cref{def:RS}, and let $\C\in \RR^{d\times d}$ be a p.s.d.\@ matrix and $\C_{\A}=(\A^\top\A+\C)^{-1/2}\C(\A^\top\A+\C)^{-1/2} \in \RR^{d \times d}$. 
Define  $d_{\eff}=\tr(\A_\C^\top\A_\C)$ with $\A_\C \equiv \A(\A^\top \A+\C)^{-1/2} \in \RR^{n \times d}$ 
and the fine-grained de-biased sampling matrix $\check{\S} =\diag\left\{ \sqrt{ m/(m - \ell^\C_{i_s}/\pi_{i_s} ) }\right\}_{s=1}^m  \cdot \S $ as in \Cref{prop:debias}.
Then, for diagonal matrix $\bar \F=\diag\{\bar F_{ii}\}^n_{i=1}$ with
\begin{equation*}
    \bar F_{ii}= \frac{{{\ba_{\C}}_i}^\top\EE_{\zeta}[(\A_{\C}^\top \check{\S}^\top \check{\S}\A_{\C}+ \C_{\A} )^{-2}]{\ba_{\C}}_i}{m\pi_i},
\end{equation*}
and  ${\ba_{\C}}_i^\top \in\RR^d$  the $i^{th}$ row of $\A_{\C}$, there exists universal constant $C > 0$ independent of $n,d_{\eff}$ such that for $m\geq C\rho_{\max} d_{\eff}(\log(d_{\eff}/\delta )+ (\log n)^{2/3} /(\epsilon\log (\log n))^{2/3} )$ with $\delta\leq m^{-3}$ and max  factor $\rho_{\max} = \max_{1\leq i \leq n}\ell^{\C}_i/(\pi_id_{\eff})$ in \Cref{def:approx_factor}, $(\A_{\C}^\top  \check{\S}^\top \check{\S}\A_{\C} + \C_{\A} )^{-2}$ is  an $(\epsilon,\delta)$-unbiased estimator of $\I_{d}+{\A_{\C}}^\top\bar \F{\A_{\C}}$. 
\end{proposition}
Note that \Cref{prop:sec.moment} can be seen as a second-order extension of (the first-order inversion bias in)  \Cref{prop:debias}.
As we shall see below in \Cref{subsec:proof_convergence}, this result is instrumental in establishing the convergence rate for SSN in \Cref{theom:RS-convergence}. 

\subsection{Proof of \Cref{prop:sec.moment}}
\label{sec:proof_sec_moment}

Following the methodology of the proofs of \Cref{theo:inverse-bias} and \Cref{prop:debias}, the proof of \Cref{prop:sec.moment} also comes in the following two steps:
\begin{enumerate}
  \item construct an high probability event $\zeta$ as in \eqref{eq:events}; and
  \item conditional on the event $\zeta$, derive a bound for the spectral norm $\|\EE_{\zeta}[(\A_{\C}^\top  \check{\S}^\top \check{\S} \A_{\C} + \C_{\A} )^{-2}] - (\I_d+\A_{\C}^\top\bar \F\A_{\C}) \|$,  using again ``leave-one-out'' type analysis.
\end{enumerate}

First, let us recall some notations from the proofs of \Cref{theo:inverse-bias} and \Cref{prop:debias}.  
For the ease of further use, denote 
\begin{equation*}
    \check\S= \diag \left\{F_{i_si_s} \right\}^m_{s=1}\cdot\S, \quad  i_s\in \{1,\ldots,n\}, \quad F_{i_si_s}=\sqrt{m/(m-\ell_{i_s}^\C/\pi_{i_s})},
\end{equation*}
$\H=\A^\top\A+\C$,    $\A_\C=\A\H^{-1/2}$, and  $\hat\x^\top_s=\ee^\top_{i_s}\A_{\C}/\sqrt{\pi_{i_s}}$ such that $\EE[\hat\x_s\hat\x_s^\top]=\A_{\C}^\top \A_{\C}$. 
Further let
\begin{equation*}
    \hat\Q=(\A_{\C}^\top  \check{\S}^\top\check{\S}\A_{\C} + \C_{\A} )^{-1}= \left(\sum^m_{s=1}\frac{1}{m}F_{i_si_s}\hat\x_s\hat\x_s^\top + \C_{\A} \right)^{-1},~\text{and}~\hat\Q_{-s}= \left(\sum^m_{j\neq s}\frac{1}{m} F_{jj}\hat\x_j\hat\x_j^\top + \C_{\A} \right)^{-1}.
\end{equation*}

To prove \Cref{prop:sec.moment}, we first rewrite 
\begin{align}\label{eq:debq^2-i-zz}
\EE_{\zeta}[\hat\Q^2] - (\I_d+\A_{\C}^\top\bar \F\A_{\C}) = \underbrace{\EE_{\zeta}[\hat\Q-\I_d]}_{\TT_1} + \EE_{\zeta}[\hat\Q(\hat\Q-\I_d)]-\A_{\C}^\top\bar\F\A_{\C}.
\end{align}
Then,  along with $\A_{\C}^\top\A_{\C}+\C_{\A}=\I_d$ and let $\hat\gamma_s=1+\frac{1}{m}F_{i_si_s}\hat\x_s^\top\hat\Q_{-s}\hat\x_s=\check \gamma_s$, $s=1,\ldots, m$, we get
\begin{align}\label{eq:dbq(bq-i)}
&\EE_{\zeta}[\hat\Q(\hat\Q-\I_d)]
=\EE_{\zeta}[\hat\Q(\hat\Q(\A_{\C}^\top\A_{\C}-\A_{\C}^\top\check\S^\top\check\S\A_{\C}))]=\EE_{\zeta}[\hat\Q^2\A_{\C}^\top\A_{\C}]-\EE_{\zeta}\left[\sum^m_{s=1}\hat\Q\frac{F_{i_si_s}\hat\Q_{-s}\hat\x_s\hat\x_s^\top}{m\hat\gamma_s}\right]\nonumber\\
&=\EE_{\zeta}[\hat\Q^2\A_{\C}^\top\A_{\C}]-\EE_{\zeta}\left[\sum^m_{s=1}\frac{F_{i_si_s}\hat\Q^2_{-s}\hat\x_s\hat\x_s^\top}{m\check \gamma_s}\right]+\EE_{\zeta}\left[\sum^m_{s=1}\frac{F_{i_si_s}^2\hat\Q_{-s}\hat\x_s\hat\x_s^\top\hat\Q^2_{-s}\hat\x_s\hat\x_s^\top}{m^2\check \gamma^2_s}\right]\nonumber\\ 
&=\EE_{\zeta}[\hat\Q^2\A_{\C}^\top\A_{\C}]-\EE_{\zeta}\left[\frac{F_{i_si_s}\hat\Q^2_{-s}\hat\x_s\hat\x_s^\top}{\check \gamma_s}\right]+\EE_{\zeta}\left[\frac{F_{i_si_s}^2\hat\Q_{-s}\hat\x_s\hat\x_s^\top\hat\Q^2_{-s}\hat\x_s\hat\x_s^\top}{m\check \gamma^2_s}\right]\nonumber\\ &=\underbrace{\EE_{\zeta}\left[\hat\Q^2\A_{\C}^\top\A_{\C}-\hat\Q^2_{-s}\hat\x_s\hat\x_s^\top\right]}_{\TT_2}+\underbrace{\EE_{\zeta}\left[\hat\Q^2_{-s} \left(1-\frac{F_{i_si_s}}{\check \gamma_s} \right)\hat\x_s\hat\x_s^\top\right]}_{\TT_3}+\EE_{\zeta}\left[\frac{F_{i_si_s}^2\hat\Q_{-s}\hat\x_s\hat\x_s^\top\hat\Q^2_{-s}\hat\x_s\hat\x_s^\top}{m\check \gamma^2_s}\right],
\end{align}
where the second and third equalities follow from the Sherman--Morrison formula. 
Further define ${\bar \F}^{'}=\diag\{{\bar F}^{'}_{ii}\}^n_{i=1}$ with ${\bar F}^{'}_{ii}={\ba_{\C}}^\top_i\EE_{\zeta}[ \hat\Q_{-s}^{2}]{\ba_{\C}}_i/m\pi_i$. 
Using \eqref{eq:debq^2-i-zz} and \eqref{eq:dbq(bq-i)},  we get
\begin{align} 
\EE_{\zeta}[\hat\Q^2] - (\I_d+\A_{\C}^\top\bar\F\A_{\C}) &=\TT_1+\TT_2+\TT_3+\EE_{\zeta}\left[\frac{F_{i_si_s}^2\hat\Q_{-s}\hat\x_s\hat\x_s^\top\hat\Q^2_{-s}\hat\x_s\hat\x_s^\top}{m\check \gamma^2_s}\right]- \A_{\C}^\top\bar\F\A_{\C} \nonumber\\
&=\TT_1+\TT_2+\TT_3+\underbrace{\EE_{\zeta}[\hat\Q-\I_d]\A_{\C}^\top\bar\F\A_{\C}}_{\TT_4}+\underbrace{\EE_{\zeta}[(\hat\Q_{-s}-\hat\Q)\A_{\C}^\top\bar\F\A_{\C}]}_{\TT_5} \nonumber\\
&+\underbrace{\EE_{\zeta}[\hat\Q_{-s}(\A_{\C}^\top\bar\F^{'}\A_{\C}-\A_{\C}^\top\bar\F\A_{\C})]}_{\TT_6} \nonumber \\ 
&+ \underbrace{\EE_{\zeta}\left[\frac{\hat\x_s^\top\hat\Q^2_{-s}\hat\x_s\hat\Q_{-s}\hat\x_s\hat\x_s^\top}{m}\right]-\EE_{\zeta}[\hat\Q_{-s}(\A_{\C}^\top\bar\F^{'}\A_{\C})]}_{\TT_7} \nonumber\\
& +  \underbrace{\EE_{\zeta}\left[\left(\frac{F_{i_si_s}^2}{\check \gamma^2_s}-1\right)\frac{\hat\x_s^\top\hat\Q^2_{-s}\hat\x_s\hat\Q_{-s}\hat\x_s\hat\x_s^\top}{m}\right]}_{\TT_8}\nonumber.
\end{align}

At this point, note from \Cref{prop:debias} that $\|\TT_1\|=O\left(\sqrt{(\log n)^2\rho_{\max}^3d^3_{\eff}/((\log (\log n))^2m^3)}\right)$. Using $\hat\Q=\H^{1/2}\check\Q\H^{1/2}$, and noting the fact that  $\zeta'$ implies that   
\begin{align}\label{eq:breve_q_I}
    \check\Q\preceq \check\Q_{-s}\preceq6\H^{-1},   
\end{align}
we have $\bar F_{ii}\leq {\ba_{\C}}^\top_i\EE_{\zeta}[\hat\Q^{2}]{\ba_{\C}}_i/m\pi_i\leq36\rho_{\max} d_{\eff}/m<1$,
which, together with \Cref{prop:debias} yields that  $\|\TT_4\|=O\left(\sqrt{(\log n)^2\rho_{\max}^3d^3_{\eff}/((\log (\log n))^2m^3)}\right)$.

Now, we proceed to bound $\TT_2$. We first write
\begin{align}
    \|\TT_2\|
    &= \|\EE_{\zeta}[\hat\Q^2\A_{\C}^\top\A_{\C}]-\EE_{\zeta}[\hat\Q^2_{-s}\A_{\C}^\top\A_{\C}]\|+\|\EE_{\zeta}[\hat\Q^2_{-s}\A_{\C}^\top\A_{\C}]-
\EE_{\zeta}[\hat\Q^2_{-s}\hat\x_s\hat\x_s^\top]\| \nonumber\\
    &\leq\underbrace{\|\EE_{\zeta}[(\hat\Q^2-\hat\Q^2_{-s})\A_{\C}^\top\A_{\C}]\|}_{T_{21}}+\underbrace{\|
\EE_{\zeta}[\hat\Q^2_{-s}(\A_{\C}^\top\A_{\C}-\hat\x_s\hat\x_s^\top)]\|}_{T_{22}} \nonumber.
\end{align}
Recalling $\A_{\C}^\top\A_{\C}=\H^{-1/2}\A^{\top}\A\H^{-1/2} \preceq \I_d$, we get,  for  the first term $T_{21}$, that
\begin{align} \label{eq:T_21}
T_{21}&\leq  \|\EE_{\zeta}[\hat\Q^2-\hat\Q^2_{-s}]\| \leq \|\EE_{\zeta}[\hat\Q(\hat\Q-\hat\Q_{-s})]\|+\|\EE_{\zeta}[(\hat\Q-\hat\Q_{-s})\hat\Q_{-s}]\|\nonumber\\
& \overset{(a)}{=}\left \| \EE_{\zeta}\left[\frac{F_{i_si_s}}{m\check \gamma_s}\hat\Q\hat\Q_{-s}\hat\x_s\hat\x_s^\top\hat\Q_{-s}\right]\right\| + \left\|\EE_{\zeta}\left[\frac{F_{i_si_s}}{m\check \gamma_s}\hat\Q_{-s}\hat\x_s\hat\x_s^\top\hat\Q^2_{-s}\right]\right\|\nonumber\\
& \overset{(b)}{=} 2\left\|\EE_{\zeta}\left[\frac{F_{i_si_s}}{m\check \gamma_s}\hat\Q^2_{-s}\hat\x_s\hat\x_s^\top\hat\Q_{-s}\right]\right\|+\left \|\EE_{\zeta}\left[\frac{F^2_{i_si_s}}{m^2\check \gamma^2_s}\hat\Q_{-s}\hat\x_s\hat\x_s^\top\hat\Q^2_{-s}\hat\x_s\hat\x_s^\top\hat\Q_{-s}\right]\right\|,
\end{align}
where we used (again) Sherman--Morrison formula twice in the step $(a)$ and $(b)$, respectively.
For the first term in \eqref{eq:T_21}, using the fact that $F_{i_si_s}/\check \gamma_s<F_{i_si_s}<2$, together with \eqref{eq:condition_expec_three} and \eqref{eq:breve_q_I}, we get
 \begin{align}
 \left\|\EE_{\zeta}\left[\frac{F_{i_si_s}}{m\check \gamma_s}\hat\Q^2_{-s}\hat\x_s\hat\x_s^\top\hat\Q_{-s}\right]\right\| &=\sup\limits_{\|\uu\|=1,\|\vv\|=1} \EE_{\zeta}\left[\frac{F_{i_si_s}}{m\check \gamma_s}\uu^\top\hat\Q^2_{-s}\hat\x_s\hat\x_s^\top\hat\Q_{-s}\vv\right]  \nonumber \\ 
 &\leq \frac{4}{m} \sup\limits_{\|\uu\|=1,\|\vv\|=1} \EE_{\zeta^{'}}\left[|\uu^\top\hat\Q^2_{-s}\hat\x_s\hat\x_s^\top\hat\Q_{-s}\vv|\right] \nonumber\\
 &\leq \frac{2}{m} \sup\limits_{\|\uu\|=1,\|\vv\|=1} \EE_{\zeta^{'}}\left[\uu^\top\hat\Q^2_{-s}\hat\x_s\hat\x_s^\top\hat\Q^2_{-s}\uu+\vv^\top\hat\Q_{-s}\hat\x_s
 \hat\x_s^\top\hat\Q_{-s}\vv\right] \nonumber\\
  &\leq \frac{2}{m}\sup\limits_{\|\uu\|=1,\|\vv\|=1} \EE_{\zeta^{'}}\left[\sum^n_{j=1}\uu^\top\hat\Q^2_{-s}\ba_{\C_j}\ba_{\C_j}^\top\hat\Q^2_{-s}\uu+\vv^\top\hat\Q_{-s}\ba_{\C_j}
\ba_{\C_j}^\top\hat\Q_{-s}\vv\right] \nonumber\\
  &\leq \frac{2}{m}\sup\limits_{\|\uu\|=1,\|\vv\|=1} \EE_{\zeta^{'}}\left[\uu^\top\hat\Q^2_{-s}\A_{\C}^\top\A_{\C}\hat\Q^2_{-s}\uu+\vv^\top\hat\Q_{-s}\A_{\C}^\top\A_{\C}\hat\Q_{-s}\vv\right] \nonumber \\ 
  &= O\Big(\frac{1}{m}\Big).\nonumber
 \end{align}
Analogously as above, we have, for the second term in \eqref{eq:T_21} that
 \begin{align}
     \left \|\EE_{\zeta}\left[\frac{F^2_{i_si_s}}{m^2\check \gamma^2_s}\hat\Q_{-s}\hat\x_s\hat\x_s^\top\hat\Q^2_{-s}\hat\x_s\hat\x_s^\top\hat\Q_{-s}\right]\right\|&\leq  \frac{8}{m^2} \|\EE_{\zeta^{'}}[\hat\Q_{-s}\hat\x_s\hat\x_s^\top\hat\Q^2_{-s}\hat\x_s\hat\x_s^\top\hat\Q_{-s}]\|\nonumber\\
    &=\frac{8}{m^2} \left\|\EE_{\zeta^{'}}\left[\sum^n_{j=1}\frac{\hat\Q_{-s}{\ba_{\C}}_j{\ba_{\C}}^\top_j\hat\Q^2_{-s}{\ba_{\C}}_j{\ba_{\C}}_j^\top\hat\Q_{-s}}{\pi_j}\right]\right\|\nonumber\\
    &=\frac{288\rho_{\max} d_{\eff}}{m^2}  \left\|\EE_{\zeta^{'}}\left[\hat\Q_{-s}\A_{\C}^\top\A_{\C}\hat\Q_{-s}\right]\right\| = O\Big(\frac{\rho_{\max} d_{\eff}}{m^2}\Big).\nonumber
 \end{align}
We thus conclude that   $T_{21}=  O(1/m)$.

Following the methodology used to bound $\|\H^{1/2}\check{\Z}_2\H^{-1/2}\|$ and $\|\H^{1/2}\check{\Z}_3\H^{-1/2}\|$ in \Cref{subsec:detailed_proof_prop:debias}, together with \eqref{eq:breve_q_I}, we get similarly 
\begin{equation*}
    T_{22}=O \left(\sqrt{\frac{\rho_{\max}^3d^3_{\eff}}{m^3}} \right), \quad \|\TT_3\|=O \left( \frac{\log n}{\log (\log n)}\sqrt{\frac{\rho_{\max}^3d^3_{\eff}}{m^3}} \right).
\end{equation*}
Utilizing the fact that $\bar F_{ii}\leq 36\rho_{\max} d_{\eff}/m$, together with \eqref{eq:Q_diffe_hat} and \eqref{eq:breve_q_I}, it follows that 
\begin{equation*}
    \|\TT_5\|\leq \|\EE_{\zeta}[\hat\Q_{-s}-\hat\Q] \|\|\A_{\C}^\top\bar\F\A_{\C}\|\leq O\left(\frac{\rho_{\max} d_{\eff}}{m^2}\right).
\end{equation*}
Since \eqref{eq:breve_q_I} and  $T_{21}=O(1/m)$, we can bound $\|\TT_6\|$ as
\begin{align}
    \|\TT_6\|&\leq\|\EE_{\zeta}[\hat\Q_{-s}]\|\|\A_{\C}\|^2\|\bar\F^{'}-\bar\F\|\leq 6\mathop{\max}\limits_{1\leq  j\leq  n}\frac{|{\ba_{\C}}^\top_j (\EE_{\zeta}[\hat\Q^2_{-s}]-\EE_{\zeta}[\hat\Q^2]){\ba_{\C}}_j|}{m\pi_j}\nonumber\\
    &\leq \frac{6\rho_{\max} d_{\eff}}{m}\|\EE_{\zeta}[\hat\Q^2_{-s}]-\EE_{\zeta}[\hat\Q^2]\| = O\Big(\frac{\rho_{\max} d_{\eff}}{m^2}\Big).\nonumber
\end{align}
We then move on to bound $\|\TT_7\|$. 
We start by rewriting
\begin{align}
    \|\TT_7\|&=\left\|\EE_{\zeta}\left[\hat\Q_{-s}\left(\A_{\C}^\top\bar\F^{'}\A_{\C}-\frac{\hat\x_s^\top\hat\Q^2_{-s}\hat\x_s\hat\x_s\hat\x_s^\top}{m}\right)\right] \right\|\leq \underbrace{\left\|\EE_{\zeta}\left[\hat\Q_{-s}\left(\A_{\C}^\top\bar\F^{'}\A_{\C}-\frac{\hat\x_s^\top\EE_{\zeta}[\hat\Q^2_{-s}]\hat\x_s\hat\x_s\hat\x_s^\top}{m}\right)\right] \right\|}_{T_{71}}\nonumber\\
    &+\underbrace{\left\|\EE_{\zeta}\left[\hat\Q_{-s}\frac{\hat\x_s^\top(\EE_{\zeta'}[\hat\Q^2_{-s}]-\hat\Q^2_{-s})\hat\x_s\hat\x_s\hat\x_s^\top}{m}\right] \right\|}_{T_{72}}+\underbrace{\left\|\EE_{\zeta}\left[\hat\Q_{-s}\frac{\hat\x_s^\top(\EE_{\zeta}[\hat\Q^2_{-s}]-\EE_{\zeta'}[\hat\Q^2_{-s}])\hat\x_s\hat\x_s\hat\x_s^\top}{m}\right] \right\|}_{T_{73}}. \nonumber
\end{align}
Note that $\EE_{\zeta^{'}}\Big[\hat\Q_{-s}\Big(\A_{\C}^\top\bar\F^{'}\A_{\C}-\frac{\hat\x_s^\top\EE_{\zeta}[\hat\Q^2_{-s}]\hat\x_s\hat\x_s\hat\x_s^\top}{m}\Big)\Big]=0$. By adapting the  techniques of bounding   $\|\H^{1/2}\check{\Z}_2\H^{-1/2}\|$ in \Cref{subsec:detailed_proof_prop:debias}, for  $\delta_3<m^{-3}$,  we further get
\begin{align}
    T_{71}&\leq 2\left\|\EE_{\zeta'}\left[\hat\Q_{-s}\left(\A_{\C}^\top\bar\F^{'}\A_{\C}-\frac{\hat\x_s^\top\EE_{\zeta}[\hat\Q^2_{-s}]\hat\x_s\hat\x_s\hat\x_s^\top}{m}\right) \cdot \mathbf{1}_{\neg\zeta_3}\right]\right\|\nonumber\\
    &\leq O\left(\frac{\rho_{\max}  d_{\eff}}{m^4}\right)+12 \EE_{\zeta'}\left[\frac{\hat\x_s^\top\EE_{\zeta}[\hat\Q^2_{-s}]\hat\x_s\hat\x_s^\top\hat\x_s}{m} \cdot \mathbf{1}_{\neg\zeta_3}\right]. \nonumber
\end{align}
Furthermore, applying the Chebyshev's inequality again, and considering  $\EE_{\zeta'}[\hat\x_s^\top\EE_{\zeta}[\hat\Q^2_{-s}]\hat\x_s\hat\x_s^\top\hat\x_s/m]\leq 36\rho_{\max} d^2_{\eff}/m$, alongside 
\begin{align}
    \Var_{\zeta'}\left[\frac{\hat\x_s^\top\EE_{\zeta}[\hat\Q^2_{-s}]\hat\x_s\hat\x_s^\top\hat\x_s}{m}
    \right]\leq \EE_{\zeta'}\left[\frac{(\hat\x_s^\top\EE_{\zeta}[\hat\Q^2_{-s}]\hat\x_s\hat\x_s^\top\hat\x_s)^2}{m^2}\right]\leq \frac{6^4\rho_{\max}^3d^4_{\eff}}{m^2},\nonumber
\end{align}
 it follows that, for  $x\geq 72\rho_{\max} d^2_{\eff}/m$, $ \Pr(\hat\x_s^\top \EE_{\zeta}[\hat\Q^2_{-s}]\hat\x_s\hat\x_s^\top\hat\x_s/m\geq x~|~\zeta')\leq 
 6^4\rho_{\max}^3 d^4_{\eff}/(m^2x^2)$.
Subsequently, by $\delta_3<m^{-3}$,  we deduce that
 \begin{align}
    \EE_{\zeta'}\left[\frac{\hat\x_s^\top\EE_{\zeta}[\hat\Q^2_{-s}]\hat\x_s\hat\x_s^\top\hat\x_s}{m} \cdot \mathbf{1}_{\neg\zeta_3}\right]&=\int^\infty_0 \Pr\left(\frac{\hat\x_s^\top \EE_{\zeta}[\hat\Q^2_{-s}]\hat\x_s\hat\x_s^\top\hat\x_s}{m} \cdot \mathbf{1}_{\neg\zeta_3}\geq x  ~\Big|~\zeta'\right)dx \nonumber\\
    &\leq 2m^{2}\delta_3+\int^\infty_{2m^{2}}\Pr\left(\frac{\hat\x_s^\top \EE_{\zeta}[\hat\Q^2_{-s}]\hat\x_s\hat\x_s^\top\hat\x_s}{m}\geq x~\Big|~\zeta'\right)dx\nonumber\\
    &\leq \frac{2}{m}+ \frac{6^4\rho_{\max}^3 d^4_{\eff}}{m^2}\int^\infty_{2m^{2}}\frac{1}{x^2}dx= \frac{2}{m}+\frac{6^4\rho_{\max}^3 d^4_{\eff}}{2m^{4}} = O\left(\sqrt{\frac{\rho_{\max}^3 d^3_{\eff}}{m^3}} \right).\nonumber
 \end{align}
This leads to $T_{71}=O\left(\sqrt{\rho_{\max}^3 d^3_{\eff}/m^3}\right)$.

Next, we show a bound on $T_{72}$.
Using again \eqref{eq:condition_expec_three} and \eqref{eq:breve_q_I}, we get
\begin{align}
    T_{72}&\leq\sup\limits_{\|\uu\|=1,\|\vv\|=1} \EE_{\zeta}\left[\left|\hat\x_s^\top(\EE_{\zeta'}[\hat\Q^2_{-s}]-\hat\Q^2_{-s})\hat\x_s\right |\left|\frac{\uu^\top\hat\Q_{-s}\hat\x_s\hat\x_s^\top\vv}{m}\right|\right]\nonumber\\
    &\leq2\sup\limits_{\|\uu\|=1,\|\vv\|=1} \EE_{\zeta^{'}}\left[\left|\hat\x_s^\top(\EE_{\zeta'}[\hat\Q^2_{-s}]-\hat\Q^2_{-s})\hat\x_s\right |\left|\frac{\uu^\top\hat\Q_{-s}\hat\x_s\hat\x_s^\top\vv}{m}\right|\right]\nonumber\\
    &\leq\frac{1}{m}\sup\limits_{\|\uu\|=1,\|\vv\|=1} \EE_{\zeta^{'}}\left[\left|\hat\x_s^\top(\EE_{\zeta'}[\hat\Q^2_{-s}]-\hat\Q^2_{-s})\hat\x_s\right |\left(\uu^\top\hat\Q_{-s}\hat\x_s\hat\x_s^\top\hat\Q_{-s}\uu+\vv^\top\hat\x_s\hat\x_s^\top\vv\right)\right]\nonumber\\
    &\leq\frac{1}{m}\sup\limits_{\|\uu\|=1,\|\vv\|=1} \EE_{\zeta^{'}}\left[\sum^n_{j=1}\frac{|{\ba_{\C}}_j^\top(\EE_{\zeta'}[\hat\Q^2_{-s}]-\hat\Q^2_{-s}){\ba_{\C}}_j|}{\pi_j}\left(\uu^\top\hat\Q_{-s}{\ba_{\C}}_j{\ba_{\C}}_j^\top\hat\Q_{-s}\uu+\vv^\top{\ba_{\C}}_j{\ba_{\C}}_j^\top\vv\right)\right]\nonumber\\
    &\leq\frac{1}{m}\sup\limits_{\|\uu\|=1} \EE_{\zeta^{'}}\left[\max\limits_{1\leq j\leq n}\frac{|{\ba_{\C}}_j^\top(\EE_{\zeta'}[\hat\Q^2_{-s}]-\hat\Q^2_{-s}){\ba_{\C}}_j|}{\pi_j}\sum^n_{j=1}\uu^\top\hat\Q_{-s}{\ba_{\C}}_j{\ba_{\C}}_j^\top\hat\Q_{-s}\uu\right]\nonumber\\
    &+\frac{1}{m}\sup\limits_{\|\vv\|=1} \EE_{\zeta^{'}}\left[\max\limits_{1\leq j\leq n}\frac{|{\ba_{\C}}_j^\top(\EE_{\zeta'}[\hat\Q^2_{-s}]-\hat\Q^2_{-s}){\ba_{\C}}_j|}{\pi_j}\sum^n_{j=1}\vv^\top{\ba_{\C}}_j{\ba_{\C}}_j^\top\vv\right]\nonumber\\
    &\leq\EE_{\zeta^{'}}\left[\max\limits_{1 \leq j \leq n}\frac{37|{\ba_{\C}}_j^\top(\EE_{\zeta'}[\hat\Q^2_{-s}]-\hat\Q^2_{-s}){\ba_{\C}}_j|}{m\pi_j}\right].\nonumber
\end{align}
Observe that   adapting the proof of \Cref{lemma:max_expec}, we can readily ascertain the following:
\begin{align*}
   \EE_{\zeta^{'}}\left[\max\limits_{1 \leq j \leq n}\frac{37|{\ba_{\C}}_j^\top(\EE_{\zeta'}[\hat\Q^2_{-s}]-\hat\Q^2_{-s}){\ba_{\C}}_j|}{m\pi_j}\right]=  O\left(\frac{\log n}{\log (\log n)}\sqrt{\frac{\rho_{\max}^3 d^3_{\eff}}{m^3}}\right),
\end{align*}
so that $ T_{72}=O\left(\sqrt{(\log n)^2\rho_{\max}^3 d^3_{\eff}/((\log (\log n))^2m^3)}\right)$.
Recalling \eqref{eq:breve_q_I} again, we derive
\begin{align}
  T_{73}& \leq 12\EE_{\zeta'}\left[\left|\frac{\hat\x_s^\top(\EE_{\zeta}[\hat\Q^2_{-s}]-\EE_{\zeta'}[\hat\Q^2_{-s}])\hat\x_s\hat\x_s^\top\hat\x_s}{m}\right|\right] 
  \leq   12\sum^n_{j=1}\left|\frac{{\ba_{\C}}_j^\top(\EE_{\zeta}[\hat\Q^2_{-s}]-\EE_{\zeta'}[\hat\Q^2_{-s}]){\ba_{\C}}_j{\ba_{\C}}_j^\top{\ba_{\C}}_j}{m\pi_j}\right| \nonumber\\
    &\leq \frac{12\rho_{\max} d_{\eff}}{m}\sum^n_{j=1} \left|{\ba_{\C}}_j^\top(\EE_{\zeta}[\hat\Q^2_{-s}]-\EE_{\zeta'}[\hat\Q^2_{-s}]){\ba_{\C}}_j \right|. \nonumber
\end{align}
Along with 
\begin{align}
   \EE_{\zeta}[\hat\Q^2_{-s}]-\EE_{\zeta'}[\hat\Q^2_{-s}]=\frac{\delta_3}{1-\delta_3}\EE_{\zeta'}[\hat\Q^2_{-s}] -\frac{\delta_3}{1-\delta_3}\EE_{\zeta'}[\hat\Q^2_{-s} ~|~ \mathbf{1}_{\neg\zeta_3}],\nonumber
\end{align}
we obtain, for   $\delta_3\leq m^{-3}$,
\begin{align}
   &|{\ba_{\C}}_j^\top(\EE_{\zeta}[\hat\Q^2_{-s}]-\EE_{\zeta'}[\hat\Q^2_{-s}]){\ba_{\C}}_j|\leq 2\delta_3 {\ba_{\C}}_j^\top {\ba_{\C}}_j(\|\EE_{\zeta'}[\hat\Q^2_{-s}]\|+\|\EE_{\zeta'}[\hat\Q^2_{-s}~|~\mathbf{1}_{\neg\zeta_3}]\|)\leq \frac{144{\ba_{\C}}_j^\top {\ba_{\C}}_j}{m^3}.\nonumber
\end{align}
This   results in 
$ T_{73}=O(\rho_{\max} d^2_{\eff}/m^4)$.

Now, it remains to bound $\|\TT_8\|$.  
Noting  \eqref{eq:breve_q_I} and  $\frac{F_{i_si_s}}{\check \gamma_s}\leq 2$, it follows that  
\begin{align}
    \|\TT_8\|&\leq\sup\limits_{\|\uu\|=1,\|\vv\|=1} \EE_{\zeta}\left[\left|\frac{F_{i_si_s}^2}{\check \gamma^2_s}-1\right|\left|\frac{\uu^\top\hat\x_s^\top\hat\Q^2_{-s}\hat\x_s\hat\Q_{-s}\hat\x_s\hat\x_s^\top\vv}{m}\right|\right]\nonumber\\
    &\leq\sup\limits_{\|\uu\|=1,\|\vv\|=1} \EE_{\zeta}\left[\left|\frac{F_{i_si_s}}{\check \gamma_s}\left(\frac{F_{i_si_s}}{\check \gamma_s}-1\right)\right |\left|\frac{\uu^\top\hat\x_s^\top\hat\Q^2_{-s}\hat\x_s\hat\Q_{-s}\hat\x_s\hat\x_s^\top\vv}{m}\right|\right]\nonumber\\
    &+\sup\limits_{\|\uu\|=1,\|\vv\|=1} \EE_{\zeta}\left[\left|\frac{F_{i_si_s}}{\check \gamma_s}-1\right |\left|\frac{\uu^\top\hat\x_s^\top\hat\Q^2_{-s}\hat\x_s\hat\Q_{-s}\hat\x_s\hat\x_s^\top\vv}{m}\right|\right]\nonumber\\
    &\leq3\sup\limits_{\|\uu\|=1,\|\vv\|=1} \EE_{\zeta}\left[\left|\frac{F_{i_si_s}}{\check \gamma_s}-1\right |\left|\frac{\uu^\top\hat\x_s^\top\hat\Q^2_{-s}\hat\x_s\hat\Q_{-s}\hat\x_s\hat\x_s^\top\vv}{m}\right|\right]\nonumber\\
     &\leq\frac{108\rho_{\max} d_{\eff}}{m}\sup\limits_{\|\uu\|=1,\|\vv\|=1} \EE_{\zeta}\left[\left|\frac{F_{i_si_s}}{\check \gamma_s}-1\right |\left|\uu^\top\hat\Q_{-s}\hat\x_s\hat\x_s^\top\vv\right|\right].\nonumber
\end{align}
Further, applying the techniques of bounding 
$ \|\H^{1/2}\check{\Z}_3\H^{-1/2}\|$ in \Cref{subsec:detailed_proof_prop:debias} again, we get  $\|\TT_8\| = O\left(\sqrt{(\log n)^2\rho_{\max}^5d^5_{\eff}/((\log (\log n))^2m^5)}\right)$. 
Putting the above together, we conclude that
\begin{equation*}
    \|  \EE_{\zeta}[\hat\Q^2] - (\I_d+\A_{\C}^\top\bar\F\A_{\C})\|=O \left( \frac{\log n}{\log (\log n)}\sqrt{\frac{\rho_{\max}^3d^3_{\eff}}{m^3}}\right).
\end{equation*}
This thus concludes the proof of \Cref{prop:sec.moment}.
\qedwhite
 
Based on the second inverse moment result in  \Cref{prop:sec.moment}, we follow the line of arguments in~\Cref{coro:inver_bias_constant_debias} and~\Cref{coro:debiasing_srht},to derive the following fine-grained second inverse moment results for exact and/or approximate leverage score sampling and SRHT using scalar debiasing. 
The proofs of these results follow from a direct combination of the proofs of \Cref{coro:inver_bias_constant_debias}, \Cref{coro:debiasing_srht}~and~\Cref{prop:sec.moment}, and are omitted here.

\begin{corollary}[Second inverse moment using scalar debiasing under approximate leverage]\label{coro:sec.moment_lev}
Under the settings and notations of \Cref{prop:sec.moment}, for any random sampling scheme with sampling distribution $\pi_i\in [\ell^{\C}_{i}/(d_{\eff}\rho_{\max}), \ell^{\C}_{i}/(d_{\eff} \rho_{\min})]$ with some $\rho_{\min} \in [1/2,1]$ as in \Cref{def:approx_factor} and diagonal matrix $\bar \F=\diag\{\bar F_{ii}\}^n_{i=1}$ with
\begin{equation*}
    \bar F_{ii}= \frac{{{\ba_{\C}}_i}^\top\EE_{\zeta}[(\frac{m}{m-d_{\eff}}\A_{\C}^\top \S^\top \S\A_{\C}+ \C_{\A} )^{-2}]{\ba_{\C}}_i}{m\pi_i}.
\end{equation*}  
Then, there exists universal constant $C > 0$ independent of $n,d_{\eff}$ such that for $m\geq C\rho_{\max} d_{\eff}\log(d_{\eff}/\delta )$ with $\delta\leq m^{-3}$, $(\frac{m}{m-d_{\eff}}\A_{\C}^\top \S^\top \S\A_{\C} + \C_{\A} )^{-2}$ is  an $(\epsilon,\delta)$-unbiased estimator of $\I_{d}+{\A_{\C}}^\top\bar \F{\A_{\C}}$ with inversion bias $\epsilon=\max\left\{\tilde O\left(\sqrt{\rho_{\max}^3d^3_{\eff}/m^3}\right),O(\epsilon_{\rho}\rho_{\max}d_{\eff}/m)\right\}$ for $\epsilon_{\rho}=\max\{\rho_{\min}^{-1}-1,1 - \rho_{\max}^{-1}\}$. 
\end{corollary}

\begin{corollary}[Second moment for SRHT using scalar debiasing ]\label{coro:sec.moment_SRHT}
Under the  settings and notations of \Cref{prop:sec.moment}, 
for $\tilde{\A}_{\rm SRHT} \in \RR^{m \times n}$ the SRHT of $\A$ as in \Cref{def:srht}, and diagonal matrix $\bar \F=\diag\{\bar F_{ii}\}^n_{i=1}$ with
\begin{equation*}
    \bar F_{ii}= \frac{\ee_i^\top\H_n\D_n\A_{\C}\EE_{\zeta}[(\frac{m}{m-d_{\eff}}\H^{-1/2}\tilde{\A}_{\rm SRHT} ^\top\tilde{\A}_{\rm SRHT}\H^{-1/2} + \C_{\A} )^{-2}]\A^\top_{\C}\D_n\H_n\ee_i}{nm\pi_i},
\end{equation*} 
then there exists universal constant $C > 0$ independent of $n,d_{\eff}$ and $n \exp(-d_{\eff}) < \delta < m^{-3} $ such that for $m\geq C\rho_{\max} d_{\eff}\log(d_{\eff}/\delta )$ with max  factor $\rho_{\max} = \max_{1\leq i \leq n}n\ell^{\C}_i(\H_n \D_n \A/\sqrt{n})/d_{\eff}$, $(\frac{m}{m-d_{\eff}}\H^{-1/2}\tilde{\A}_{\rm SRHT} ^\top \tilde{\A}_{\rm SRHT}\H^{-1/2} + \C_{\A} )^{-2}$ is  an $(\epsilon,\delta)$-unbiased estimator of $\I_{d}+\frac{1}{n}\A^\top_{\C}\D_n\H_n\bar\F\H_n\D_n\A_{\C}$ with inversion bias $\epsilon=\tilde O\left(\sqrt{\rho_{\max}^3d^3_{\eff}/m^3}\right)+O\left(\sqrt{d_{\eff}\log(n/\delta)}/m
\right)$. 
\end{corollary}

These results are extension of the results in Corollaries in \ref{coro:inver_bias_constant_debias}~and~\ref{coro:debiasing_srht}, and can be used to derive SSN local convergence rates under exact/approximate leverage and SRHT similar to \Cref{theom:RS-convergence}. 
See Corollaries~\ref{coro:ssn_exact_approximate_lev}~and~\ref{coro:ssn_SRHT} in \Cref{subsec:additional_SSN}.

\subsection{Proof of \Cref{theom:RS-convergence}} 
\label{subsec:proof_convergence}

Here, we prove \Cref{theom:RS-convergence} by adapting the proof approaches of our Propositions~\ref{prop:debias}~and~\ref{prop:sec.moment} and~\citet[Theorem~10]{derezinski2021newtonless}. 
The proof of \Cref{theom:RS-convergence} comes in the following two steps: 
\begin{itemize}
    \item we first establish \Cref{lemm:pre_convergence} that connects the de-biased SSN iterations $\check{\bbeta}_{t+1}$ to the true $\bbeta_{t+1}$; and
    \item perform detailed convergence analysis of de-biased SSN.
\end{itemize}

To begin, we define a high probability event $\zeta_t$ for each $t=0,\ldots, T-1$, as defined in \eqref{eq:events}. 
These events are established independently for each iteration, and we denote $\zeta=\bigcap^{T-1}_{t=0} \zeta_t$. 
Under the setting of \Cref{theom:RS-convergence}, it follows from \Cref{lem:sub_embed} that $\Pr(\zeta_t)\geq 1-\delta/T$ and $\Pr(\zeta)\geq 1-\delta$.  
These events will be constantly exploited in the remainder of the proof of \Cref{theom:RS-convergence}.

Proceeding to the core of our proof,  we present an auxiliary lemma, which uses Propositions~\ref{prop:debias}~and~\ref{prop:sec.moment} to connect the de-biased iteration $\check{\bbeta}_{t+1}$ to the true Newton iteration $\bbeta_{t+1}$. 
This is pivotal to prove \Cref{theom:RS-convergence}.
Note that this result is universally applicable to any de-biased $\check{\bbeta}_t$ and is independent of the smoothness assumption on $F$ in, e.g., \Cref{assum:lip_f}.
 
\begin{lemma}\label{lemm:pre_convergence}
Let $\H_t=\nabla^2f(\check{\bbeta}_{t})+\C(\check{\bbeta}_t)$ with  $\nabla^2f(\check{\bbeta}_{t})=\A(\check{\bbeta}_t)^\top \A(\check{\bbeta}_t)$  and $\C(\check{\bbeta}_t)=\nabla^2\Phi(\check{\bbeta}_t)$ a p.s.d.\@ matrix, and let $\check{\bbeta}_{t+1}$ be the de-biased SSN iteration as in \eqref{eq:de_newton-sub-sampled} with de-biased $\check \S_t$ as in \Cref{prop:debias}.  
Denote $\Delta_{t}=\bbeta_{t}-\bbeta^{\ast}$ and $\check{\Delta}_{t}=\check{\bbeta}_t-\bbeta^{\ast}$. If the exact Newton step $\bbeta_{t+1}=\check{\bbeta}_t-\mu_t\H^{-1}_t\g_t$ (with $\g_t$ the gradient of $F$ at $\check\bbeta_t$) is a descent direction, i.e., $\|\Delta_{t+1}\|_{\H_t}\leq \|\check{\Delta}_{t}\|_{\H_t}$, then, there exists universal constant $C > 0$ independent of $n, d_{\eff}$ such that for $m\geq C\rho_{\max} d_{\eff}(\log(d_{\eff}T/\delta )+ (\log n)^{2/3} /(\epsilon\log (\log n))^{2/3} ) $ with $\epsilon > 0$, $\delta\leq m^{-3}$, 
we have
 \begin{align}
   \EE_{\zeta_t}[\|\check{\Delta}_{t+1}\|^2_{\H_t}] \leq \|\Delta_{t+1}\|^2_{\H_t} +\epsilon\|\check{\Delta}_{t}\|^2_{\H_t}+\frac{36\rho_{\max}  d_{\eff}}{m}\|\bbeta_{t+1}-\check{\bbeta}_t\|^2_{\nabla^2f(\check{\bbeta}_{t})}.\nonumber
\end{align} 
Here, $\rho_{\max}$ is the max importance sampling approximation factor in \Cref{def:approx_factor} for $\ell_{i}^\C=\max_{1\leq t \leq T}\ell_{i}^\C(\check{\bbeta}_t)$ and $d_{\eff}=\max_{1\leq t \leq T}d_{\eff}(\check{\bbeta}_t)$ with $\ell_{i}^\C(\check{\bbeta}_t)$ and $d_{\eff}(\check{\bbeta}_t)$ the leverage scores and effective dimension of $\A(\check{\bbeta}_t)$ given $\C(\check{\bbeta}_t)$, respectively.
\end{lemma}

\begin{proof}[Proof of \Cref{lemm:pre_convergence}]
In both this proof and the subsequent proof of  \Cref{theom:RS-convergence}, to  simplify the notation, we abbreviate 
$\A_{\C}(\check{\bbeta}_t)$ as $\A_{\C}$.
We first recall some notations from (the proof of) \Cref{prop:sec.moment}. 
Let $\pp_t=\bbeta_{t+1}-\check{\bbeta}_t=-\mu_t\H_t^{-1}\g_t$,  and $\hat \Q_t=(\A^\top_{\C}\check{\S}^\top_t\check{\S}_t\A_\C +\C_{\A} )^{-1}$ with the de-biased sampling matrix $\check{\S}_t = \diag\left\{\sqrt{  m/(m - \ell^\C_{i_s}(\check{\bbeta}_t)/\pi_{i_s} )} \right \}_{s=1}^m \cdot \S_t$ as in \Cref{prop:debias} and  $\C_{\A} =\H_t^{-1/2}\C(\check{\bbeta}_t)\H_t^{-1/2}$.
Define the diagonal matrix $\bar \F=\diag\{\bar F_{ii}\}^n_{i=1}$ with  $\bar F_{ii}={{\ba_{\C}}_i}^\top\EE_{\zeta_t}[(\A_{\C}^\top\check{\S}^\top\check{\S}\A_{\C}+ \C_{\A} )^{-2}]{\ba_{\C}}_i/m\pi_i$ as in \Cref{prop:sec.moment} and ${\ba_{\C}}_i^\top \in\RR^d$  the $i^{th}$ row of $\A_{\C}$.
  
Building on the results from~\Cref{prop:debias} and~\Cref{prop:sec.moment}, for $\epsilon=O \left(\frac{\log n}{\log (\log n)}\sqrt{\frac{\rho_{\max}^3 d^3_{\eff}} {m^3}} \right)$,   we obtain 
    \begin{align}
       &\EE_{\zeta_t}[\|\check{\Delta}_{t+1}\|^2_{\H_t}] -\|\Delta_{t+1}\|^2_{\H_t}
       =\EE_{\zeta_t}[\|\check{\bbeta}_{t+1}-\bbeta_{t+1}\|^2_{\H_t}]+2\Delta_{t+1}^\top\H_t\EE_{\zeta_t}[\check{\bbeta}_{t+1}-\bbeta_{t+1}]\nonumber\\
        &=2\mu_{t}\Delta_{t+1}^\top\H^{\frac{1}{2}}_t\EE_{\zeta_t}[(\hat \Q_t-\I_d)]\H^{\frac{1}{2}}_t\H_t^{-1}\g_t+\mu^2_{t}\g^\top_t\H^{-1}_t\H^{\frac{1}{2}}_t\EE_{\zeta_t}[(\hat \Q_t-\I_d)^2]\H^{\frac{1}{2}}_t\H_t^{-1}\g_t\nonumber\\
    &\leq 2\|\Delta_{t+1}\|_{\H_t}\|\pp_t\|_{\H_t}\|\EE_{\zeta_t}[\hat \Q_t]-\I_d\|+\pp^\top_t\H^{\frac{1}{2}}_t\EE_{\zeta_t}[\hat \Q_t^2-\I_d-2(\hat \Q_t-\I_d)]\H^{\frac{1}{2}}_t\pp_t\nonumber\\ 
    &\leq2\|\Delta_{t+1}\|_{\H_t}\|\pp_t\|_{\H_t}\|\EE_{\zeta_t}[\hat \Q_t]-\I_d\|+2\|\pp_t\|^2_{\H_t}\|\EE_{\zeta_t}[\hat \Q_t]-\I_d\| +\pp^\top_t\H^{\frac{1}{2}}_t\EE_{\zeta_t}[\hat \Q_t^2-\I_d-\A_\C^\top\bar\F_t\A_\C]\H^{\frac{1}{2}}_t\pp_t \nonumber \\
    &+\pp^\top_t\H^{\frac{1}{2}}_t\A_\C^\top\bar\F_t\A_\C\H^{\frac{1}{2}}_t\pp_t
       \nonumber\\
       &\leq 2\epsilon\|\Delta_{t+1}\|_{\H_t}\|\pp_t\|_{\H_t} +2\epsilon\|\pp_t\|^2_{\H_t}+\|\pp_t\|^2_{\H_t}\|\EE_{\zeta_t}[\hat \Q_t^2]-\I_d-\A_\C^\top\bar\F_t\A_\C\| +\|\pp_t\|^2_{\nabla^2f(\check{\bbeta}_{t})}
    \|\bar\F_t\|
       \nonumber\\
        &\leq  \epsilon\|\Delta_{t+1}\|_{\H_t}\|\pp_t\|_{\H_t} +\epsilon\|\pp_t\|^2_{\H_t}+\|\pp_t\|^2_{\nabla^2f(\check{\bbeta}_{t})}\max\limits_{1\leq i\leq n}\frac{{{\ba_{\C}}_i}^\top\EE_{\zeta_t}[\hat \Q_t^2]{\ba_{\C}}_i}{m\pi_i}
       \nonumber\\
       &\leq  \epsilon\|\check{\Delta}_{t}\|_{\H_t}\|\pp_t\|_{\H_t} +\epsilon\|\pp_t\|^2_{\H_t}+\frac{36\rho_{\max} d_{\eff}}{m}\|\pp_t\|^2_{\nabla^2f(\check{\bbeta}_{t})},
       \nonumber
    \end{align}
    which combined with $\|\pp_t\|_{\H_t}\leq\|\Delta_{t+1}\|_{\H_t}+\|\check{\Delta}_{t}\|_{\H_t}\leq2\|\check{\Delta}_{t}\|_{\H_t}$ and $\|\check{\Delta}_{t}\|_{\H_t}\|\pp_t\|_{\H_t}\leq 2\|\check{\Delta}_{t}\|^2_{\H_t}$ leads to
    \begin{align}
     \EE_{\zeta_t}[\|\check{\Delta}_{t+1}\|^2_{\H_t}] -\|\Delta_{t+1}\|^2_{\H_t}   &\leq2 \epsilon\|\check{\Delta}_{t}\|^2_{\H_t} +4 \epsilon\|\check{\Delta}_{t}\|^2_{\H_t}+\frac{36\rho_{\max} d_{\eff}}{m}\|\pp_t\|^2_{\nabla^2f(\check{\bbeta}_{t})}\nonumber\\
     &=\epsilon\|\check{\Delta}_{t}\|^2_{\H_t} +\frac{36\rho_{\max} d_{\eff}}{m}\|\pp_t\|^2_{\nabla^2f(\check{\bbeta}_{t})}.\nonumber
    \end{align}
   This concludes the proof of \Cref{lemm:pre_convergence}.
\end{proof}

\paragraph{Proof of \Cref{theom:RS-convergence}.}
First note that by~\Cref{assum:lip_f}, we have the following result
\begin{align}\label{eq:ass_for_conver}
    \|\H_t\check{\Delta}_{t}-\g_t\|_{\H_t^{-1}}\leq \upsilon \|\check{\Delta}_{t}\|_{\H_t},~~~~~\H_t\approx_{\upsilon}\H,~\text{for}~\upsilon=O \left(\frac{\log n}{\log (\log n)}\sqrt{\frac{\rho_{\max}^3 d^3_{\eff}} {m^3}} \right).
\end{align}
Letting  $\check{\Delta}_{t}=\check{\bbeta}_t-\bbeta^{\ast}$ and $\Delta_{t}=\bbeta_{t}-\bbeta^{\ast}$, we then show \eqref{eq:ass_for_conver} holds under \Cref{assum:lip_f}. 
Consider $\check\bbeta_t\in U$ for all $t$,  
we get the following condition:
\begin{align}
   \|\check{\Delta}_{t}\|_{\H} <\frac{\log n}{\log (\log n)} \left(\frac{\rho_{\max} d_{\eff}\sigma_{\min}}{m}\right)^{3/2}/L\nonumber,
\end{align}
where $\sigma_{\min}$ denotes the smallest eigenvalue of $\H$. 
We further derive
\begin{align}
\|\H^{-\frac{1}{2}}(\H_t-\H)\H^{-\frac{1}{2}}\|&\leq\frac{1}{\sigma_{\min}}\|\H_t-\H\|\leq\frac{L}{\sigma_{\min}}\|\check{\bbeta}_t-\bbeta^{\ast}\|= \frac{L}{\sigma_{\min}}\|\check{\Delta}_{t}\|\leq \frac{L}{\sigma^{3/2}_{\min}}\|\check{\Delta}_{t}\|_{\H}\nonumber\\
&\leq \frac{\log n}{\log (\log n)}\sqrt{\frac{\rho_{\max}^3d^3_{\eff}} {m^3}}=\upsilon,\nonumber
\end{align}
which leads to that $\H_t\approx_\upsilon\H$.  
We now advance to verify the second part of \eqref{eq:ass_for_conver}. 
Recalling that the standard analysis of Newton's method~\citep{boyd2004convex}, we have
\begin{align}
\|\H_t\check{\Delta}_{t}-\g_t\|
& = \left\|\H_t\check{\Delta}_{t}-\left(\int^1_0\nabla^2F(\bbeta^{\ast}+\tau\check{\Delta}_{t})d\tau\right)\check{\Delta}_{t}\right\|\leq\|\check{\Delta}_{t}\|\int^1_0\|\nabla^2F(\check{\bbeta}_t)-\nabla^2F(\bbeta^{\ast}+\tau\check{\Delta}_{t})\|d\tau\nonumber\\
&\leq\|\check{\Delta}_{t}\|\int^1_0(1-\tau)L\|\check{\Delta}_{t}\|d\tau=\frac{L}{2}\|\check{\Delta}_{t}\|^2,\nonumber
\end{align}
which together with $\upsilon< 1/2$  and the fact that  $\H_t\approx_\upsilon\H$ implies that $\|\H_t^{-1}\|\leq \frac{1+\upsilon}{\sigma_{\min}}$  achieves 
\begin{align}
\|\H_t\check{\Delta}_{t}-\g_t\|_{\H_t^{-1}} & \leq \sqrt{\frac{1+\upsilon}{\sigma_{\min}}}\|\H_t\check{\Delta}_{t}-\g_t\|\leq \sqrt{\frac{L^2(1+\upsilon)}{4\sigma_{\min}}}\|\check{\Delta}_{t}\|^2\leq\sqrt{\frac{L^2(1+\upsilon)}{4\sigma^3_{\min}}}\|\check{\Delta}_{t}\|^2_{\H}\leq\frac{\upsilon\sqrt{1+\upsilon}}{2} \|\check{\Delta}_{t}\|_{\H} \nonumber\\
&\leq \frac{\upsilon(1+\upsilon)^{3/2}}{2}\|\check{\Delta}_{t}\|_{\H_t} \leq \upsilon\|\check{\Delta}_{t}\|_{\H_t}. \nonumber
\end{align}
Thus, \eqref{eq:ass_for_conver} holds.

Next,  we proceed to the main part of the proof. We first rewrite $\|\Delta_{t+1}\|^2_{\H_t}$ as
\begin{align}    \|\Delta_{t+1}\|^2_{\H_t}&=\Delta_{t+1}^\top\H_t(\bbeta_{t+1}-\check{\bbeta}_{t}+\check{\bbeta}_{t}-\bbeta^{\ast})=\Delta_{t+1}^\top\H_t\check{\Delta}_{t}-\mu_t\Delta_{t+1}^\top\g_t \nonumber\\
    &=(1-\mu_t)\Delta_{t+1}^\top\g_t+\Delta_{t+1}^\top(\H_t\check{\Delta}_{t}-\g_t )
    \nonumber\\
    &=(1-\mu_t)\Delta_{t+1}^\top\H_t\check{\Delta}_{t}-(1-\mu_t)\Delta_{t+1}^\top(\H_t\check{\Delta}_{t}-\g_t )+\Delta_{t+1}^\top(\H_t\check{\Delta}_{t}-\g_t )\nonumber\\
    &=(1-\mu_t)(\check{\Delta}_{t}^\top\H_t\check{\Delta}_{t}-\mu_t\g_t^\top \check{\Delta}_{t})+\mu_t\Delta_{t+1}^\top(\H_t\check{\Delta}_{t}-\g_t )\nonumber\\
  &  =(1-\mu_t) ^2\|\check{\Delta}_{t}\|^2_{\H_t}+\mu_t(\Delta_{t+1}+(1-\mu_t)\check{\Delta}_{t})^\top(\H_t\check{\Delta}_{t}-\g_t).\nonumber
\end{align}
Invoking \eqref{eq:ass_for_conver} and triangle inequality, we get
\begin{align}
  \|\Delta_{t+1}\|^2_{\H_t} \leq(1-\mu_t) ^2\|\check{\Delta}_{t}\|^2_{\H_t}+  \upsilon \mu_t\|\Delta_{t+1}\| _{\H_t}\|\check{\Delta}_{t}\|_{\H_t}+\upsilon\mu_t(1-\mu_t)\|\check{\Delta}_{t}\|^2_{\H_t}. \nonumber
\end{align}
With the fact that if $x^2\leq a x+b$ then $x^2\leq a^2+2 b$, it follows  that:
\begin{align*}
\|\Delta_{t+1}\|^2_{\H_t}&  \leq \upsilon^2 \mu^2_t\|\check{\Delta}_{t}\|^2_{\H_t}+2(1-\mu_t) ^2\|\check{\Delta}_{t}\|^2_{\H_t}+2\upsilon\mu_t(1-\mu_t)\|\check{\Delta}_{t}\|^2_{\H_t}\\ 
&=(2(1-\mu_t) ^2+2\upsilon\mu_t+\mu^2_t(\upsilon^2-2\upsilon))\|\check{\Delta}_{t}\|^2_{\H_t}\nonumber\\
&\leq2((1-\mu_t) ^2+\upsilon\mu_t)\|\check{\Delta}_{t}\|^2_{\H_t}. \nonumber
\end{align*}
It then follows from \Cref{lemm:pre_convergence} and the inequality $\nabla^2f(\check{\bbeta}_t)\preceq \H_t$ that
\begin{align}\label{eq:norm_E[tilde]-real}
   \EE_{\zeta_t}[\|\check{\Delta}_{t+1}\|^2_{\H_t}] -\|\Delta_{t+1}\|^2_{\H_t}
     &\leq \upsilon\|\check{\Delta}_{t}\|^2_{\H_t} +\frac{36\rho_{\max} d_{\eff}}{m}\|\pp_t\|^2_{\H_t}. 
\end{align}
For the second term in \eqref{eq:norm_E[tilde]-real}, we rewrite
\begin{align}
\frac{36\rho_{\max} d_{\eff}}{m}\|\pp_t\|^2_{\H_t}&= \frac{36\rho_{\max} d_{\eff}}{m}\mu^2_t(\g_t^\top\check{\Delta}_{t}-\g_t^\top\H_t^{-1}(\H_t\check{\Delta}_{t}-\g_t))  \nonumber\\
  &=\frac{36\rho_{\max} d_{\eff}}{m}\mu^2_t(\check{\Delta}_{t}^\top\H_t\check{\Delta}_{t}-\check{\Delta}_{t}^\top(\H_t\check{\Delta}_{t}-\g_t)-\g_t^\top\H_t^{-1}(\H_t\check{\Delta}_{t}-\g_t))\nonumber\\
  &=\frac{36\rho_{\max} d_{\eff}}{m}\mu^2_t\|\check{\Delta}_{t}\|^2_{\H_t}-\frac{36\rho_{\max} d_{\eff}}{m}\mu^2_t(\check{\Delta}_{t}+\H_t^{-1}\g_t)^\top(\H_t\check{\Delta}_{t}-\g_t)).\nonumber
\end{align}
Using again \eqref{eq:ass_for_conver}, we further get 
\begin{align}
-\frac{36\rho_{\max} d_{\eff}}{m}\mu^2_t(\check{\Delta}_{t}+\H_t^{-1}\g_t)^\top(\H_t\check{\Delta}_{t}-\g_t)
&=\frac{36\rho_{\max} d_{\eff}}{m}(-\mu^2_t\check{\Delta}_{t}+\mu_t(\Delta_{t+1}-\check{\Delta}_{t}))^\top(\H_t\check{\Delta}_{t}-\g_t) \nonumber\\
&\leq \frac{36\rho_{\max} d_{\eff}\upsilon}{m}\mu_t (\mu_t+2)\|\check{\Delta}_{t}\|^2_{\H_t}. \nonumber
\end{align}
Then, putting the above together, we have 
\begin{align}
    \EE_{\zeta_t}[\|\check{\Delta}_{t+1}\|^2_{\H_t}] & \leq \left((2(1-\mu_t) ^2+2\upsilon\mu_t+\upsilon+\frac{36\rho_{\max} d_{\eff}}{m}\mu^2_t+\frac{36\rho_{\max} d_{\eff}\upsilon}{m}\mu_t (\mu_t+2)\right)  \|\check{\Delta}_{t}\|^2_{\H_t}\nonumber\\
    &=\left(\left(2(1-\mu_t) ^2+\frac{36\rho_{\max} d_{\eff}}{m}\mu^2_t\right) +\upsilon\left(2\mu_t+1+\frac{36\rho_{\max} d_{\eff}}{m}\mu_t (\mu_t+2)\right) \right)  \|\check{\Delta}_{t}\|^2_{\H_t},\nonumber
\end{align}
which along  with  $\upsilon=O\left(\sqrt{(\log n)^2\rho_{\max}^3 d^3_{\eff}/((\log(\log  n))^2m^3)}\right)$ and $\mu_t=\frac{1}{1+\rho_{\max} d_{\eff}/m}<1$ results in 
\begin{align}
     \EE_{\zeta_t}[\|\check{\Delta}_{t+1}\|^2_{\H_t}] &  \leq\left(\left(2(1-\mu_t) ^2+\frac{36\rho_{\max} d_{\eff}}{m}\mu^2_t\right) +\upsilon\right)  \|\check{\Delta}_{t}\|^2_{\H_t}
\leq \left(\frac{37\rho_{\max} d_{\eff}}{m}+\upsilon\right)\|\check{\Delta}_{t}\|^2_{\H_t}.
\nonumber
\end{align}
Applying $\upsilon<\frac{1}{2}$ and  the fact that $ \H_t\approx_{\upsilon}\H$ indicates  $ \|\vv\|^2_{\H_t}\approx_{\upsilon}\|\vv\|^2_{\H}
$, we get the following bound: 
\begin{align}\label{eq:each_iterate_error}
     \EE_{\zeta_t}[\|\check{\Delta}_{t+1}\|^2_{\H}]
     & \leq    (1+\upsilon) \EE_{\zeta_t}[\|\check{\Delta}_{t+1}\|^2_{\H_t}]\leq (1+\upsilon) \left( \frac{37\rho_{\max} d_{\eff}}{m}+\upsilon\right) \|\check{\Delta}_{t}\|^2_{\H_t}\nonumber\\
     &\leq (1+\upsilon)^2 \left( \frac{37\rho_{\max} d_{\eff}}{m}+\upsilon\right) \|\check{\Delta}_{t}\|^2_{\H}=\left( \frac{\rho_{\max} d_{\eff}}{m}+\upsilon\right)\|\check{\Delta}_{t}\|^2_{\H},
\end{align}
where the constant ``$(1+\upsilon)^2\cdot37$'' is  absorbed into $m$.

Now, it remains to  check that $\check\bbeta_t\in U$ for all $t$ when conditioned on the event $\zeta=\bigcap^{T-1}_{t=0}\zeta_t$.  Assuming that this holds for $t=0$, then it remains to prove that, conditioned on  the event $\zeta$, $\|\check{\Delta}_{t+1}\|_{\H}\leq\|\check{\Delta}_{t}\|_{\H} $ holds for each $t$ almost surely.  Following the fact that  \Cref{prop:debias} implies that, conditioned on $\zeta_t$,  we obtain 
\begin{align}
  \|\hat \Q_t-\I_d\|\leq \varepsilon, \nonumber
\end{align}
where Lemma \ref{lem:sub_embed} guarantees that $\varepsilon$ is small. Adapting the techniques  used in the proof of Lemma \ref{lemm:pre_convergence} and by the analysis of the exact Newton step $\Delta_{t+1}$, we arrive at
\begin{align}
    \|\check{\Delta}_{t+1}\|^2_{\H_t}&\leq \|\Delta_{t+1}\|^2_{\H_t} +2\|\Delta_{t+1}\|_{\H_t}\|\pp_t\|_{\H_t}\|\hat \Q_t-\I_d\|+\pp^\top_t\H^{\frac{1}{2}}_t(\hat \Q_t-\I_d)^2\H^{\frac{1}{2}}_t\pp_t \nonumber\\
   & \leq \|\Delta_{t+1}\|^2_{\H_t} +8\|\check{\Delta}_{t}\|^2_{\H_t}\|\hat \Q_t-\I_d\| 
   \leq \left(\varepsilon+\frac{\rho_{\max} d_{\eff}}{m}+\upsilon\right)\|\check{\Delta}_{t}\|^2_{\H_t}. \nonumber
\end{align}
From Lemma \ref{lem:sub_embed}, it is deduced that  $\varepsilon+\frac{\rho_{\max} d_{\eff}}{m}+\upsilon$ is sufficiently small. Using $\H_t\approx_\upsilon\H$ again, we deduce  $\|\check{\Delta}_{t+1}\|^2_{\H}\leq\|\check{\Delta}_{t}\|^2_{\H}$. Consequently, we infer that every iterate lies within $U$, and the result \eqref{eq:each_iterate_error} holds for $t=0,1,\ldots,T-1$.
Putting the above together, we conclude that, conditioned on the event  $\zeta$ that  holds with probability at least $1-\delta$, the union bound  \eqref{eq:result_conver} is achieved, thereby completing the proof.
\qedwhite

\subsection{Additional Results on SSN and Further Discussions} 
\label{subsec:additional_SSN}

Based on the above analysis, for any sampling method whose sampling probabilities are close to the exact approximate leverage scores, we then show that using scalar debiasing $m/(m-d_{\eff})$ yields a slightly weaker local convergence result for SSN  than that in~\Cref{theom:RS-convergence}, yet achieves enhanced computational efficiency. The result naturally follows  by
recalling~\Cref{coro:inver_bias_constant_debias} and~\Cref{coro:sec.moment_lev},  and adapting the proof of ~\Cref{theom:RS-convergence} with a neighborhood  $U=\{\bbeta \colon \|\bbeta-\bbeta^{\ast}\|_{\H}< u_{\rm lev}\sigma_{\min}^{3/2}/L\}$ in place of  $U=\{\bbeta \colon \|\bbeta-\bbeta^{\ast}\|_{\H}< (\rho_{\max} d_{\eff} \sigma_{\min}/m)^{3/2}/L\}$,  and  the corresponding  proof  is omitted for brevity.

\begin{corollary}[Local convergence of SSN using scalar debiasing under approximate leverage  ]\label{coro:ssn_exact_approximate_lev}
Under the settings and notations of~\Cref{theom:RS-convergence},
for any random sampling scheme with sampling distribution
$\pi_i\in [\ell^{\C}_{i}(\check\bbeta_t)/(d_{\eff}(\check\bbeta_t) \rho_{\max}(\check\bbeta_t)), \ell^{\C}_{i}(\check\bbeta_t)/(d_{\eff}(\check\bbeta_t) \rho_{\min}(\check\bbeta_t))]$ as in~\Cref{coro:inver_bias_constant_debias}, there exists a neighborhood $U$ of $\bbeta^{\ast}$ such that the \emph{de-biased} SSN iteration  $ \check{\bbeta}_{t+1}= \check{\bbeta}_{t}-\mu_t \left(\frac{m}{m-d_{\eff}(\check{\bbeta}_t)}\A(\check\bbeta_t)^\top \S_t^\top  \S_t \A(\check\bbeta_t)+\C(\check\bbeta_t) \right)^{-1}\g_t$ starting from $\check{\bbeta}_0 \in U$ satisfies, for $U=\{\bbeta \colon \|\bbeta-\bbeta^{\ast}\|_{\H}< u_{\rm lev}\sigma_{\min}^{3/2}/L\}$, step size $\mu_t=1-\frac{\rho_{\max} }{m/d_{\eff} + \rho_{\max} }$,  and $m\geq C\rho_{\max} d_{\eff}\log(d_{\eff}T/\delta ) $, that
\begin{equation*}
   \left( \EE_{\zeta}\left[\frac{\|\check{\bbeta}_T-\bbeta^{\ast}\|_{\H}}{\|\check{\bbeta}_0-\bbeta^{\ast}\|_{\H}}\right]\right)^{1/T}\leq \frac{\rho_{\max} d_{\eff}}{m} ( 1+ \epsilon+\epsilon_{\rho}), 
\end{equation*}
holds for $\epsilon=\tilde O\left(\sqrt{\rho_{\max} d_{\eff}/m}\right), \epsilon_{\rho}=\max\{\rho_{\min}^{-1}-1,1 - \rho_{\max}^{-1}\}$ and conditioned on an event $\zeta$ that happens with probability at least $1-\delta$. Here, $u_{\rm lev}=\log n(\rho_{\max} d_{\eff} /m)^{3/2}/\log (\log n)+\epsilon_{\rho}\rho_{\max} d_{\eff}/m$, $\sigma_{\min}$ is the smallest singular value of $\H \equiv\A(\bbeta^{\ast}) ^\top\A(\bbeta^{\ast})+\C(\bbeta^{\ast})$, $d_{\eff}=\max_{1\leq t \leq T}d_{\eff}(\check{\bbeta}_t)$ with $\ell_{i}^\C(\check{\bbeta}_t)$ and  $d_{\eff}(\check{\bbeta}_t)$ the leverage scores and effective dimension of $\A(\check{\bbeta}_t)$ given $\C(\check{\bbeta}_t)$, respectively, $\rho_{\max}=\max_{1\leq t \leq T}\rho_{\max}(\check{\bbeta}_t)$ and $\rho_{\min}=\min_{1\leq t \leq T}\rho_{\min}(\check{\bbeta}_t)$ with $\rho_{\max}(\check{\bbeta}_t)$  and $\rho_{\min}(\check{\bbeta}_t)$ the (max and min)  importance sampling approximation   factors (of $\A(\check{\bbeta}_t)$ given $\C(\check{\bbeta}_t)$) in \Cref{def:approx_factor}.
\end{corollary}

Similarly, for SRHT, applying the scalar debiasing  $m/(m-d_{\eff})$ to SSN  also  leads to a   slightly weaker local convergence than that in~\Cref{theom:RS-convergence}  but offers greater computational efficiency. This follows  readily  by
recalling~\Cref{coro:debiasing_srht} and~\Cref{coro:sec.moment_SRHT},  and adapting the proof of ~\Cref{theom:RS-convergence} with a neighborhood  $U=\{\bbeta \colon \|\bbeta-\bbeta^{\ast}\|_{\H}< u_{\rm SRHT}\sigma_{\min}^{3/2}/L\}$ instead of  $U=\{\bbeta \colon \|\bbeta-\bbeta^{\ast}\|_{\H}< (\rho_{\max} d_{\eff} \sigma_{\min}/m)^{3/2}/L\}$, and we omit the detailed proof for brevity.

\begin{corollary}[Local convergence of SSN for SRHT using scalar debiasing]\label{coro:ssn_SRHT}
Under the settings and notations of~\Cref{theom:RS-convergence},
for $\tilde{\A}_{\rm SRHT}(\check{\bbeta}_t)= \S_t\H_n\D_n\A(\check{\bbeta}_t)/\sqrt{n}  \in \RR^{m \times n}$ the SRHT of $\A$ as in \Cref{def:srht}, there exists a neighborhood $U$ of $\bbeta^{\ast}$ such that the \emph{de-biased} SSN iteration $ \check{\bbeta}_{t+1}= \check{\bbeta}_{t}-\mu_t \left(\frac{m}{m-d_{\eff}(\check{\bbeta}_t)}\tilde{\A}_{\rm SRHT}^\top(\check\bbeta_t) \S_t^\top  \S_t \tilde{\A}_{\rm SRHT}(\check\bbeta_t)+\C(\check\bbeta_t) \right)^{-1}\g_t$ starting from $\check{\bbeta}_0 \in U$ satisfies, for $U=\{\bbeta \colon \|\bbeta-\bbeta^{\ast}\|_{\H}<u_{\rm SRHT}\sigma_{\min}^{3/2}/L\}$, step size $\mu_t=1-\frac{\rho_{\max} }{m/d_{\eff} + \rho_{\max} }$, $m\geq C\rho_{\max} d_{\eff}\log(d_{\eff}T/\delta ) $, and $Tn \exp(-d_{\eff}) < \delta < Tm^{-3}$, that
\begin{equation}\label{eq:SSN_SRHT_scalar_debiasing}
   \left( \EE_{\zeta}\left[\frac{\|\check{\bbeta}_T-\bbeta^{\ast}\|_{\H}}{\|\check{\bbeta}_0-\bbeta^{\ast}\|_{\H}}\right]\right)^{1/T}\leq \frac{\rho_{\max} d_{\eff}}{m} ( 1+\epsilon), 
\end{equation}
holds for $\epsilon=\tilde  O\left(\sqrt{\rho_{\max} d_{\eff}/m} \right) + O\left(\rho^{-1}_{\max} \sqrt{d_{\eff}^{-1}\log(nT/\delta)}\right)$  and conditioned on an event $\zeta$ that happens with probability at least $1-\delta$. Here, $u_{\rm SRHT}=\log n(\rho_{\max} d_{\eff} /m)^{3/2}/\log (\log n)+\sqrt{d_{\eff} \log(nT/\delta)}/m $, $\sigma_{\min}$ is the smallest singular value of $\H \equiv\A(\bbeta^{\ast})^\top \A(\bbeta^{\ast})+\C(\bbeta^{\ast})$, $d_{\eff}=\max_{1\leq t \leq T}d_{\eff}(\check{\bbeta}_t)$ with $\ell_{i}^\C(\check{\bbeta}_t)$ and  $d_{\eff}(\check{\bbeta}_t)$ the leverage scores and effective dimension of $\H_n\D_n\A(\check{\bbeta}_t)/\sqrt{n}$ given $\C(\check{\bbeta}_t)$, respectively, $\rho_{\max}=\max_{1\leq t \leq T}\rho_{\max}(\check{\bbeta}_t)$  with $\rho_{\max}(\check{\bbeta}_t)$ the max  importance sampling  approximation factor (of $\H_n\D_n\A(\check{\bbeta}_t)/\sqrt{n}$ given $\C(\check{\bbeta}_t)$) in \Cref{def:approx_factor}. 
\end{corollary}

\begin{remark}[Comparison between \Cref{coro:ssn_SRHT} and{~\citet[Theorem 2]{lacotte2021adaptive}}]\normalfont \label{rem:ssn_srht_scalar}
Using~\Cref{coro:debiasing_srht} and \Cref{coro:sec.moment_SRHT} with $\nu = 0$, it follows, by adapting the proof of \Cref{theom:RS-convergence}, that for a number $m>Cd_{\eff}$ of samples, the linear convergence rate in \Cref{coro:ssn_SRHT} holds \emph{but} with a right-hand side term of $O(1)$ in \eqref{eq:SSN_SRHT_scalar_debiasing}. 
That is similar to the \emph{problem-dependent} linear convergence rate obtained in e.g.,~\citet[Theorem 2]{lacotte2021adaptive} for self-concordant $f$, $\Phi$, and $F$ in~\eqref{eq:optimization_problem}. 
Notably, using a larger number of samples $m$ as in \Cref{coro:ssn_SRHT} (than $\Theta(d_{\eff})$), the (linear) convergence rate in \Cref{coro:ssn_SRHT} becomes dependent on $\rho_{\max} d_{\eff}/m$, whereas the linear or quadratic rates stated in \citet[Theorem 2]{lacotte2021adaptive} cannot be characterized using $n$, $d_{\eff}$ or $m$.
\end{remark}

\section{Additional Numerical Experiments and Implementation Details}
\label{sec:imple_detail_nmerical_exper}

Here, we provide in \Cref{sec:sketch_matrix} and \Cref{sec:Datasets} implementation details of the numerical experiments on SSN in \Cref{sec:num}, and then in \Cref{subsec:numerical inversion bias} additional numerical results on the inversion bias under approximate versus exact leverage score sampling. 

\subsection{Sketching Matrices and Step Size}
\label{sec:sketch_matrix}

For a data matrix $\A\in\RR^{n\times d}$, the approximate ridge leverage scores and  leverage scores are computed using the method provided in~\citet{drineas2012fast}~and~\citet{cohen2017input}. 
The LESS-uniform sketching matrix  is constructed as in \citet[Section~E.1]{derezinski2021newtonless}.  Further implementation details, including those for the SRHT, are available in the public repository provided by \citet{derezinski2021newtonless} at \url{https://github.com/lessketching/newtonsketch}. 
Additionally, the Shrinkage Leverage Score sampling probability is formulated by combining uniform sampling probability and approximate leverage score sampling probability in equal proportions. 

In our experiments, the first-order methods, Gradient Descent and Stochastic Gradient Descent, are used with a fixed step size. 
As done in~\citet{derezinski2021newtonless}, second-order methods, specifically Sub-sampled Newton and Newton Sketch with Less-uniform Sketches, employ step sizes that are dynamically adjusted using a line search algorithm based on the Armijo condition~\citep[Chapter~3]{bonnans2006numerical}.

For a fair comparison of the ``convergence-complexity'' trade-off across different optimization methods, the time reported in Figures~\ref{fig:sketch size}~and~\ref{fig:convergence_time} include the time for input data pre-processing, e.g.,  the computation of exact or approximate leverage scores, and Walsh–Hadamard transform, as well as the computational overhead associated with the sketching process.
And we fix in Figures~\ref{fig:sketch size}~and~\ref{fig:convergence_time} the ridge regularization parameter to $\lambda=10^{-2}$ for both MNIST and CIFAR-10 data in \Cref{sec:num}.

\subsection{Datasets}\label{sec:Datasets}

For MNIST data matrix, we have $n = 2^{13}$ and $d =2^7$, and for CIFAR-10 data, we have $n =2^{14}$ and $d =2^8$.  
We use \textsf{torchvision.transforms} from PyTorch to pre-process each image. 
We divide the ten classes of MNIST and CIFAR-10 datasets into two groups, assigning them labels of $-1$ and $+1$, respectively.

\subsection{Inversion Bias for Exact versus Approximate Leverage Score Sampling under Scalar Debiasing}
\label{subsec:numerical inversion bias}

In this section, we present empirical experiments to compare the inversion bias of exact versus approximate leverage score sampling under scalar debiasing $m/(m-d_{\eff})$, as shown in \Cref{coro:inver_bias_constant_debias} and discussed in \Cref{rem:inv_bias_exact_VS_approx}, for both MNIST and CIFAR-10 data.

\Cref{fig:sketch size of inversion bias} depicts the inversion bias (with and without the scalar debiasing $m/(m-d_{\eff})$ measured in spectral norm:
$\|\H^{1/2}(\EE[(\frac{m}{m-d_{\eff}}\A^\top \S^\top\S \A + \lambda\I_d )^{-1}]-(\A^\top \A + \lambda\I_d )^{-1})\H^{1/2}\|$ and $\|\H^{1/2}(\EE[(\A^\top \S^\top\S \A + \lambda\I_d )^{-1}]-(\A^\top \A + \lambda\I_d )^{-1})\H^{1/2}\|$,  $\H=\A^\top \A + \lambda\I_d$ ) as a function of the sketch size $m$, using the following random sampling schemes:
\begin{enumerate}
    \item \textbf{RLev}: Exact $\lambda$-ridge leverage score sampling with scalar debiasing $m/(m-d_{\eff})$.
      \item \textbf{NO-RLev}: Standard exact $\lambda$-ridge leverage score sampling \emph{without} scalar debiasing $m/(m-d_{\eff})$.
     \item \textbf{ARLev}: Approximate $\lambda$-ridge leverage score sampling (identical to that used in \Cref{sec:num}) with scalar debiasing $m/(m-d_{\eff})$, where the approximate ridge leverage scores $\hat \ell_i^{\C}\approx \|\ee_i^\top\A(\A^\top\S_1^\top\S_1 \A + \lambda\I_d )^{-1/2}\|^2$  are computed using sparse Johnson Lindenstrauss transform (SJLT)~\citep{clarkson2017low} $\S_1\in\RR^{m_1\times n }$ of size $m_1$.
      \item \textbf{NO-ARLev}: Standard approximate $\lambda$-ridge leverage score sampling without scalar debiasing $m/(m-d_{\eff})$.
      \item \textbf{DARLev}: A more efficient double-sketches variant of approximate $\lambda$-ridge leverage score sampling, together with  scalar debiasing $m/(m-d_{\eff})$, where the approximate ridge leverage scores $\hat \ell_i^{\C}\approx \|\ee_i^\top\A((\A^\top\S_1^\top\S_1  \A + \lambda\I_d )^{-1/2}\S^\top_2)\|^2$ are constructed using two SJLT matrices $\S_1\in\RR^{m_1 \times n }$ and $\S_2\in\RR^{m_2 \times d }$, with $m_2<m_1$, see~\citep{drineas2012fast,cohen2017input}.
      \item \textbf{NO-DARLev}: As \textbf{DARLev} but without the scalar debiasing $m/(m-d_{\eff})$.
\end{enumerate}

We observe from \Cref{fig:sketch size of inversion bias} that the inversion bias of all random sampling methods consistently decreases as $m$ increases. 
Comparing solid to dashed lines, we see that the proposed scalar debiasing effectively reduces the inversion bias by a significant margin.
We also find that \textbf{RLev} yields a lower inversion bias compared to its approximate counterparts \textbf{ARLev} and \textbf{DARLev}, under the scalar debiasing.
These results corroborate the conclusions in \Cref{coro:inver_bias_constant_debias} and \Cref{rem:inv_bias_exact_VS_approx}.

\begin{figure}[!t]
  \centering
   \begin{subfigure}[c]{0.48\textwidth}
\begin{tikzpicture}
\renewcommand{\axisdefaulttryminticks}{4} 
\pgfplotsset{every major grid/.style={densely dashed}}       
\tikzstyle{every axis y label}+=[yshift=-10pt] 
\tikzstyle{every axis x label}+=[yshift=5pt]
\pgfplotsset{every axis legend/.append style={cells={anchor=east},fill=none,   at={(0.99,0.99)}, anchor=north east, font=\tiny}}
\begin{axis}[
width=0.95\columnwidth,
height=0.75\columnwidth,
xlabel style={font=\small},
    ylabel style={font=\small},
    tick label style={font=\tiny},
    xmin =500,
        xmax =3000,
        ymax =0.4,
        ymin =0.01, 
        ymajorgrids=true,
        scaled ticks=true,
          xlabel = { Sketch size $m$ },
          ylabel = { Inversion bias },
        ymode=log
        ]

        \addplot[mark=pentagon*,color={cmyk:cyan,1;magenta,0;yellow,0;black,0},line width=0.8pt] coordinates{
          (500,0.0455) ( 1400, 0.0202) (2000, 0.0172)  (2500, 0.0147) (3000,   0.0117)
        };
          \addlegendentry{{RLev}}[ font=\tiny];

            \addplot[mark=pentagon*,color={cmyk:cyan,1;magenta,0;yellow,0;black,0},line width=0.8pt, dashed, mark options={solid,fill={cmyk:cyan,1;magenta,0;yellow,0;black,0}}] coordinates{
          (500,0.1858) ( 1400,  0.0631) (2000, 0.0482)  (2500, 0.0382) (3000,   0.0325)

        };
          \addlegendentry{{NO-RLev}}[ font=\tiny];
      \addplot[mark=*,color=RED,line width=0.8pt] coordinates{
   (500,0.0527) ( 1400,0.0223)(2000, 0.0202)  (2500, 0.0175) (3000, 0.0137)
      };
\addlegendentry{{ARLev}}[ font=\tiny];

 \addplot[mark=*,color=RED,line width=0.8pt, dashed, mark options={solid,fill=RED}] coordinates{
(500,0.1952) ( 1400, 0.0669) (2000, 0.0503)  (2500, 0.0421) (3000,  0.0347)
      };
\addlegendentry{{NO-ARLev}}[ font=\tiny];

      \addplot[mark=|,color={cmyk:cyan,0.4;magenta,0.2;yellow,1;white,0;green,1},line width=0.8pt] coordinates{
   (500,0.1821) ( 1400,   0.0316) (2000, 0.0252)  (2500,  0.0221) (3000,  0.0205)
      };
\addlegendentry{{DARLev}}[ font=\tiny];

      \addplot[mark=|,color={cmyk:cyan,0.4;magenta,0.2;yellow,1;white,0;green,1},line width=0.8pt, dashed, mark options={solid,fill={cmyk:cyan,0.4;magenta,0.2;yellow,1;white,0;green,1}}] coordinates{
 (500,0.3092) ( 1400,  0.0787) (2000, 0.0579)  (2500, 0.0472) (3000,  0.0419)
      };
\addlegendentry{{NO-DARLev}}[ font=\tiny];

\end{axis}
\end{tikzpicture}
 \caption{MNIST data}
\end{subfigure}
  \hfill{}
  \begin{subfigure}[c]{0.48\textwidth}
  \begin{tikzpicture}
\renewcommand{\axisdefaulttryminticks}{4} 
\pgfplotsset{every major grid/.style={densely dashed}}       
\tikzstyle{every axis y label}+=[yshift=-10pt] 
\tikzstyle{every axis x label}+=[yshift=5pt]
\pgfplotsset{every axis legend/.append style={cells={anchor=east},fill=white, at={(0.8,1.4)}, anchor=north west, font=\tiny}}
\begin{axis}[
width=0.95\columnwidth,
height=0.75\columnwidth,
xlabel style={font=\small},
    ylabel style={font=\small},
    tick label style={font=\tiny},
    xmin =500,
        xmax = 4000,
        ymax =1.8,
        ymin = 0.02,
          ymajorgrids=true,
        scaled ticks=true,
          xlabel = { Sketch size $m$ },
          ylabel = {Inversion bias },
        ymode=log
        ]

        \addplot[mark=pentagon*,color={cmyk:cyan,1;magenta,0;yellow,0;black,0},line width=0.8pt] coordinates{
    (500, 0.1091) (1000, 0.0547) (1500, 0.0413) (3500, 0.0249) (4000, 0.0230)
        };

            \addplot[mark=pentagon*,color={cmyk:cyan,1;magenta,0;yellow,0;black,0},line width=0.8pt, dashed] coordinates{
      (500, 1.2728) (1000, 0.4176) (1500, 0.2556) (3500, 0.1056) (4000, 0.0929)
        };
      \addplot[mark=*,color=RED,line width=0.8pt] coordinates{
 (500, 0.1114) (1000, 0.0571) (1500, 0.0429) (3500, 0.0258) (4000, 0.0236)
      };

 \addplot[mark=*,color=RED,line width=0.8pt, dashed] coordinates{
(500, 1.2775) (1000, 0.4208) (1500, 0.2575) (3500, 0.1067) (4000, 0.0936)
      };

      \addplot[mark=|,color={cmyk:cyan,0.4;magenta,0.2;yellow,1;white,0;green,1},line width=0.8pt] coordinates{
   (500, 0.1854) (1000, 0.0915) (1500, 0.0664) (3500, 0.0362) (4000, 0.0346)
      };

      \addplot[mark=|,color={cmyk:cyan,0.4;magenta,0.2;yellow,1;white,0;green,1},line width=0.8pt, dashed] coordinates{
(500, 1.4290) (1000, 0.4671) (1500, 0.2858) (3500, 0.1180) (4000, 0.1054)
      };
\end{axis}
\end{tikzpicture}
 \caption{CIFAR-10 data}
  \end{subfigure}
\caption{{ Inversion bias as a function of the sketch size $m$, for various sampling methods on both MNIST and CIFAR-10 data, with ridge parameter $\lambda=10^{-1}$ for MNIST data and $\lambda=10^{-6}$ for CIFAR-10 data. The results are averaged over $500$ independent runs.
}}
\label{fig:sketch size of inversion bias}
\end{figure}
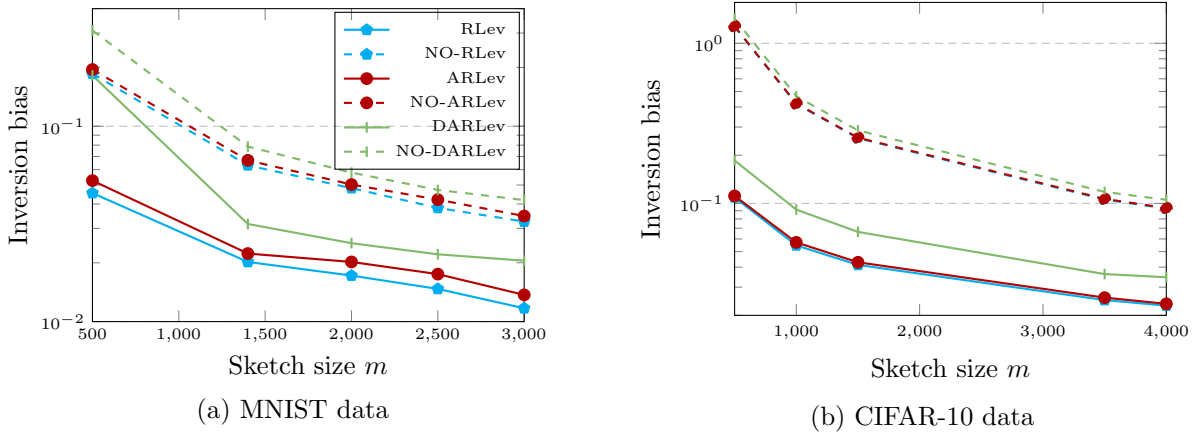

\end{document}